%% file: XXO.tex
\documentclass[12pt]{amsart}
\pdfoutput=1
 \usepackage{amsmath,amsfonts,mathrsfs,amsthm,amssymb}
 \usepackage{amssymb}
 \usepackage{graphicx}
 \usepackage{lscape}
 \usepackage[all,cmtip]{xy}
 \usepackage{cancel}
 \usepackage{bm}
 \usepackage{multirow}
 \usepackage{framed}
 \usepackage{cprotect}
 \usepackage{hyperref}
 \usepackage{xcolor}
  \usepackage{tikz}
  \usepackage{tikz-cd}
  \usepackage{float}
  \usetikzlibrary{decorations.pathmorphing}
\tikzcdset{arrow style=math font}
\usetikzlibrary{arrows.meta}

\tikzset{>={Stealth[scale=1.5]}}

\def\graycolor{gray!60}
 \newcommand\red[1]{{\color{red} #1}}

 \usepackage{ifthen}
 \usepackage{xspace}
 \input{xstring}
 \input{Dynkin}
 \newcommand\hol{\mathfrak{hol}}
 \newcommand\proj{\operatorname{proj}}
 \newenvironment{psm}
  {\left(\begin{smallmatrix}}
  {\end{smallmatrix}\right)}
 \newcommand\we{\widetilde{e}}
 \newcommand\wfg{\widetilde\fg}
 \newcommand\wfp{\widetilde\fp}
 \newcommand\wfq{\widetilde\fq}

\usepackage{DT-macros}
\newtheorem{theorem}{Theorem}[section]
\newtheorem{lemma}[theorem]{Lemma}

\newtheorem{prop}{Proposition}[section]
\theoremstyle{defn}
\newtheorem{defn}[theorem]{Definition}
\theoremstyle{problem}

\newtheorem{example}[theorem]{Example}

\theoremstyle{remark}
\newtheorem{remark}[theorem]{Remark}
\numberwithin{equation}{section}

\begin{document}

 \title[Conformal structures with $G_2$-symmetric twistor distribution]{Conformal structures with\\ $G_2$-symmetric twistor distribution}

 \author{Pawel Nurowski}
 \address{Center for Theoretical Physics, Polish Academy of Sciences, Al. Lotników 32/46, 02-668, Warszawa, Poland\\
 and Guangdong Technion - Israel Institute of Technology, 241 Daxue Road, Jinping District, Shantou, Guangdong Province, China}
 \email{nurowski@cft.edu.pl}

 \author{Katja Sagerschnig}
 \address{Center for Theoretical Physics, Polish Academy of Sciences, Al. Lotników 32/46, 02-668, Warszawa, Poland\\
 and Guangdong Technion - Israel Institute of Technology, 241 Daxue Road, Jinping District, Shantou, Guangdong Province, China}
 \email{katja@cft.edu.pl}

 \author{Dennis The}
 \address{Department of Mathematics \& Statistics, UiT The Arctic University of Norway, N-9037, Troms\o, Norway}
 \email{dennis.the@uit.no}

 \begin{abstract}
 For any 4D split-signature conformal structure, there is an induced twistor distribution on the 5D space of all self-dual totally null 2-planes, which is $(2,3,5)$ when the conformal structure is not anti-self-dual.  Several examples where the twistor distribution achieves maximal symmetry (the split-real form of the exceptional simple Lie algebra of type $\mathrm{G}_2$) were previously known, and these include fascinating examples arising from the rolling of surfaces without twisting or slipping.  Relaxing the rolling assumption, we establish a complete local classification result among those homogeneous 4D split-conformal structures for which the symmetry algebra induces a multiply-transitive action on the 5D space.  Furthermore, we discuss geometric properties of these conformal structures such as their curvature, holonomy, and existence of Einstein representatives.
 \end{abstract}

 \date{\today}
 
 


 \maketitle

\tableofcontents

\section{Introduction}

One of the unexpected appearances of the split-real form $\fg$ of the exceptional simple Lie algebra of type $\mathrm{G}_2$ is the following. Consider a ball with radius $r$ rolling on another ball with radius $R$ without slipping or twisting. Then the configuration space is a $5$-manifold $N$, and a curve represents `rolling without slipping or twisting' if and only if it is everywhere tangent to a certain rank two distribution $\cD = \cD_{r:R}\subset TN$ -- the {\sl rolling distribution} of the two balls. Provided that $r\neq R$, $\cD$ is a $(2,3,5)$-distribution, i.e. $[\cD,\cD]$ has constant rank $3$  and $[\cD,[\cD,\cD]]=TN$. The distribution $\cD$ always has a natural $\mathfrak{so}(3)\times\mathfrak{so}(3)$-symmetry, but it turns out that the Lie algebra of (infinitesimal) symmetries of $\cD$ is isomorphic to $\fg=\mathsf{Lie}(\mathrm{G}_2)$ if and only if the ratio of the radii of the two balls is $1:3$ or $3:1$, see \cite{Agr2007, AN2014, BH2014, BM2009,The2022,Zel2006}.

 In \cite{AN2014}, the authors asked whether there are
  other pairs of surfaces whose rolling distribution has symmetry $\fg=\mathsf{Lie}(\mathrm{G}_2)$. In the course of their investigations, they identified the rolling distribution as a special case of the following more general construction. Consider a $4$-manifold $M^4$ equipped with an {\sl oriented split-signature conformal structure} $[\sfg]$.  It has two families of totally null $2$-planes: self-dual (SD) and anti-self-dual (ASD) ones. In particular, one can form the circle bundle 
 \begin{equation} \label{E:twistorbundle}
 \mathbf{S}^1\hookrightarrow \mathbb{T}^{+} \overset{\pi}{\longrightarrow }M^4
 \end{equation} 
 of all SD totally null $2$-planes.  The total space $N^5=\bbT^+$ is 5-dimensional and carries a natural rank two distribution $\cD$, called the {\sl twistor distribution}, see Section \ref{sec-construction} for the definition.\footnote{There is an analogous twistor distribution on the ASD side, but the majority of our article uses the SD construction, so we have chosen to not emphasize this via augmented notation, e.g.\ $(\ell^+,\cD^+)$.  Generally, SD will be implicit, but we will emphasize when we consider the ASD construction instead.} It is $(2,3,5)$  on  open subsets in $\mathbb{T}^+$ provided the SD part $\cW^+$ of the Weyl tensor of $[\sfg]$ is \emph{non-vanishing}.   Moreover, given  two Riemannian surfaces $(\Sigma_1,\sfg_1)$ and $(\Sigma_2,\sfg_2)$,  the twistor distribution of the conformal manifold $(\Sigma_1 \times \Sigma_2, [\sfg_1 - \sfg_2])$ can be naturally identified with the rolling distribution of the two surfaces.

 This leads to the more general question \cite{AN2014}: \emph{For which $4D$ split-conformal structures does the twistor distribution have (infinitesimal) $\mathrm{G}_2$-symmetry?}  Several of such  structures are already known.  One example is represented by Brinkmann's Ricci-flat pp-wave metric $\sfg=\mathsf{d}x\mathsf{d}u+\mathsf{d}y\mathsf{d}v+x^2\mathsf{d}v^2$ originating in \cite{Brinkmann1925}, for which $[\sfg]$ is submaximally symmetric and has 9-dimensional symmetry \cite{CDT2013, KT2014}. (In 4D, the maximum conformal symmetry dimension is 15.)  Two more examples come from the $\fsl(3,\bbR)$-invariant {\sl dancing metric} and its $\fsu(1,2)$-invariant analogue \cite{CDT2013, BLN2018, KM2023}.

 The aforementioned  rolling distribution of two balls with ratio of radii $1:3$ comes from a conformal structure with 6-dimensional symmetry. 
 Beyond this  rolling example, three more conformal structures of this sort where found in \cite{AN2014}. They correspond to surfaces of revolution  $(\Sigma_1,\sfg_1)$ rolling on the flat plane $(\Sigma_2 = \bbR^2, \sfg_2 = \mathsf{d}x^2 + \mathsf{d}y^2)$. Their conformal structures $[\sfg]$ are represented by $\sfg = \sfg_1 - \sfg_2$, with\
 \begin{align} \label{E:ANex}
 \sfg_1 = (\rho^2+\epsilon)^2 \mathsf{d}\rho^2 + \rho^2 \mathsf{d}\varphi^2, \qquad \epsilon \in \{ 0, \pm 1 \},
 \end{align}
 and their twistor/rolling distributions have $\mathrm{G}_2$-symmetry.  Concerning  conformal symmetries of $[\sfg]$, it admits the symmetries $\partial_\varphi, \partial_x, \partial_y, -y\partial_x + x\partial_y$.  When $\epsilon = \pm 1$, this is the entire conformal symmetry algebra of $[\sfg]$, and the conformal structure is {\em non-homogeneous}.  When $\epsilon = 0$, these symmetries are augmented by the scaling symmetry $\rho\partial_\rho + 2\varphi \partial_\phi + 3x\partial_x + 3y\partial_y$, making the $\epsilon = 0$ structure  {\em homogeneous}.  
 
 In our present article, we begin a systematic classification programme of $\mathrm{G}_2$-symmetric twistor distributions. We concentrate  on the conformally  homogeneous setting.
 
 In order to formulate our main result, let us make the following key observations: First, the vertical bundle $\ell = \ker(\pi_*)$ of \eqref{E:twistorbundle} is  distinguished. This leads us to consider on $\bbT^+$  the enhanced structure $(\ell,\cD)$, which we call the {\sl twistor XXO-structure of $[\sfg]$}. (The ``XXO" name comes from the marked Dynkin diagram encoding of the geometry.)   Diagrammatically, we have the following: 
 
 \begin{center}
 \begin{tikzcd} 
 \begin{tabular}{c}
 twistor XXO-structure\\ 
 $(N^5 = \bbT^+; \ell, \cD)$ 
 \end{tabular}
 \arrow[rrr,rightsquigarrow,"\mbox{$[\cD,\cD] \neq \cD$ on $U \subset N^5$}"]\arrow[d,"\pi"]&&& \begin{tabular}{c} $(2,3,5)$-distribution\\ $(U^5;\cD)$\end{tabular} \\ 
 \begin{tabular}{c}
 (oriented) split-conformal \\
 structure $(M^4;[\sfg])$
 \end{tabular}
 \end{tikzcd}
 \end{center}
 Moreover, infinitesimal conformal symmetries, i.e. conformal Killing fields of $(M^4,[\sfg])$, lift to to symmetries of the twistor XXO-structure, i.e. to vector fields on $N^5$ that via the Lie derivative preserve the pair of distributions $(\ell,\cD)$. Indeed the lift defines an isomorphism between the symmetry algebra of the conformal structure and that of the twistor XXO-structure.
 
From now on, we will restrict to conformal structures with {\sl multiply-transitive} twistor XXO-structures, i.e.\ we assume that the symmetry algebra $\ff$ of the XXO-structure acts transitively with isotropy algebra $\ff^0$ (at a generic point) having $\dim(\ff^0) \geq 1$.  Thus, we are focusing on homogeneous structures with $\dim(\ff) \geq 6$.  (In particular, the examples \eqref{E:ANex} lie outside the scope of our present investigations.)   Our main result is:

 \begin{theorem} \label{T:main}
The complete local classification of $4D$ split-conformal structures with multiply-transitive twistor XXO-structure and $\mathrm{G}_2$-symmetric twistor distribution is given in Table \ref{F:main}.
 \end{theorem}
 
\begin{table}[H]
 \[
 \begin{array}{|l|c|c|c|} \hline
 \mbox{Label} & \mbox{Representative metric $\sfg$} & \mbox{Petrov type }& \begin{tabular}{c} Conformal \\ symmetry algebra \end{tabular}
 \\ \hline\hline
 \mathsf{M9} & \mathsf{d}x\mathsf{d}u+\mathsf{d}y\mathsf{d}v+x^2\mathsf{d}v^2  &\mathsf{N.O}  &\fp_1 \\ \hline
 \mathsf{M8.1}& 
\frac{2}{(v+ux-y)^2}\left(u\mathsf{d}v \mathsf{d}x-\mathsf{d}v \mathsf{d}y-(v-y)\mathsf{d}u\mathsf{d}x-x\mathsf{d}u\mathsf{d}y\right)
 & \sfD^+.\sfO &\mathfrak{sl}(3,\bbR) \\ \hline
 \mathsf{M8.2} & 
 \begin{aligned}
\tfrac{1}{(2 u+x^2+y^2)^2}&(\mathsf{d}u^2+\mathsf{d}v^2-2y\mathsf{d}v\mathsf{d}x+2x\mathsf{d}v\mathsf{d}y\\+2x&\mathsf{d}u\mathsf{d}x
+ 2y\mathsf{d}u\mathsf{d}y-2u\mathsf{d}x^2-2u\mathsf{d}y^2)
\end{aligned}
 & \sfD^-.\sfO &\mathfrak{su}(1,2) \\ \hline
 \begin{array}{@{}l@{}} \mathsf{M7}_\sfa \\ \mbox{{\tiny $(\sfa^2 \in \bbR^\times)$}} \\ \mathsf{M7_0^\pm} \end{array} & 
 \begin{aligned}
&{ 9(2 \epsilon r^2 + 4 r x + y^2 - 1)\mathsf{d}u^2+12\epsilon( r y + 2 x  )\mathsf{d}u\mathsf{d}v}\\& -12\epsilon   \mathsf{d}u\mathsf{d}x+    (12 r^2- 6y-10 \epsilon ) \mathsf{d}v^2+ 6  \mathsf{d}v\mathsf{d}y, \\& \mbox{where}\quad  r=\vert \sfa\vert\geq 0,\quad \epsilon=\begin{cases}
\mathrm{sgn}(\sfa^2), & \sfa \neq 0;\\
 \pm 1, & \sfa = 0
\end{cases}
\end{aligned}
 &\begin{array}{@{}l@{\,}l@{}} 
 \sfN.\sfO: & \sfa^2 = \frac{4}{3}\\
 \sfN.\sfN: & \mbox{\tiny otherwise}
 \end{array}
 & \bbR^2 \ltimes \mathfrak{heis}_5\\ \hline
 \begin{array}{@{}c} \mathsf{M6S}.1\\ \mathsf{M6S}.2\\ \mathsf{M6S}.3\end{array}& 
 \begin{aligned}
&\frac{\mathsf{d}x^2+\epsilon \mathsf{d}y^2}{(1+\kappa(x^2+\epsilon y^2))^2} - \frac{\mathsf{d}u^2+\epsilon \mathsf{d}v^2}{(1+9\kappa(u^2+\epsilon v^2))^2},\\&\mbox{where}\quad \begin{array}{c|ccc} & \mathsf{M6S}.1 & \mathsf{M6S}.2 & \mathsf{M6S}.3\\ \hline (\kappa,\epsilon) & (1,-1) & (-1,1) & (1,1) \end{array}
\end{aligned}
 & \begin{array}{c}\sfD^+.\sfD^+ \\ \sfD^-.\sfD^- \\ \sfD^-.\sfD^- \end{array} & 
 \begin{array}{c} 
 \fsl(2,\bbR) \times \fsl(2,\bbR)\\
 \fsl(2,\bbR) \times \fsl(2,\bbR)\\
 \fso(3) \times \fso(3)
 \end{array} \\ \hline
 \mathsf{M6N} & 
 \mathsf{d}y\mathsf{d}u+ \mathsf{d}x\mathsf{d}v+2 e^{2 v}\mathsf{d}x \mathsf{d}y-2 e^{2 v}u\mathsf{d}y^2- u\mathsf{d}y\mathsf{d}v
 &\mathsf{III.O} &\mathfrak{aff}(2) \\ \hline
 \end{array}
 \] 
 \caption{All $4D$ split-conformal structures $(M^4;[\sfg])$ with multiply-transitive twistor XXO-structure $(N^5; \ell, \cD)$ and  $\mathrm{G}_2$-symmetric twistor distribution $(N^5; \cD)$}
 \label{F:main}
 \end{table}


  Equally important in our story are the techniques used to derive our main result above, and alternative manners of  presenting it.  For any homogeneous structure, we emphasize three {\em equivalent} ways of presenting it: (i) in local coordinates, (ii) Lie-theoretic, and (iii) Cartan-theoretic.  The Petrov types listed above concern the root types of the SD part $\cW^+$ and ASD part $\cW^-$ of the Weyl tensor, viewed as a binary quartic in spinorial language.  We circumvent working directly with spinorial Weyl tensor coefficients, and instead demonstrate how these Petrov types are efficiently obtained via (SD \& ASD) twistor XXO-structures starting from each of the presentations (i)--(iii) above.
 
 The Cartan-theoretic presentation is the most abstract, but has the advantage of giving us a clear understanding of the curvature of these geometries, and an easy way to compute their (conformal) holonomy (Table \ref{F:CHol}).  In turn, this allows us to efficiently {\em algebraically} assess which of the  conformal structures  $[\sfg]$ admit  Einstein representatives on an open subset (Theorem \ref{T:AE}).


 \section{Geometric structures}


 \subsection{$(2,3,5)$-distributions}
 
 
  The study of $(2,3,5)$-distributions has a long history, originating in the seminal 1910 work of \'Elie Cartan \cite{Car1910}.
\begin{defn}
 A {\sl $(2,3,5)$-distribution} $\cD$ on a $5$-manifold $N$ is a rank 2 subbundle $\cD\subset TN$ such that  $[\cD,\cD]$ is a rank $3$ subbundle of $TN$ and $[\cD,[\cD,\cD]] = TN$. 
 \end{defn}
 
Any $(2,3,5)$-distribution can be locally specified in terms of a function $F=F(x,y,z,p,q)$ satisfying $F_{qq}\neq 0$. Namely, there exist local coordinates $(x,y,p,q,z)$  such that the distribution is the common kernel of the $1$-forms
\begin{equation}\label{Monge}\mathsf{d}y-p\mathsf{d}x,\quad\mathsf{d}p-q\mathsf{d}x, \quad\mathsf{d}z-F \mathsf{d}x,
\end{equation}
or equivalently, it is spanned by the vector fields
$\partial_q$
and $\partial_x+p\partial_y+q\partial_p+F \partial_z.$
Infinitesimal symmetries of $\cD$ are vector fields $X\in\Gamma(TN)$ such that $\mathcal{L}_XY=[X,Y]\in\Gamma(\cD)$ for all $Y\in\Gamma(\cD)$. A local model of the maximally symmetric $(2,3,5)$-distribution whose Lie algebra of infinitesimal symmetries is the split-real form of the $14$-dimensional exceptional simple Lie algebra of type $\mathrm{G}_2$ can, for example, be obtained by choosing $F=q^2$ in \eqref{Monge}.

 
 \subsection{Geometry on the circle twistor bundle}
 \label{sec-construction}
 
 
In \cite{AN2014}, the construction of $(2,3,5)$-distributions from rolling surfaces was interpreted as a special case of a construction familiar from twistor theory.
 
  Let $(M,[\sfg])$ be a smooth, oriented, conformal $4$-manifold of signature $(2,2)$. Locally, one can always find a null coframe such that the metric $\sfg$ and volume form $\mathrm{vol}_\sfg$ are expressed as
 \begin{align}\label{null-cof}
 \sfg = \sfg_{ij}\theta^i\theta^j = 2(\theta^1\theta^2+\theta^3\theta^4), \quad 
 \mathrm{vol}_\sfg=\theta^1\wedge\theta^2\wedge\theta^3\wedge\theta^4,
 \end{align}
 where $\theta^i\theta^j:=\tfrac{1}{2}(\theta^i\otimes\theta^j+\theta^j\otimes\theta^i)$.
 Then the Hodge-star operator $\star:\Lambda^2 T^*M\to\Lambda^2 T^*M$, defined via the relation $\alpha\wedge *\beta=\sfg(\alpha,\beta)\mathrm{vol}_\sfg$, satisfies
\begin{equation}
\begin{aligned}
 &\star \theta^1\wedge\theta^2=-\theta^3\wedge\theta^4,\quad \star\theta^1\wedge\theta^3=-\theta^1\wedge\theta^3,\quad   \star\theta^1\wedge\theta^4=\theta^1\wedge\theta^4,\\
 &\star\theta^3\wedge\theta^4=-\theta^1\wedge\theta^2,\quad   \star\theta^2\wedge\theta^4=-\theta^2\wedge\theta^4,\quad   \star\theta^2\wedge\theta^3=\theta^2\wedge\theta^3.
\end{aligned}
\end{equation}
Let $\Pi_x\subset T_xM$ be a totally null 2-plane and  suppose that $\phi$ and $\psi$ span its annihilator. Then $\Pi$ is {\sl self-dual (SD)} (respectively, {\sl anti-self-dual (ASD)}) if and only if $\phi\wedge\psi$ is an eigenvector of the Hodge star operator with eigenvalue $+1$ (respectively, $-1$).
 
 The set of all totally null 2-planes in $T_xM$ is a disjoint union of the SD and ASD ones, denoted $\bbT^{+}(T_xM)$ and $\bbT^-(T_xM)$ respectively, each of which is diffeomorphic to $\mathbf{S}^1$.  Let us focus on SD family.  (Considerations for the ASD family are similar.)  One can form the {\sl circle twistor bundle} 
 \begin{align}
 \mathbf{S}^1\hookrightarrow \bbT^+ := \bigsqcup_{x\in M}\bbT^+(T_xM) \overset{\pi}{\longrightarrow }M,
 \end{align}
where $\pi$ is the natural projection $(x,\Pi)\mapsto x$.  The total space $\bbT^+$ is naturally equipped with the following geometric structures:
\begin{itemize}
\item A rank one vertical bundle $\ell = \ker(\pi_*) \subset T\bbT^+$.
\item  A canonical rank three distribution $\cH \subset T\bbT^+$:  A point $p=(x,\Pi)\in\bbT^+$ corresponds to a SD totally null $2$-plane $\Pi\subset T_xM$ and we define $\cH_{p}:=\{ X\in T_p\bbT^+ \ \vert\ \pi_* X\in\Pi \}$.
This distribution is {\sl maximally non-integrable}, i.e.\ $[\cH,\cH]=T\bbT^+$.
\item Maximal non-integrability means that the map $\Lambda^2\cH\to T\bbT^+ / \cH$ (induced by the Lie bracket) is surjective, so has a $1$-dimensional kernel. The kernel is spanned by decomposable elements and thus defines a rank two distribution $\cD$ uniquely characterized by the properties
 \begin{align} \label{E:D-def}
 \cD \subset \cH, \qbox{and} [\cD,\cD] \subset \cH.
 \end{align}
\end{itemize}

 \begin{defn}
 The canonical rank two distribution $\cD$ on $\bbT^+$ is called the (SD) {\sl twistor distribution}. The structure $(\ell, \cD)$ is called the (SD) {\sl twistor XXO-structure}. We have $\cH = \ell \op \cD$.
 \end{defn}

A vector field $X$ on $\bbT^+$ such that $\mathcal{L}_X\Gamma(\cD)\subset\Gamma(\cD)$ and $\mathcal{L}_X\Gamma(\ell)\subset\Gamma(\ell)$ is called an {\sl infinitesimal symmetry} of the XXO-structure.   Conformal symmetries of $[\sfg]$ naturally lift from $M$ to $\bbT^+$. Indeed, the lift defines an isomorphism between the symmetry algebras of the conformal structure $(M,[\sfg])$ and that of the XXO-geometry on $\bbT^+$. This follows most easily from the Cartan geometric description of the construction discussed in Section \ref{correspondence}. The dimension of the symmetry algebra of the twistor distribution $\cD$ alone is in general larger than that of the conformal structure.


 To obtain the twistor distribution $\cD$ explicitly,  let $(\theta^1, \theta^2, \theta^3, \theta^4)$ be a positively oriented null (local) coframing on $M$ as in \eqref{null-cof} and let $(e_1, e_2, e_3, e_4)$ be the dual (local) framing. The set of all SD totally null $2$-planes at a point\footnote{We note that the ASD totally null 2-planes are given by $\langle s e_1 + t e_4, s e_3 - t e_2 \rangle$.} can be parametrized  as follows
 \begin{align}
 \Pi(s,t)=\langle s e_1+ t e_3, s e_4-t e_2 \rangle =\mathsf{ker}(s\theta^2+t\theta^4, s\theta^3-t\theta^1), \quad [s:t]\in\bbR\bP^1\cong\bS^1.
 \end{align}
Let us work in a local trivialization $\bbT^+|_U := \pi^{-1}(U) \to U \times \bbR\bP^1$. Restrict to the open subset where $s\neq 0$ and introduce an affine (fibre) coordinate $\xi=\tfrac{t}{s}$.
 

Then $\cH = \mathsf{ker}(\omega^1,\omega^2)$, with
\begin{equation}
\label{om12}
\begin{aligned}
\omega^1 = \theta^2 + \xi\theta^4, \quad
\omega^2 = \theta^3 - \xi\theta^1,
\end{aligned}
\end{equation}
where we have omitted pullbacks to simplify notation. Moreover, $\cD = \mathsf{ker}(\omega^1,\omega^2,\omega^3)$, where $\omega^3$ is chosen such that $[\cD,\cD]\subseteq \cH$. Equivalently, this means that
\begin{equation}\label{DfromH}
\mathsf{d}\omega^1\wedge\omega^1\wedge\omega^2\wedge\omega^3=0,\quad \mathsf{d}\omega^2\wedge\omega^1\wedge\omega^2\wedge\omega^3=0.
\end{equation}

 At $p=(x,\Pi)\in\bbT^+$, $\cD_{p}$ can alternatively be obtained as the horizontal lift of $\Pi\subset T_x M$ with respect to a metric in $[\sfg]$ (and the result is independent of the choice of metric) \cite{AN2014}.
Let $\Gamma^i{}_j$ be the Levi-Civita connection $1$-forms for \eqref{null-cof} determined by the structure equations
 \begin{align}
 \mathrm{d}\theta^i+\Gamma^i{}_j\wedge\theta^j=0, \quad 
 \Gamma_{ij}+\Gamma_{ji}=0,
 \end{align}
 where $\Gamma_{ij}=\sfg_{ik}\Gamma^k{}_j$, and $\Gamma^i{}_{jk}$ are the connection coefficients defined via $\Gamma^i{}_j=\Gamma^i{}_{jk}\theta^k$. 
Then $\cD = \mathsf{ker}(\omega^1,\omega^2,\omega^3)$, for $\omega^1$, $\omega^2$ as in \eqref{om12}, and
\begin{equation}
\label{om3}
\begin{aligned}
\omega^3&=\mathsf{d}\xi-\Gamma^2{}_4+\xi(\Gamma^3{}_3+\Gamma^2{}_2)-\xi^2\Gamma^1{}_3 .
\end{aligned}
\end{equation}
Adding $\omega^4=\theta^1$ and $\omega^5=\theta^4$ completes these to a coframing. Then  a direct computation shows:
\begin{lemma} \label{lemm-Weyl}  We have
\begin{equation} \label{W(xi)}
\mathsf{d}\omega^3\wedge\omega^1\wedge\omega^2\wedge\omega^3=\, (\Psi_0+4\Psi_1\xi+6\Psi_2\xi^2-4\Psi_3\xi^3+\Psi_4\xi^4)\, \omega^1\wedge\omega^2\wedge\omega^3\wedge\omega^4\wedge\omega^5,
\end{equation}
where
the scalars $\Psi_0,\Psi_1,\Psi_2,\Psi_3,\Psi_4$ are the components of the SD Weyl tensor $\mathcal{W}^+$ in the notation from \cite[p.\ 330]{GHN2011}. 
 \end{lemma}

It follows that $\cD$ is integrable if and only if $\cW^+ \equiv 0$, i.e.\ $[\sfg]$ is ASD. Assuming that $\cW^+ \equiv 0$, one can pass to the local leaf space for the integrable distribution $\mathcal{D}$, which leads to the split-signature version of Penrose's twistor construction \cite{Penrose1967}, \cite[Sec.7.3 \& 7.4]{PR1986}.  In \cite{AN2014}, the opposite case was considered, i.e.\ when $\cW^+$ is \emph{non-vanishing}. It is shown that:

\begin{prop} Let $\cD$ be the twistor distribution on the total space $\bbT^+$ of the bundle of SD totally null $2$-planes. Then $\cD$ is a $(2,3,5)$-distribution on the open set of points in $\bbT^+$ where 
$$\cW^+(\xi):=\Psi_0+4\Psi_1\xi+6\Psi_2\xi^2-4\Psi_3\xi^3+\Psi_4\xi^4\neq 0.$$
 \end{prop}


 \subsubsection{An Example}
 \label{sec-Ex} 
 
 
 Let us illustrate the construction from Section \ref{sec-construction} on a special class of metrics that contains {\em conformally inequivalent} examples having $\mathrm{G}_2$-symmetric twistor distribution.
 Let
 \begin{align} \label{E:metric}
 \sfg=\mathsf{d}x\mathsf{d}u+\mathsf{d}y\mathsf{d}v-\Theta_{xx}\mathsf{d}v^2-\Theta_{yy}\mathsf{d}u^2+2\Theta_{xy}\mathsf{d}u\mathsf{d}v
 \end{align}
be the general Ricci-flat, SD, split-signature metric,
where the function $\Theta=\Theta(x,y,u,v)$ satisfies {\sl Pleba\'nski's second heavenly equation} \cite{Plebanski1975}:
 \begin{align}
 \Theta_{ux}+\Theta_{vy}+\Theta_{xx}\Theta_{yy}-\Theta^2_{xy}=0.
 \end{align}
In \cite{AN2014}, the special case $\Theta=\Theta(x)$ was considered with $\Theta^{(4)}=\Theta_{xxxx}\neq 0$. They showed that the associated twistor distribution $\cD$
has $\mathrm{G}_2$-symmetry if and only if $\Theta$ satisfies the ODE
 \begin{align} \label{ODE}
 10 (\Theta^{(4)})^3 \Theta^{(8)} - 70 (\Theta^{(4)})^2 \Theta^{(5)} \Theta^{(7)} 
 - 49 (\Theta^{(4)})^2 (\Theta^{(6)})^2 + 280 \Theta^{(4)} (\Theta^{(5)})^2 \Theta^{(6)} - 175(\Theta^{(5)})^4 = 0.
 \end{align}
 This is equivalent to $\Theta'$ being a solution to the 7th order ODE studied by Dunajski--Sokolov \cite{DS2011}.

 As a particular example, 
consider the metric of the form $\sfg=2(\theta^1\theta^2+\theta^3\theta^4)$, where
 \begin{align} \label{coframeTheta}
 \theta^1 = \mathsf{d}x, \quad 
 \theta^2 = \mathsf{d}u, \quad
 \theta^3 = \mathsf{d}y - lx^{l-1}\mathsf{d}v, \quad
 \theta^4 = \mathsf{d}v.
 \end{align}
 On the open subset with coordinates by $(x,y,u,v,\xi)$, we have $\cD = \mathsf{ker}(\omega^1,\omega^2,\omega^3)$, where
 \begin{align}\label{ex-om123}
 \omega^1 = \mathsf{d}u + \xi\mathsf{d}v, \quad
 \omega^2 = -\xi\mathsf{d}x+\mathsf{d}y-lx^{l-1}\mathsf{d}v, \quad
 \omega^3 = \mathsf{d}\xi+(1-l)lx^{l-2}\mathsf{d}v.
 \end{align}
(In \cite{AN2014}, a transformation is provided that brings \eqref{ex-om123} to the Monge form \eqref{Monge}.)  From \eqref{W(xi)}, 
 \begin{align}
 \cW^+(\xi) = l(l-1)(l-2)x^{l-3}.
 \end{align}
 Thus, $\cW^+ \equiv 0$ when $l \in \{ 0,1,2 \}$ (cases with $\Theta^{(4)} = 0$), in which case $\sfg$ is conformally flat and both SD and ASD twistor distributions are integrable.  Otherwise, $\cW^+$ is of type $\sfN$ (with quadruple root at $\xi = \infty$) and $\cD$ is $(2,3,5)$.  From \eqref{ODE}, $\cD$ is $\mathrm{G}_2$-symmetric if $l \in \{ -1, \tfrac{1}{2}, \tfrac{3}{2}, 3\}$.
 
 For any $l$, the symmetry algebra of $[\sfg]$ includes the 7-dimensional subalgebra
 \begin{equation}
 \begin{aligned}
 & \partial_y, \quad \partial_u, \quad \partial_v, \quad x\partial_y-v\partial_u, \quad -2 x^l\partial_y+y\partial_u-x\partial_v,\\
 & x\partial_x + (l-1)y\partial_y + (l-2)u\partial_u, \quad y\partial_y + 2u\partial_u + v\partial_v.
 \end{aligned}
 \end{equation}
 If $l\in\{\tfrac{1}{2},\tfrac{3}{2}\}$, this is the entire symmetry algebra.  When $l \in \{ -1, 3 \}$, it is augmented by two additional generators, e.g.\ when $l = 3$, we also have
 \begin{align}
 \partial_x + 6xv \partial_y - 3 v^2 \partial_u, \quad
 -v\partial_x + (u - 3xv^2) \partial_y + v^3 \partial_u.
 \end{align} 
 Thus, there are conformally-inequivalent cases here.  Referring to Theorem \ref{T:main}, we have:
 \begin{itemize}
 \item $\mathsf{M9}$: $l = -1$ and $l = 3$ (which are conformally equivalent);
 \item $\mathsf{M7}_\sfa$ for $\sfa^2 = \frac{4}{3}$: $l=\tfrac{1}{2}$ and $l=\tfrac{3}{2}$ (which are conformally equivalent).
 \end{itemize}


 \subsection{XXO structures}
 \label{S:XXO}
 
 
 Motivated by the structure on the total space of the twistor bundle over a split-signature conformal structure, let us now introduce the general notion of an XXO-structure. 
 \begin{defn}
 An {XXO}-structure $(\ell,\cD)$ on a $5$-dimensional smooth manifold $N$ consists of a rank $2$ subbundle $\cD \subset TN$ and a  line subbundle $\ell \subset TN$  such that
 \begin{enumerate}
 \item $\ell \cap \cD = 0$,
 \item $[\Gamma(\cD),\Gamma(\cD)] \subset \Gamma(\cH)$, where $\cH := \ell \op \cD$,
 \item the Lie bracket induces an isomorphism $\ell \otimes \cD \cong TN/\cH$.
 \end{enumerate}
 The XXO-structure $(\ell,\cD)$ is {\sl integrable} if $\cD$ is (Frobenius) integrable.  Otherwise, it is {\sl non-integrable}.
 \end{defn}

 \begin{remark} 
 We note that for a general XXO-structure $(\ell,\cD)$, the derived distribution $[\cD,\cD]$ may not have constant rank.
 \end{remark}
 
 If $(\ell,\cD)$ is an integrable XXO-structure, then it is locally equivalent to a pair of 2nd order ODE. (This follows by combining Proposition 1.5 and Theorem 1.6 of \cite{Yam1983}.)
 
 \begin{example} \label{ex-ODE} The geometry (up to point-equivalence) of a pair of 2nd order ODE, 
 \begin{align} \label{E:2ODE}
 \ddot{x} = F(t,x,y,\dot{x},\dot{y}), \quad \ddot{y} = G(t,x,y,\dot{x},\dot{y}),
 \end{align}
 is encoded by an integrable XXO-structure as follows. Consider the 1st jet space $N = J^1(\bbR,\bbR^2)$, which has dimension $5$, and standard local coordinates $(t,x,y,p,q)$ for which the canonical Cartan distribution is $\cC = \ker\langle dx - p dt, dy - q dt \rangle = \langle \partial_t + p \partial_x + q \partial_y, \partial_p, \partial_q \rangle \subset TN$.  Define a splitting of $\cC $ into rank 1 and rank $2$ subdistributions
 \begin{align} \label{E:2ODE-LD}
 \ell = \langle \partial_t + p \partial_x + q \partial_y + F \partial_p + G \partial_q \rangle, \quad \cD = \langle \partial_p, \partial_q \rangle,
 \end{align}
 where $F$ and $G$ are functions of $(t,x,y,p,q)$.  Then $(\ell,\cD)$ is an integrable XXO-structure on $N$.
 
 \end{example}
 
 Augmenting a $(2,3,5)$-distribution with additional structure yields non-integrable XXO-structures:
 
 \begin{example}\label{ex-235+line}
 Consider the $\mathrm{G}_2$-symmetric $(2,3,5)$-distribution
 \begin{align}
 \cD=\left\langle D_x:=\partial_x+p\partial_y+q\partial_p+\tfrac{1}{2}q^2\partial_z,\,\partial_q\right\rangle
 \end{align}
 corresponding to $F=\tfrac{1}{2}q^2$ in \eqref{Monge}. Then $[\cD,\cD]/\cD$ is represented by $\partial_p+q\partial_z$ and we can define a non-integrable XXO-structure $(\ell,\cD)$ via any choice of line field
 \begin{align}
 \ell=\left\langle \partial_p+q\partial_z+A D_x+B\partial_q\right\rangle,
 \end{align}
 for functions $A,B$ of $(x,y,p,q,z)$.
 \end{example}
 
 In each of these examples, the distribution $\cD$ is the same, but different choices of line fields $\ell$ may lead to inequivalent XXO-structures.


 \section{Cartan geometric picture}
 
 
 Many features of the construction in Section \ref{sec-construction} are best understood if one rephrases the construction in terms of the associated (parabolic) Cartan geometries. 


 \subsection{Parabolic geometries}
 \label{S:PG}


 Let $G/P$ be a homogeneous space with corresponding Lie algebras $\fp \subset \fg$.  A {\sl Cartan geometry} $(\cG \to M,\omega)$ of type $(\fg,P)$  is given by a right principal $P$-bundle $p:\cG\to M$ together with
 a {\sl Cartan connection} $\omega\in\Omega^1(\cG,\fg)$.  This satisfies:
 \begin{enumerate}
 \item[(i)] $\omega$ defines an isomorphism $\omega_u:T_u\cG\to \fg$ for any $u\in\cG$;
 \item[(ii)] $\omega$ is $P$-equivariant, i.e.\ $r_g^*\omega=\mathrm{Ad}_{g^{-1}}\circ\omega$ for any $g\in P$;
 \item[(iii)] $\omega$ maps fundamental vector fields to their generators, i.e.\ $\omega(\zeta_X)=X$ for any $X\in\fp$.
\end{enumerate}
 In particular, the Cartan connection $\omega$ provides an identification $\cG \times_P\fg/\fp\cong TM$.
A morphism between Cartan geometries $(\cG\to M,\omega)$ and $(\cG'\to M',\omega')$ is a principal bundle morphism $\Phi:\cG\to\cG'$ such that $\Phi^*\omega'=\omega$.

 The {\sl curvature} of $\omega$ is the $2$-form $K\in\Omega^2(\cG,\fg)$ defined as 
 \begin{align}
 K(X,Y)=\mathsf{d}\omega(X,Y)+[\omega(X),\omega(Y)]\quad\mbox{for}\quad X,Y\in\Gamma(T\cG).
 \end{align}
 It is $P$-equivariant, horizontal, and can be equivalently encoded by the {\sl curvature function}
 \begin{align}
 \kappa:\cG\to\Lambda^2(\fg/\fp)^*\otimes\fg,\quad \kappa(v,w):=K(\omega^{-1}(v),\omega^{-1}(w)) \quad\mbox{for}\quad v,w \in \fg.
 \end{align}
The curvature is a complete obstruction to local equivalence of $(\cG\to M, \omega)$ with the homogeneous model $(G\to G/P,\omega_G)$, where $\omega_G$ denotes the left-invariant Maurer--Cartan form on $G$.

A {\sl parabolic geometry} is a Cartan geometry of type $(\fg,P)$, where $\fg$ is semisimple, $\fp\subset\fg$ is a parabolic subalgebra, and $P \subset \Aut(\fg)$ is a Lie subgroup with Lie algebra $\fp$.  Parabolic subalgebras can be characterized by the property that the orthogonal complement $\fp^{\perp}$ of $\fp$ in $\fg$ with respect to the Killing form coincides with the nilradical of $\fp$, which will be denoted by $\fp_+$. It then follows that $\fp/\fp_+$ is reductive and $\fp_+\cong (\fg/\fp)^*$ as $\fp$-modules. Defining $\fg^1=\fp_+$, $\fg^i=[\fg^{i-1},\fp_+]$ for $i\geq 2$ and $\fg^{j+1}=(\fg^{-j})^{\perp}$ for $j\leq -1$, one obtains a ({\sl depth} $k$) filtration
 \begin{equation} \label{E:filt}
 \fg = \fg^{-k} \supset ... \supset \fg^0 \supset... \supset \fg^k,\quad
 [\fg^i,\fg^j]\subseteq \fg^{i+j},
\end{equation}
which is $P$-invariant with respect to the adjoint action.  The filtrand $\fg^i$ is said to be of {\sl homogeneity} $\geq i$.  There is a bijective correspondence between conjugacy classes of parabolic subalgebras and subsets of simple restricted roots, see \cite{CS2009} for details. In particular, parabolic subalgebras can be depicted by means of marked Satake diagrams. More precisely, a parabolic subalgebra $\fp$ is represented by crosses  on the nodes of the Satake diagram  corresponding to simple roots $\alpha_i$ such that the root space $\fg_{-\alpha_i}$ is not contained in $\fp$.

 Given the filtration \eqref{E:filt}, we can form its associated-graded Lie algebra $\tgr(\fg) = \bigoplus_i \tgr_i(\fg)$, where $\tgr_i(\fg) := \fg^i / \fg^{i+1}$, and the bracket is induced.  There is always a choice of grading element $\sfZ \in \fp$ that lifts this grading on $\tgr(\fg)$ to a grading on $\fg$ that is compatible with the filtration:
 \begin{align}
 \fg = \fg_{-k} \op ... \op \fg_0 \op ... \op \fg_k, \quad [\fg_i,\fg_j] \subseteq \fg_{i+j}, \quad \fg^i = \fg_{\geq i}.
 \end{align}
 Thus, $\fg^0 = \fp$ and $\fp_+ = \fg_{\geq 1}$.
 Here, $\sfZ \in \fz(\fg_0)$ (center of $\fg_0$), and $\fg_i$ is the eigenspace of $\ad_\sfZ$ for the eigenvalue $i \in \bbZ$, also called its {\sl homogeneity} (which is compatible with $\fg^i$ having homogeneity $\geq i$).  Eigenvalues of $\sfZ$ on any $\fg_0$-module will also be referred to as homogeneities.

A parabolic geometry is:
\begin{itemize}
\item {\sl regular} if $\kappa$ is of homogeneity $\geq 1$ with respect to this filtration, i.e.\ $\kappa_u(\fg^i,\fg^j)\subseteq\fg^{i+j+1}$ for any $u\in\cG$, 
\item {\sl normal} if $\partial^*\circ\kappa=0$, where $\partial^*$ is the Lie algebra homology differential in the chain complex $(C_{\bullet}(\fp_+,\fg),\partial^*)$, where $C_k(\fp_+,\fg):=\Lambda^k\fp_+\otimes\fg \cong_\fp \Lambda^k (\fg/\fp)^* \otimes \fg$. 
\end{itemize}
The homology spaces $H_k(\fp_+,\fg)=\mathrm{ker}(\partial^*)/\mathrm{im}(\partial^*)$ are completely reducible $P$-modules (so $\fp_+$ acts trivially on them), their $\fg_0$-module structure can be algorithmically computed using Kostant's theorem \cite{Kos1961,CS2009}, and this can be explicitly identified with so-called {\sl harmonic elements} of $\Lambda^2\fg_-^* \otimes \fg$.  The projection of the curvature $\kappa$ of a regular normal parabolic geometry to  $\mathrm{ker}(\partial^*)/\mathrm{im}(\partial^*)=H_2(\fp_+,\fg)$ is called the {\sl harmonic curvature} $\kappa_H$.  (By regularity, it is valued in $H_2(\fp_+,\fg)^1$, i.e.\ the positive homogeneity part.)  The harmonic curvature is a simpler object than the full curvature, but it determines the full curvature via a differential operator.  In particular, vanishing of $\kappa_H$ on an open neighbourhood implies vanishing of $\kappa$ there \cite[Theorem 3.1.12]{CS2009}.


 \subsection{Categorical equivalences}
 \label{S:cat}


 There is a fundamental theorem establishing an equivalence of categories between regular normal parabolic geometries of type $(\fg,P)$ and certain underlying geometric structures, see \cite[Theorem 3.1.14]{CS2009}.  In Table \ref{F:modeldata}, we specify the structure and corresponding model data $(\fg,\fp)$, for which we define $P \subset \Aut(\fg)$ as the subgroup preserving the associated filtration \eqref{E:filt}.  For these three cases, the underlying structure corresponds to the $P$-submodule $\fg^{-1} / \fp \subset \fg / \fp$ via the associated bundle construction $\cG \times_P \fg / \fp$. 
 \begin{table}[h]
 \[
 \begin{array}{|c|c|c|c|c|} \hline
 \mbox{Structure} & \multicolumn{2}{c|}{\mbox{Model data}} & \mbox{Depth} & \mbox{Grading}\\ \hline\hline
 \mbox{$(2,3,5)$} & \fg = \Lie(G_2) & \fp = \Gdd{xw}{} & 3 & \begin{tabular}{cc}
\begin{tikzpicture}[scale=0.8]
\foreach \i in {-3,...,3}
	\draw[\graycolor, xshift=3*\i mm, yshift=4*\i mm]  (150:2) to (150:-2) node[below right]{$\mathfrak{g}_{\i}$};
\foreach \i in {180,240,300}
	\draw[->] (0,0) to (\i:1);
\foreach \i in {210, 270}
	\draw[->] (0,0) to (\i:{sqrt(3)});
\foreach \i in {0,60,120}
	\draw[->, red] (0,0) to (\i:1);
\foreach \i in {-30,30,...,150}
	\draw[->, red] (0,0) to (\i:{sqrt(3)});
\fill[red] (0,0) circle (0.1);
\end{tikzpicture} 
\end{tabular}\\ \hline
 \mbox{XXO} & \wfg = \fsl(4,\bbR) & \wfp =  \Athree{xxw}{} & 2 & \begin{pmatrix} \red{0} & \red{1} & \red{2} & \red{2}\\ -1 & \red{0} & \red{1} & \red{1}\\ -2 & -1 & \red{0} & \red{0}\\ -2 & -1 & \red{0} & \red{0} \end{pmatrix}\\ \hline
 \mbox{4D split-conformal} & \wfg= \fsl(4,\bbR) & \wfq = \Athree{wxw}{} & 1 & \begin{pmatrix} \red{0} & \red{0} & \red{1} & \red{1}\\ \red{0} & \red{0} & \red{1} & \red{1}\\ -1 & -1 & \red{0} & \red{0}\\ -1 & -1 & \red{0} & \red{0} \end{pmatrix}\\ \hline
 \end{array}
 \] 
 \caption{Model data for our structures of interest}
 \label{F:modeldata}
 \end{table}
 
 We remark that the model homogeneous space in the 4D split-conformal case is the Grassmannian of 2-planes in $\bbR^4$, while in the XXO case it is the flag manifold $\bbF_{1,2}(\bbR^4)$ of lines incident on 2-planes.  For parabolic subgroups, we will use notations $P, \widetilde{P}, \widetilde{Q}$ respective to $\fp,\wfp,\wfq$ defined in Table \ref{F:modeldata}.
 
 Recall that $\mathrm{SL}(4,\bbR) \cong \mathrm{Spin}(3,3)$, and we can take $\widetilde{Q}\subset\mathrm{SL}(4,\bbR)$ as the stabilizer of the 2-plane $\left\langle e_1,e_2\right\rangle\subset\bbR^4$, i.e.\
 \begin{align}
 \widetilde{Q}\cong\left\{ \begin{pmatrix}
 A& C\\
 0&B
 \end{pmatrix}:\ A,B\in\mathrm{GL}(2,\bbR),\ \mathrm{det}(A)\mathrm{det}(B)=1 \right\}.
 \end{align}
For this choice of groups $\widetilde{Q}\subset\mathrm{SL}(4,\bbR)$, the general theory implies that there an equivalence of categories between 4D split-conformal {\em spin} structures and regular normal parabolic geometries of type $(\mathfrak{sl}(4,\bbR),\widetilde{Q})$.  In particular, $\pm \id_4$ act trivially on $\mathrm{SL}(4,\bbR) / \widetilde{Q}$.

To see the conformal structure on the quotient $\mathfrak{sl}(4,\bbR)/\widetilde{\mathfrak{q}}$, identify this with the space of  $2\times 2$ matrices $\mathrm{M}_2\bbR$, embedded as the bottom-left $2\times 2$ block in $\mathfrak{sl}(4,\bbR)$. The $\widetilde{Q}$-action (induced by the adjoint action) on this space is given by $X\mapsto BXA^{-1}$ (the $C$-block acts trivially on the quotient). Now the determinant defines a scalar product of signature $(2,2)$ on $\mathrm{M}_2\bbR$, which is conformally preserved by the action, since $\mathrm{det}(BXA^{-1})=(\mathrm{det}(B))^2\mathrm{det}(X)$.


 \subsection{Harmonic curvatures}


 In Table \ref{F:KH}, we summarize harmonic curvatures (and introduce notation for them).

 \begin{table}[h]
 \[
 \begin{array}{|c|c|c|c|c|} \hline
 \mbox{Structure} & H_2(\fp_+,\fg)^1 & \mbox{Dimension} & \mbox{Homogeneity} & \mbox{Harmonic curvatures} \\ \hline\hline
 (2,3,5) & \Gdd{xw}{-8,4} & 5 &4 &\mbox{Cartan quartic: $\cQ$} \\ \hline
 \multirow{4}{*}{\mbox{XXO}} & \Athree{xxw}{0,-4,4} & 5 & 3 & \cS:\cD\times TN/\cH\to\mathrm{End}(\cD) \\
 & \Athree{xxw}{-4,1,2} & 3 & 2 & \cT:\ell\times TN/\cH\to\cD \\
 & \Athree{xxw}{4,-4,0} & 1 & 1 & \cI: \Lambda^2 \cD \to \ell\\ \hline
 \multirow{2}{*}{\mbox{4D split-conformal} }& \Athree{wxw}{0,-4,4} & 5 & 2 & \mbox{ASD Weyl: $\cW^-$} \\
 & \Athree{wxw}{4,-4,0} & 5 & 2 & \mbox{SD Weyl: $\cW^+$}\\ \hline
 \end{array}
 \]
 \caption{Harmonic curvature for our structures of interest}
 \label{F:KH}
 \end{table}

The component $\cI$ is given by $(x,y)\mapsto\mathrm{pr}_{\ell}([x,y])$, where $\mathrm{pr}_\ell : \cH \to \ell$ is the projection of $\cH = \ell \op \cD$ along $\ell$.  It is precisely the obstruction to integrability of the rank $2$ distribution $\cD$.

 For pairs of 2nd order ODEs, i.e.\ integrable XXO-structures, the components $\cS$ and $\cT$ specialize to the {\sl Fels curvature} and {\sl Fels torsion} respectively \cite{Fels1995}.  In Appendix \ref{S:torsion}, we give a formula for $\cT$ for general XXO-structures -- see \eqref{E:XXOtorsion}.  

 In Section \ref{S:CarTh}, we will need specific information about the harmonic curvature module $H_2(\wfq_+,\wfg)$ in the 4D split-conformal setting that we summarize here.  Recall the model data $(\wfg = \fsl(4,\bbR),\wfq)$ from Section \ref{S:cat}, which induces a decreasing filtration $\wfg^{-1} \supset \wfg^0 = \wfq \supset \wfg^1 = \wfq_+$.  As a module for $\widetilde\fg_0 \cong \bbR \times \fsl(2,\bbR) \times \fsl(2,\bbR)$, $H_2(\wfq_+,\wfg)$ is identified with so-called harmonic elements in $\Lambda^2\wfg_-^* \otimes \wfg$.  By Kostant's theorem \cite{Kos1961}, this module decomposes into two $\wfg_0$-irreps.


  Explicitly, let $E_{ij}$ be the $4\times4$ matrix with a 1 in the $(i,j)$-position, and 0 elsewhere.  Also define $H_{ij} := E_{ii} - E_{jj}$.  The trace form $T(A,B) = \tr(AB)$ on $\wfg$ identifies $(\wfg/\wfq)^* \cong_{\wfq} \wfq_+$, in particular
 \begin{align}
 (E_{13}, E_{14}, E_{23}, E_{24}) \leftrightarrow  (E_{31}^*, E_{41}^*, E_{32}^*, E_{42}^*),
 \end{align}
Kostant's theorem yields lowest weight vectors $\psi_0,\phi_0$ for the two $\wfg_0$-irreps, on which the raising operators $E_{12}$ and $E_{34}$ can be applied -- see Table \ref{F:harmonic}.  The second column gives the abstract structure of these modules when viewed as binary quartics for the respective $\fsl(2,\bbR)$-actions.

  \begin{table}[h]
  \begin{footnotesize}
  \[
 \begin{array}{|c|c|l|c}\hline
 \mbox{Module} & \mbox{Label} & \mbox{Harmonic 2-cochain}\\ \hline\hline
  \multirow{5}{*}{ $\Athree{wxw}{0,-4,4}$ }& 1 \boxtimes x^4 & \psi_4 = E_{41}^* \wedge E_{42}^* \otimes E_{34} \\
 & 1 \boxtimes 4 x^3 y & \psi_3 = (E_{42}^* \wedge E_{31}^* + E_{32}^* \wedge E_{41}^*) \otimes E_{34} + E_{41}^* \wedge E_{42}^* \otimes H_{43} \\
 & 1 \boxtimes 6 x^2 y^2 & \psi_2 = E_{42}^* \wedge E_{41}^* \otimes E_{43} + E_{31}^* \wedge E_{32}^* \otimes E_{34} + (E_{42}^* \wedge E_{31}^* + E_{32}^* \wedge E_{41}^*) \otimes H_{43}
 \\
 & 1 \boxtimes 4x y^3 & \psi_1 = (E_{31}^* \wedge E_{42}^* + E_{41}^* \wedge E_{32}^*) \otimes E_{43} + E_{31}^* \wedge E_{32}^* \otimes H_{43} \\
 & 1 \boxtimes y^4 & \psi_0 = E_{32}^* \wedge E_{31}^* \otimes E_{43}\\ \hline
  \multirow{5}{*}{ $\Athree{wxw}{4,-4,0}$ }& 
 x^4 \boxtimes 1& \phi_4 = E_{41}^* \wedge E_{31}^* \otimes E_{12} \\
 & 4 x^3 y \boxtimes 1 & \phi_3 = (E_{42}^* \wedge E_{31}^* + E_{41}^* \wedge E_{32}^*) \otimes E_{12} + E_{41}^* \wedge E_{31}^* \otimes H_{21}\\
 & 6 x^2 y^2 \boxtimes 1 & \phi_2 = E_{31}^* \wedge E_{41}^* \otimes E_{21} + E_{42}^* \wedge E_{32}^* \otimes E_{12} + (E_{31}^* \wedge E_{42}^* + E_{32}^* \wedge E_{41}^*)\otimes H_{12}\\
 & 4 x y^3 \boxtimes 1 & \phi_1 = (E_{31}^* \wedge E_{42}^* + E_{32}^* \wedge E_{41}^*) \otimes E_{21} + E_{32}^* \wedge E_{42}^* \otimes H_{12}\\
 & y^4 \boxtimes 1& \phi_0 = E_{32}^* \wedge E_{42}^* \otimes E_{21}\\ \hline
 \end{array}
 \]
 \end{footnotesize}
 \caption{Harmonic 2-cochains associated to the SD \& ASD Weyl curvature modules}
 \label{F:harmonic}
 \end{table}

 \subsection{Twistor XXO-structures as  correspondence spaces }\label{correspondence}


 The construction of twistor XXO-structures discussed in Section \ref{sec-construction} can be described as a Cartan-geometric correspondence space construction \cite{Cap2005}. This has the advantage that it allows us to locally characterize these structures in terms of a curvature condition.

 Given  nested parabolic subgroups $P\subset Q\subset G$, the correspondence space construction is a natural construction that assigns to any parabolic geometry  of type $(\fg,Q)$ a parabolic geometry of type   $(\fg,P)$. Starting with  $(\pi:\cG\to M,\omega)$ of type $(\fg,Q)$, one forms the orbit space
 \begin{align}
 \cC M:=\cG/P=\cG\times_Q(Q/P).
 \end{align}
 The natural projection $\cG\to\cC M$ is a $P$-principal bundle and  $(\cG\to \cC M,\omega)$ is  a Cartan geometry of type $(\fg,P)$ (referred to as the {\sl correspondence space}). If $Q/P$ is connected, then this construction defines an equivalence of categories between Cartan geometries of type $(\fg,Q)$ and a subcategory of Cartan geometries of type $(\fg,P)$, see \cite[Prop. 1.5.13]{CS2009}. If $(\cG\to M,\omega)$ is a regular, normal parabolic geometry, then  $(\cG\to \cC M,\omega)$ is automatically normal, but regularity is in general not preserved. 

Given a correspondence space $(\cG\to \cC M,\omega)$, the vertical bundle $V\cC M$ of  $\cC M\to M$ corresponds to $\mathfrak{q}/\fp\subset\fg/\fp$ and the curvature has the property that $\kappa(X,\cdot)=0$ for any $X\in\mathfrak{q}/\fp$.
 It can be shown that this curvature condition locally characterizes correspondence spaces. For regular, normal parabolic geometries, there is an even more efficient characterization in terms of harmonic curvature.  Theorems 2.7 and 3.3 in \cite{Cap2005} show that  a regular, normal parabolic geometry of type $(\fg,P)$ is locally isomorphic to the correspondence space of a normal parabolic geometry of type $(\fg,Q)$ if and only  if its harmonic curvature has the property that $\kappa_H(X,\cdot)=0$ for any $X\in\mathfrak{q}$.

 Here we apply these results to  the nested parabolic subgroups $\widetilde{P}\subset\widetilde{Q}\subset\mathrm{SL}(4,\bbR)$ defined in Section \ref{S:cat}.
The first observation is the following:
\begin{lemma}\label{lemm-cor}
The twistor XXO-structure on the bundle of SD totally null 2-planes over a 4D split-conformal structure can be naturally identified with the XXO-structure on the correspondence space $\cC M=\cG\times_{\widetilde{Q}}\widetilde{Q}/\widetilde{P}$.
\end{lemma} 

\begin{proof}
Let $(\cG\to M,\omega)$ be a normal parabolic geometry of type $(\mathrm{SL}(4,\bbR),\widetilde{Q})$ with associated oriented split-conformal structure $(M,[\sfg])$. Recall the identification of  $\mathfrak{sl}(4,\bbR)/\widetilde{\mathfrak{q}}$ with the space of  $2\times 2$ matrices $\mathrm{M}_2\bbR$, embedded as the bottom-left $2\times 2$ block in $\mathfrak{sl}(4,\bbR)$. The $\widetilde{Q}$-action on this space, as discussed in  Section \ref{S:cat}, defines a surjection from $\widetilde{Q}$ onto the connected component of the identity of the group $\mathrm{CSO}_0(2,2)$ of linear conformal transformations of $\mathfrak{sl}(4,\bbR)/\widetilde{\mathfrak{q}}$. This action preserves the two distinguished sets of SD/ASD  totally null $2$-planes in $\mathfrak{sl}(4,\bbR)/\widetilde{\mathfrak{q}}$ and is transitive on them.
 It is easy to see that the set of maps in $\mathrm{M}_2\bbR$ that vanish on the line $\left\langle e_1\right\rangle$ through the first basis vector in the standard representation of $\mathfrak{sl}(4,\bbR)$ defines such a totally null $2$-plane and that 
   its stabilizer  in $\widetilde{Q}$ is exactly the subgroup $\widetilde{P}\subset \widetilde{Q}$.
   
   This provides the identification $\Phi:\cC M=\cG\times_{\widetilde{Q}}\widetilde{Q}/\widetilde{P}\cong\bbT^+(M)$. Since the normal conformal geometry $(\cG\to M,\omega)$ is torsion-free, the geometry $(\cG\to \cC M,\omega)$ is regular and thus has an induced XXO-structure. Moreover, via the Cartan connection, $T\cC M\cong \cG\times_{\widetilde{P}}(\wfg/\widetilde{\fp})$ and the map $\pi_*:T\cC M\to TM$ corresponds to the projection $\wfg/\widetilde{\fp}\to\wfg/\widetilde{\mathfrak{q}}$. Using this, it is straightforward to see that $\Phi$ is indeed an isomorphism of XXO-structures.
\end{proof}

The general results on correspondence spaces then immediately lead to the following:

 \begin{prop} \label{P:T}
 An XXO-structure is locally isomorphic to the twistor XXO-structure on the bundle of SD totally null 2-planes over a 4D split-conformal structure  if and only if it has vanishing harmonic curvature component $\mathcal{T}$.
 \end{prop}
 
 \begin{proof}
 By Lemma \ref{lemm-cor}, the line bundle $\ell$ of the XXO-structure corresponds to the vertical bundle $V\cC M$ of  $\cC M\to M$.  From Table \ref{F:KH}, we see that the harmonic curvature of an XXO-structure vanishes upon insertion of any element of  $\ell$ if and only if $\cT \equiv 0$.  Thus the result is an immediate consequence of  Theorems 2.7 and 3.3 in \cite{Cap2005}.
\end{proof}

\begin{remark}
Having identified the twistor XXO-structure as the underlying structure of a correspondence space, this also allows us to relate the components of the harmonic curvatures $\kappa_H^{\cC M}$ and $\kappa_H^M$ (see Table \ref{F:KH}) via
 \begin{align}
 \kappa_H^{\cC M}=\mathrm{pr}\circ\kappa_H^M: \cG\to H_2(\wfq_+,\wfg)/((\wfp_+\cap \wfq_0)\cdot H_2(\wfq_+,\wfg))  \,\subset\, H_2(\wfp_+,\wfg).
 \end{align}
 In this way, we recover the relationship between the Weyl tensor $\cW^+$ and the harmonic curvature component $\cI$  via parabolic geometry machinery, which was obtained by direct computation in Lemma \ref{lemm-Weyl}. Similarly, $\cW^-$ can be related to $\cS$.
\end{remark} 


 \subsection{Symmetries and gaps} 
 

 The (infinitesimal) symmetry algebra of each of the three structures of interest is formulated naturally via the Lie derivative, e.g.\ for an XXO-structure $(\ell,\cD)$ on a 5-manifold $N$, a symmetry is a vector field $\bX \in \fX(N)$ such that $\cL_\bX \ell \subset \ell$ and $\cL_\bX \cD \subset \cD$.  By the aforementioned equivalence of categories statements, we may equivalently study the symmetry of the corresponding Cartan geometry.
 
 Given any Cartan geometry $(\cG \to M, \omega)$ of type $(\fg,P)$, its symmetry algebra is 
 \begin{align}
 \mathfrak{inf}(\cG,\omega) := \{ \xi \in \fX(\cG)^P : \cL_\xi \omega = 0 \}.
 \end{align}
 Since $\omega$ is in particular a coframing, then $\dim(\mathfrak{inf}(\cG,\omega)) \leq \dim(\cG) = \dim(\fg)$, with equality realized on the homogeneous model $(G \to G/P, \omega_G)$.  For parabolic geometries, equality is {\em locally uniquely} realized by this model.  The {\sl submaximal} symmetry dimension $\fS$ is the next realizable symmetry dimension below $\dim(\fg)$, among all regular normal parabolic geometries of type $(\fg,P)$, and often there is a significant symmetry gap \cite{KT2014}.
 
 Cartan \cite{Car1910} showed that $\fS = 7$ for $(2,3,5)$-distributions.  For 4D split-conformal structures and XXO-structures, we have $\fS = 9$.  (See \cite[Table 12]{KT2014} for $G = A_3$ and $P = P_1$ or $P_{1,2}$ respectively.)  In particular, we have the following:
 
 \begin{prop}
 The symmetry algebra of a non-integrable XXO-structure has dimension $\leq 9$.
 \end{prop}

 \begin{remark}
Recall that symmetries of $[\sfg]$ lift to symmetries of the twistor distribution $\cD$.
Together with the fact that the submaximal symmetry dimension of a $(2,3,5)$ distribution is $7$, this immediately implies that the twistor distribution $\cD$ of a conformal structure $[\sfg]$ whose conformal symmetry algebra is at least 8-dimensional is  necessarily $\mathrm{G}_2$-symmetric.
\end{remark}
 
 \section{Abstract classification}

 
 \subsection{Classification problem for Lie-theoretic models}  
 \label{S:LieTh}
 

 A locally homogeneous XXO-structure is completely encoded by a {\sl Lie-theoretic} model $(\ff,\ff^0;\ff_\ell,\ff_\cD)$: namely, its symmetry algebra $\ff$ and isotropy subalgebra $\ff^0 \subset \ff$ at a generic basepoint $o \in N$, and $\ff^0$-invariant subspaces $\ff_\ell,\ff_\cD\subset \ff/\ff^0$ corresponding to $\ell,\mathcal{D}\subset TN$ at $o$.

 Now assume that $\cD$ is $(2,3,5)$ and has vanishing Cartan tensor $\cQ \equiv 0$, so that $\cD$ has $\fg = \Lie(G_2)$ symmetry.  Let $(\cG \stackrel{\pi}{\to} N, \omega)$ be the parabolic geometry of type $(\fg = \Lie(G_2),P)$ associated to $\cD$.  Since $\cQ \equiv 0$, then any $u \in\pi^{-1}(o) \subset \cG$, $\omega_u$ restricts to a Lie algebra injection $\ff \inj \fg$.  From Section \ref{S:cat}, recall that $\fg$ admits the (depth 3) {\sl $(2,3,5)$-filtration} $\fg = \fg^{-3} \supset ... \supset \fg^3$ with stabilizer $P \subset \Aut(\fg)$, and so $\ff$ inherits a filtration of the form
 \begin{align}
 \ff = \ff^{-3} \supset \ff^{-2} \supset \ff^{-1} \supset \ff^0 \supset \ff^1 \supset \ff^2 \supset \ff^3.
 \end{align}
 The $\ff^0$ here agrees with that above, being the isotropy subalgebra at $o$.  Moreover, 
 \begin{align}
 \ff_\cD=\ff^{-1} / \ff^0 \cong\cD_o, \quad\mbox{and}\quad  \ff_{\ell} \op \ff_{\cD}=\ff^{-2} / \ff^0  \cong [\cD,\cD]_o.
 \end{align}
 {\em Henceforth, we will abuse notation and write simply $(\ell,\cD)$ in the abstract setting in place of $(\ff_\ell,\ff_\cD)$.}
 
 Define the associated-graded Lie algebra $\fs = \tgr(\ff)$, i.e.\ $\fs_i = \ff^i / \ff^{i+1}$.  Since the filtration on $\fg$ comes from a grading $\fg = \fg_{-3} \op ... \op \fg_3$, then we can identify $\fs$ with a graded Lie subalgebra of $\fg$.  Transitivity of $\ff$ implies $\tgr_-(\ff) = \fg_-$, and if we assume that $\ff$ is multiply-transitive, then $\dim(\ff^0) \geq 1$.  The condition that the harmonic curvature component $\cT$ vanish can be understood as a condition on the data $(\ff, \ff^0; \ell, \cD)$.   (See for instance Example \ref{X:M6N-torsion} in Appendix \ref{S:torsion}.)

 \begin{defn} We say that the data $(\ff, \ff^0; \ell, \cD)$ is an {\sl admissible} (Lie-theoretic) model if:
 \begin{enumerate}
 \item[(X.1)] $\ff \inj \fg$ is a filtered Lie subalgebra with $\fg$ equipped with the $(2,3,5)$-filtration and $\ff$ equipped with the induced filtration such that $\ff^i=\fg^i\cap\ff$;
 \item[(X.2)] $\tgr_-(\ff) = \fg_-$ and $\dim(\ff^0) \geq 1$;
 \item[(X.3)] $ \cD = \ff^{-1} / \ff^0$, and $\ell \subset \ff^{-2} / \ff^0$ is an $\ff^0$-invariant line such that $\ell \cap \cD = 0$;
 \item[(X.4)]  $\cT \equiv 0$;
 \item[(X.5)]  it is a {\em maximal element} (among the set of those satisfying (X.1)-(X.4)) with respect to the natural partial ordering.
 \end{enumerate}
 \end{defn}
 
 For (X.5), the partial ordering is given by declaring $(\ff, \ff^0; \ell, \cD) \leq (\widetilde\ff, \widetilde\ff^0; \widetilde\ell, \widetilde\cD)$ if there is an {\sl embedding} of filtered Lie algebras $\ff \inj \widetilde\ff$ inducing an isomorphism $\ff^{-2} / \ff^0\cong\widetilde{\ff}^{-2} / \widetilde{\ff}^0$ mapping $\ell$ to $\widetilde\ell$. The parabolic subgroup $P$ naturally acts on admissible $(\ff,\ff^0;\ell,\cD)$ via the adjoint action, e.g.\ $\ff \mapsto \Ad_p \ff$, $\forall p \in P$, so we view $P$ as the (initial) {\sl structure group} for our problem:

 \begin{framed}
 {\bf Classify admissible $(\ff,\ff^0;\ell, \cD)$ up to the $P$-action.}
 \end{framed}
 
 Here, $P$ will be used to bring $\ff$ into specific (canonical) forms.  In doing so, it will be successively reduced, so we will refer to these (reduced) subgroups as {\sl residual structure groups}.   We say that the filtered Lie subalgebra $\ff \subset \fg$ is a {\sl filtered deformation} of $\fs = \tgr(\ff) \subset \fg$.  It will be convenient to use the following notation.  Suppose that $x \in \fg^i \backslash \fg^{i+1}$, so $x = x_i + x_{i+1} + ...$ with $x_i \neq 0$ and $x_j \in \fg_j$ for all $j \geq i$.  For $x$, we write $\tgr_i(x) = x_i$ for its {\sl leading part}, while the terms of higher homogeneity $x_{i+1} + ...$ will be referred to as its {\sl tails}.
 
 \subsection{The simple Lie algebra $\fg = \Lie(G_2)$}
 \label{S:G2}


Let us introduce a convenient basis of $\fg = \Lie(G_2)$ (adapted to the root space decomposition), pictured on the $G_2$ root diagram in Figure \ref{F:G2}.

\begin{center} 
\begin{figure}[h]
 \begin{tikzpicture}[scale=1]
\foreach \i in {-3,...,3}
	\draw[\graycolor, xshift=3*\i mm, yshift=4*\i mm]  (150:2) to (150:-2.5) node[below right]{$\mathfrak{g}_{\i}$};
\foreach \i in {180,240,300}
	\draw[->] (0,0) to (\i:1);
\foreach \i in {210, 270}
	\draw[->] (0,0) to (\i:{sqrt(3)});
\foreach \i in {0,60,120}
	\draw[->, red] (0,0) to (\i:1);
\foreach \i in {-30,30,...,150}
	\draw[->, red] (0,0) to (\i:{sqrt(3)});
\fill[red] (0,0) circle (0.1);
   \node at (-0.3,0.4) {$\sfZ_1,$};
    \node at (0.3,0.4) {$\sfZ_2$};
    \node at (1.3,0) {$e_{10}$};
    \node at (-1.3,0) {$f_{10}$};
    \node at (-1.8,1.155) {$e_{01}$};
    \node at (1.8,-1.155) {$f_{01}$};
    \node at (-1.8,-1.155) {$f_{31}$};
    \node at (1.8,1.155) {$e_{31}$};
    \node at (0,-2.1) {$f_{32}$};
    \node at (0,2.1) {$e_{32}$};
    \node at (0.666,1.155) {$e_{21}$};
    \node at (-0.666,-1.155) {$f_{21}$};
    \node at (-0.666,1.155) {$e_{11}$};
    \node at (0.666,-1.155) {$f_{11}$};

 \end{tikzpicture} 
 \caption{A basis of $\fg = \Lie(G_2)$ adapted to the root space decomposition}
 \label{F:G2}
 \end{figure}
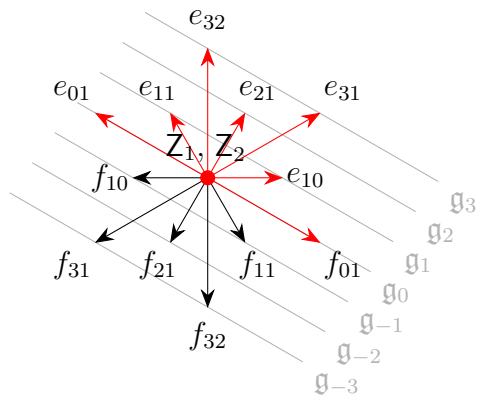
 \end{center}
 
 The generators $\sfZ_i, f_{ij}, e_{ij}$ in Figure \ref{F:G2} are obtained by differentiation with respect to $z_i, a_{ij}, b_{ij}$ of the following $7 \times 7$ matrix realization of $\fg$, as presented in \cite{The2022}.
 \begin{align} \label{E:g2rep}
 \begin{footnotesize}
 \begin{array}{c}
 \begin{pmatrix}
 2z_1 + z_2 & b_{10} & b_{11} & \sqrt{2} \,b_{21} & b_{31} & b_{32} & 0\\
 a_{10} & z_1 + z_2 & b_{01} & \sqrt{2}\, b_{11} & -b_{21} & 0 & -b_{32}\\
 a_{11} & a_{01} & z_1 & -\sqrt{2}\, b_{10} & 0 & b_{21} & -b_{31}\\
 \sqrt{2} \, a_{21} & \sqrt{2}\, a_{11} & -\sqrt{2}\, a_{10} & 0 & \sqrt{2}\, b_{10} & -\sqrt{2}\, b_{11} & -\sqrt{2}\,b_{21}\\
 a_{31} & -a_{21} & 0 & \sqrt{2}\, a_{10} & -z_1 & -b_{01} & -b_{11}\\
 a_{32} & 0 & a_{21} & -\sqrt{2}\, a_{11} & -a_{01} & -z_1 - z_2 & -b_{10}\\
 0 & -a_{32} & -a_{31} & -\sqrt{2}\, a_{21} & -a_{11} & -a_{10} & -2z_1 - z_2
 \end{pmatrix}
 \end{array}.
 \end{footnotesize}
 \end{align}
For $i,j \geq 0$, we have root vectors $e_{ij}$ ({\sl raising operators}) and $f_{ij}$ ({\sl lowering operators}) lying in the root spaces for $i \alpha_1 + j\alpha_2$ and $-i\alpha_1 - j\alpha_2$ respectively (where $\alpha_1$ and $\alpha_2$ are the simple roots).   The {\sl grading element} $\sfZ_1$ gives rise to the grading of $\fg$ indicated in Figure \ref{F:G2}.  Namely, $\fg_i$ is the $\ad_{\sfZ_1}$-eigenspace with eigenvalue $i$, and $[\sfZ_1,e_{ij}] = i e_{ij}$ while $[\sfZ_1,f_{ij}] = -i f_{ij}$.  The full commutator table for $\fg$ is given in Table \ref{F:G2br}.
 
 \begin{footnotesize}
 \begin{table}[h]
 \[
 \begin{array}{c|cccccccccccccc}
 [\cdot,\cdot] & \sfZ_1 & \sfZ_2 & e_{01} & e_{10} & e_{11} & e_{21} & e_{31} & e_{32} & f_{01} & f_{10} & f_{11} & f_{21} & f_{31} & f_{32} \\ \hline
 \sfZ_1 & \cdot & \cdot & \cdot & e_{10} & e_{11} & 2 e_{21} & 3 e_{31} & 3 e_{32} & \cdot & -f_{10} & -f_{11} & -2 f_{21} & -3 f_{31} & -3 f_{32} \\
 \sfZ_2 & & \cdot & e_{01} & \cdot & e_{11} & e_{21} & e_{31} & 2 e_{32} & -f_{01} & \cdot & -f_{11} & -f_{21} & -f_{31} & -2 f_{32}\\
 e_{01} &&& \cdot & -e_{11} & \cdot & \cdot & e_{32} & \cdot & -\sfZ_1 + 2\sfZ_2 & \cdot & f_{10} & \cdot & \cdot & -f_{31}\\
 e_{10} &&&&\cdot & 2 e_{21} & -3 e_{31} & \cdot & \cdot & \cdot & 2\sfZ_1 - 3\sfZ_2 & -3 f_{01} & -2 f_{11} & f_{21} & \cdot\\
 e_{11} &&&&&\cdot & 3 e_{32} & \cdot & \cdot & e_{10} & -3 e_{01} & -\sfZ_1 + 3\sfZ_2 & 2 f_{10} & \cdot & -f_{21}\\
 e_{21} &&&&&&\cdot & \cdot & \cdot & \cdot & -2 e_{11} & 2 e_{10} & \sfZ_1 & -f_{10} & f_{11}\\
 e_{31} &&&&&&&\cdot & \cdot & \cdot & e_{21} & \cdot & -e_{10} & \sfZ_1 - \sfZ_2 & f_{01}\\
 e_{32} &&&&&&&&\cdot & -e_{31} & \cdot & -e_{21} & e_{11} & e_{01} & \sfZ_2\\
 f_{01} &&&&&&&&&\cdot & f_{11} & \cdot & \cdot & -f_{32} & \cdot\\
 f_{10} &&&&&&&&&&\cdot & -2 f_{21} & 3 f_{31} & \cdot & \cdot\\
 f_{11} &&&&&&&&&&&\cdot & -3f_{32} & \cdot & \cdot\\
 f_{21} &&&&&&&&&&&&\cdot & \cdot & \cdot\\
 f_{31} &&&&&&&&&&&&&\cdot & \cdot\\
 f_{32} &&&&&&&&&&&&&&\cdot \\
 \end{array}
 \]
 \caption{Bracket relations for $\fg = \Lie(G_2)$}
 \label{F:G2br}
 \end{table}
 \end{footnotesize}


 \subsection{Isotropy constraints} 

   
 \begin{lemma} Let $(\ff,\ff^0;\ell,\cD)$ be an admissible model.  Then $\ff^1 = 0$, $\dim(\ff^0) \leq 4$, and $\dim(\ff)\leq 9$.
 \end{lemma}
 
 \begin{proof} Let $X \in \ff^1 \subseteq \fg^1 = \langle e_{10}, e_{11}, e_{21}, e_{31}, e_{32} \rangle$, so $X_1 = \tgr_1(X) = a e_{10} + b e_{11}$.  Let $L \in \ell$ with $f_{21} = \tgr_{-2}(L)$. But then 
 \begin{align}
 \tgr([X,L]) = [\tgr_1(X),\tgr_{-2}(L)] = [X_1,f_{21}] = -2af_{11} + 2b f_{10}.
 \end{align}
 We know that $\ell$ must be invariant under $X \in \ff^1 \subseteq \ff^0$, so this forces $a=b=0$, i.e.\ $X \in \ff^2$.  Thus, $\fs = \tgr(\ff)$ has $\fs_1 = 0$.  Since the Lie bracket induces $\fg_0$-module isomorphisms $\fg_2 \times \fg_{-1} \to \fg_1$ and $\fg_3 \times \fg_{-2} \to \fg_1$, then from $\fs_- = \fg_-$ (by homogeneity) and $\fs_1 = 0$, we conclude that $\fs_2 = 0$ and $\fs_3 = 0$.  Hence, $\ff^1 = 0$ and since $\fs_0 = \tgr(\ff^0) \subseteq \fg_0 \cong \fgl_2$, then $\dim(\ff^0) \leq 4$.
 \end{proof}

 In Section \ref{S:CXXO}, we establish the complete classification of complex admissible $(\ff,\ff^0;\ell,\cD)$, while Section \ref{S:RXXO} will treat real forms. 
 
 
 \subsection{Complex classification}
 \label{S:CXXO}
 
  
 \begin{theorem}
 The complete classification of complex admissible $(\ff,\ff^0;\ell,\cD)$ is given in Table \ref{F:XXO}.  Moreover, the models $\mathsf{M7}_\sfa$ and $\mathsf{M7}_\sfb$ are isomorphic iff $\sfb^2 = \sfa^2$.
 \end{theorem}
 
 \begin{table}[h]
 \[
 \begin{array}{|c|c|l|l|l|} \hline
 \mbox{Label} & \ff^0 & \ff^{-1} / \ff^0 & \ff^{-2} / \ff^{-1} & \ff^{-3} / \ff^{-2} \\ \hline\hline
 \mathsf{M9} & \begin{array}{c}
  e_{01}\\
  \sfZ_1, \sfZ_2\\
  f_{01}
 \end{array} &
 \begin{array}{l@{\,}l}
 X_1 &= f_{10} \\
 X_2 &= f_{11}
 \end{array} & 
 \begin{array}{l@{\,}l}
 L = X_3 &= f_{21} \\
 \end{array} &  \begin{array}{l@{\,}l}
 X_4 &= f_{31} \\
 X_5 &= f_{32}
 \end{array} \\ \hline
 \mathsf{M8} &
 \begin{array}{c}
  e_{01}\\
 \sfH\\
  f_{01}
 \end{array} & 
 \begin{array}{l@{\,}l}
 X_1 &= f_{10} + e_{32}\\
 X_2 &= f_{11} + e_{31}
 \end{array} & 
 \begin{array}{l@{\,}l}
 L = X_3 &= f_{21} + e_{21}
 \end{array} &
 \begin{array}{l@{\,}l}
 X_4 &= f_{31} + e_{11} \\
 X_5 &= f_{32} + e_{10} 
 \end{array} \\ \hline
 \mathsf{M7}_\sfa &
 \begin{array}{c}
 \sfZ_2\\
  f_{01}
 \end{array} & 
 \begin{array}{l@{\,}l}
 X_1 &= f_{10} + a \sfZ_1 + e_{10} \\
 X_2 &= f_{11}
 \end{array} & 
 \begin{array}{l@{\,}l}
 L = X_3 &= f_{21}
 \end{array} &
 \begin{array}{l@{\,}l}
 X_4 &= f_{31}\\
 X_5 &= f_{32}
 \end{array} \\ \hline
 \mathsf{M6S} & \sfH & 
 \begin{array}{l@{\,}l}
 X_1 &= f_{10} + e_{11} \\
 X_2 &= f_{11} + e_{10}
 \end{array} &
 \begin{array}{l@{\,}l}
 L = X_3 &= f_{21} + e_{21}
 \end{array} &
 \begin{array}{l@{\,}l}
 X_4 &= f_{31} + e_{32} \\
 X_5 &= f_{32} + e_{31}
 \end{array} \\ \hline
 \mathsf{M6N} & f_{01} & 
 \begin{array}{l@{\,}l}
 X_1 &= f_{10} + 2\sfH + e_{01} + e_{32} \\
 X_2 &= f_{11} - \sfH + e_{31}
 \end{array} &
 \begin{array}{l@{\,}l}
 X_3 &= f_{21} + 3\sfH + e_{21}\\
 L &= X_3 + 3 X_2
 \end{array} &
 \begin{array}{l@{\,}l}
 X_4 &= f_{31} + 2\sfH - e_{01} + e_{11} \\
 X_5 &= f_{32} - \sfH + e_{10}
 \end{array} \\\hline
 \end{array}
 \]
 \caption{Complete classification of complex admissible $(\ff,\ff^0;\ell,\cD)$.  Here, $\ell = \langle L \rangle \mod \ff^0$ and $\cD = \langle X_1,X_2 \rangle \mod \ff^0$.  (Notation: $\sfH = [e_{01},f_{01}] = -\sfZ_1 + 2\sfZ_2$.)}
  \label{F:XXO}
 \end{table}


 \subsubsection{Classification strategy}
 
 
 Focus first on $\fs = \tgr(\ff)$.  Since $\fs_- = \fg_-$, it suffices to determine $\fs_0 = \tgr(\ff^0) \subseteq \fg_0$, where $\fg_0 \cong \fgl_2 \cong \bbC \op \fsl_2$.  The isomorphism $\fg_0 \cong \fgl_2$ is given by $x \mapsto \ad_x|_{\fg_{-1}}$, expressed as a $2\times 2$ matrix in the basis $\{ f_{10}, f_{11} \}$.  In particular, this identifies $\sfZ_1, f_{01}, \sfH, e_{01}$ respectively with
 \begin{align}
 \begin{pmatrix}
 -1 & 0\\
 0 & -1
 \end{pmatrix}, \quad
 \begin{pmatrix}
 0 & 0\\
 1 & 0
 \end{pmatrix}, \quad
 \begin{pmatrix}
 1 & 0\\
 0 & -1
 \end{pmatrix}, \quad 
 \begin{pmatrix}
 0 & 1\\
 0 & 0
 \end{pmatrix},
 \end{align} 
 where 
 \begin{align}
 \sfH := [e_{01},f_{01}] = -\sfZ_1 + 2\sfZ_2.
 \end{align}
 
 The $P$-action induces an action by $G_0\cong \GL_2$ on $\fg_0$, and up to $\GL_2$-conjugacy, the (non-zero) subalgebras of $\fgl_2$ are given in Table \ref{F:gl2-subalg}.  (Here, $*$ denotes arbitrary values, while $\lambda, \lambda_i \in \bbC$ are fixed.)   
   
 \begin{table}[h]
 \[
 \begin{array}{|c|cc|} \hline
 \dim & \multicolumn{2}{|c|}{\mbox{Subalgebras}}\\ \hline\hline
 4 & \fgl_2 & \\
 3 & \fsl_2, & \left\langle \begin{pmatrix}
 * & 0\\
 * & *
 \end{pmatrix} \right\rangle\\
 2 & \left\langle \begin{pmatrix}
 * & 0\\
 0 & *
 \end{pmatrix} \right\rangle, & \left\langle \begin{pmatrix} \lambda_1 & 0\\ * & \lambda_2 \end{pmatrix} \right\rangle\\
 1 & \left\langle \begin{pmatrix}
 \lambda_1 & 0\\ 0 & \lambda_2
 \end{pmatrix} \right\rangle, & 
 \left\langle \begin{pmatrix}
 \lambda & 0\\ 1 & \lambda
 \end{pmatrix} \right\rangle\\ \hline
 \end{array}
 \]
 \caption{Classification over $\bbC$ of subalgebras of $\fgl_2$, up to conjugacy}
 \label{F:gl2-subalg}
 \end{table}
 
 Almost all such subalgebras contain a {\em semisimple} (diagonalizable) element.  Given $S_0 \in 
\fs_0 = \tgr_0(\ff^0)$ (nonzero) semisimple, we use the $P$-action to normalize the tails of $S = S_0 + ... \in \ff^0$ as much as possible.  It turns out that this is always possible (Lemma \ref{L:ss}), so for purposes of giving an outline let us assume henceforth that $S = S_0 \in \ff^0$.  The existence of such a semisimple element strongly restricts the filtered deformations of $\fs$ that can arise.  Namely, choose an $S$-invariant graded subspace $\fs^\perp \subset \fg$ complementary to $\fs \subset \fg$.  Then $\ff$ is spanned by $x + \fd(x) \in \ff^i$, where $x \in \fs_i$, with {\sl tail} $\fd(x) \in \bigoplus_{k > 0} (\fs^\perp)_{i+k}$, i.e.\ this unique {\sl deformation map} $\fd$ is of positive homogeneity: 
 \begin{align} \label{E:Shom}
 \fd \in (\fs^* \otimes \fs^\perp)_+.
 \end{align}
 For $x \in \fs_i$, $x + \fd(x) \in \ff^i$, and $[S,x+\fd(x)] = [S,x] + [S,\fd(x)] \in \ff^i$, where $[S,x] \in \fs_i$ and $[S,\fd(x)] \in \fs^\perp \cap \fg^{i+1}$.  By uniqueness of $\fd$, we have $[S,\fd(x)] = \fd([S,x])$, so $\fd$ is $S$-annihilated:
 \begin{align} \label{E:Sann}
 S \cdot \fd = 0.
 \end{align}
 The two constraints \eqref{E:Shom} and \eqref{E:Sann} give a priori restrictions on the admissible filtered deformations, which will be efficiently obtained via eigenvalue (weight) considerations. 
 
 The cases not containing a semisimple element are the (1-dimensional) Jordan cases $\left\langle \begin{pmatrix} \lambda & 0\\ 1 & \lambda \end{pmatrix} \right\rangle$, so we may assume $\fs_0 = \langle f_{01} + r \sfZ_1 \rangle$.  The $r \neq 0$ yields no new models (beyond those found when assuming existence of a semisimple element), while the $r=0$ (i.e.\ nilpotent) case is the most complicated case, and leads to one more model.  Here is a (complex) classification summary:
 \begin{align}
 \begin{array}{|c|c|c|} \hline
 \fs_0\mbox{-element} & \mbox{Constraint} & \mbox{Models}\\ \hline\hline
 \sfZ_1 & - & \mathsf{M9}\\ \hline
 \sfH & \sfZ_1 \not\in \fs_0 & \mathsf{M8},\,\,\mathsf{M6S}\\ \hline
 \begin{array}{c} \sfZ_1 + c \sfH \\ (c \neq 0)\end{array} & 
 \fs_0 \subseteq \langle\sfZ_1 + c \sfH, f_{01} \rangle
 & \mathsf{M7} \mbox{ when $c=1$}\\ \hline
 f_{01} + r \sfZ_1 & \dim\,\fs_0 = 1 & \mathsf{M6N} \mbox{ when $r=0$}\\ \hline
 \end{array}
 \end{align} 


 \subsubsection{Semisimple cases}
 \label{S:SS}
 

 \begin{lemma} \label{L:ss}
 Suppose that one of the following hold:
 \begin{enumerate}
 \item $S_0 = \sfZ_1 \in \fs_0$;
 \item $\sfZ_1 \not\in \fs_0$ and $S_0 = \sfH \in \fs_0$;
 \item $\sfZ_1 \not\in \fs_0$ and $S_0 = \sfZ_1 + c \sfH \in \fs_0$.
 \end{enumerate}
 Then, normalizing by the $P$-action, we may assume that $S = S_0 \in \ff^0$.
 \end{lemma}
 
 \begin{proof} Let $S_0 = \sfZ_1 + c\sfH \in \fs_0$.  Then $\{ e_{10}, e_{11}, e_{21}, e_{31}, e_{32} \}$ is an $\ad_{S_0}|_{\fg^1}$-eigenbasis with eigenvalues 
 \begin{align}
 (\lambda_1, \lambda_2, \lambda_3, \lambda_4, \lambda_5) = (1 - c, \quad 1+c, \quad 2, \quad 3 - c, \quad 3 + c).
 \end{align} 
 Suppose $c \not\in \{ \pm 1, \pm 3 \}$, so all $\lambda_i \neq 0$.  Let $S = S_0 + s_1 e_{10} + s_2 e_{11} + s_3 e_{21} + s_4 e_{31} + s_5 e_{32} \in \ff^0$.  Given $x \in \fg^1$, we have $ \Ad_{\exp(x)}(S) = \exp(\ad_x)(S) = S + [x,S] + \frac{1}{2} [x,[x,S]] + ...$  Restrict to $x = x_1 e_{10} + x_2 e_{11} \in \fg_1$:
 \begin{align}
 \widetilde{S} = \Ad_{\exp(x)}(S) &= S_0 + [x,S_0] + ... = S_0 + \widetilde{s_1} e_{10} + \widetilde{s_2} e_{11} + ...
 \end{align}
 where $(\widetilde{s_1}, \widetilde{s_2}) = (s_1 - x_1 \lambda_1, s_2 - x_2 \lambda_2)$, and the final ellipsis denotes higher homogeneity terms. Choosing $(x_1,x_2) = (\frac{s_1}{\lambda_1}, \frac{s_2}{\lambda_2})$ yields $\widetilde{s_1} = \widetilde{s_2} = 0$ for $\widetilde{S} \in \widetilde\ff^0 \subset \widetilde\ff$.  Dropping tildes, we have $s_1 = s_2 = 0$.  Similarly using $\fg_2$ and $\fg_3$, we normalize $s_3 = s_4 = s_5 = 0$.  This handles (1) and most of (3).
 
  Consider the remaining part of (3), i.e.\ $S_0 = \sfZ_1 + c \sfH \in \fs_0$ with $c \in \{ \pm 1, \pm 3 \}$, with the additional hypothesis $\sfZ_1 \not\in \fs_0$.  From Table \ref{F:gl2-subalg}, we may assume (via the action of $G_0 \subset P$) that $\fs_0 \subseteq \langle \sfZ_1 + c\sfH, f_{01} \rangle$.  As above, we normalize using $P_+$, but there is a residual tail term for $S$ (depending on $s$):
  \begin{align}
  \begin{array}{|c||c|c|c|c|} \hline
  c & -1 & 1 & -3 & 3\\ \hline
  S \in \ff^0 & \sfZ_1 - \sfZ_2 + s e_{11} & \sfZ_2 + s e_{10} & 2\sfZ_1 - 3\sfZ_2 + s e_{32} & -\sfZ_1 + 3\sfZ_2 + s e_{31}\\ \hline
  \end{array}
  \end{align}
    In fact, $s=0$ is forced in each case.  For example when $c = -3$, $\exists X_5 \in \ff^{-3}$ with $X_5 \equiv f_{32} \mod \fp$.  (This is possible since $f_{32} \in \fs_{-3}$ and tails of $X_5$ along $f_{21} \in \fs_{-2}$ and $f_{11}, f_{10} \in \fs_{-1}$ can be eliminated.)
 Then $\ff^0\ni [S,X_5] \equiv s[e_{32},f_{32}] \equiv s\sfZ_2 \,\,\mod \langle f_{01},e_{01} \rangle \op \fg^1$.  But $\fs_0 \subset \langle 2\sfZ_1-3\sfZ_2, f_{01} \rangle$, so $s=0$. (The other cases are similar, using $[e_{11}, f_{11}] = -\sfZ_1 + 3\sfZ_2$, $[e_{10},f_{10}] = 2\sfZ_1 - 3\sfZ_2$, $[e_{31}, f_{31}] = \sfZ_1 - \sfZ_2$.)
    
  Finally, assume (2), i.e. $\sfZ_1 \not\in \fs_0$, but $\sfH \in \fs_0$.  Then $\exists S \in \ff^0$ with $S \equiv \sfH \mod \fg^1$.  Since $\ad_\sfH|_{\fg^1}$ has zero-eigenspace $\fg_2 = \langle e_{21} \rangle$, then we normalize $S = \sfH + s\, e_{21} \in \ff^0$.  We know $\exists X_3 \in \ff^{-2}$ with $X_3 \equiv f_{21} \mod \fp$, and $\ff^0 \ni [S,X_3] \equiv s[e_{21}, f_{21}] \equiv s\sfZ_1 \,\, \mod \langle f_{01}, e_{01} \rangle \op \fg^1$, so $s=0$.
 \end{proof}
 
 \begin{lemma} \label{L:E}
 Suppose $0 \neq S := c_1 \sfZ_1 + c_2 \sfH \in \ff^0$.  If $c_2 \neq \pm c_1$, then the $\ff^0$-invariant line field $\ell = \langle L  \rangle\mod \ff^0 \subset \ff^{-2} / \ff^0$ has the form $L\equiv f_{21} \mod \fp$.
 \end{lemma}
 
 \begin{proof} The $\ad_S$-eigenvalues of $f_{21}, f_{11}, f_{10}$ are $-2c_1, -c_1 - c_2, -c_1 + c_2$.  By hypothesis, the last two are distinct from the first, so $\ff^0$-invariance of $\ell$ implies the result.
 \end{proof}
 

 \subsubsection{\framebox{$\sfZ_1 \in \fs_0$}}  

 
 By Lemma \ref{L:ss}, assume $S := \sfZ_1 \in \ff^0$.  Take $\fs^\perp = \fg_+$.  then $\fd = 0$ by \eqref{E:Shom} and \eqref{E:Sann}.  Thus, $\ff = \langle f_{10}, f_{11}, f_{21}, f_{31}, f_{32} \rangle \op \ff^0$. Note that $f_{21}$ spans the $-2$ eigenspace of $\ad_{\sfZ_1}|_\fg$, so $\ell = \langle f_{21} \rangle \mod \ff^0$. 
 Let $U = U_0 + U_+ \in \ff^0$, where $U_0 \in \fg_0$ and $U_+ \in \fg_+$.  Then $[\sfZ_1, U] = [\sfZ_1,U_+] \in \ff^1 = 0$.  But $\ker(\ad_{\sfZ_1}|_{\fg_+}) = 0$, so $U_+ = 0$. Thus, $\ff^0 \inj \widetilde\ff^0 := \fg_0 = \fgl_2$, which induces $\ff \inj \widetilde\ff$ and $\ell \inj \widetilde{\ell} := \langle f_{21} \mod \widetilde\ff^0 \rangle$.  Since $\dim(\widetilde\ff) = 9$, we label the tilded structure as $\mathsf{M9}$.
 
 
 \subsubsection{\framebox{$\sfZ_1 \not\in \fs_0$ and $\sfH \in \fs_0$}} 
 
 
 From Table \ref{F:gl2-subalg}, we have $\fs_0 \subseteq \fsl_2 = \langle f_{01}, \sfH, e_{01} \rangle$, and by Lemma \ref{L:ss}, we may assume $S := \sfH \in \ff^0$.  Note $[\sfH,e_{01}] = 2 e_{01}$ and $[\sfH,f_{01}] = -2 f_{01}$, while $\pm 2$ are not eigenvalues of $\ad_\sfH|_{\fg_+}$.  By closure under $\ad_\sfH$, we must have $\ff^0 \subseteq \langle f_{01}, \sfH, e_{01} \rangle$.
 
  Choose $\fs^\perp = \langle f_{01}, \sfZ_1, e_{01} \rangle \op \fg_+$.  Then $\fd \in (\fs^* \otimes \fs^\perp)_+$ satisfies $\sfH \cdot \fd = 0$, i.e.\ $\fd$ is a sum of weight vectors with weights that are multiples of $2\alpha_1 + \alpha_2$.  For example, take $\ff^{-1} \ni X_1 \equiv f_{10} \mod \fp$.  The tail terms correspond to
 \begin{align}
 & f_{10}^* \otimes f_{01}, \quad
 f_{10}^* \otimes \sfZ_1, \quad
 f_{10}^* \otimes \sfZ_2, \quad
 f_{10}^* \otimes e_{01},\\
 & f_{10}^* \otimes e_{10}, \quad
 f_{10}^* \otimes e_{11}, \quad
 f_{10}^* \otimes e_{21}, \quad
 f_{10}^* \otimes e_{31}, \quad
 f_{10}^* \otimes e_{32},
 \end{align}
 but $\sfH = -\sfZ_1 + 2\sfZ_2$ acts on these with eigenvalues $-3, -1, -1, 1$ in the first row, and $-2,0,-1,-2,0$ in the second row.  Since $\sfH \cdot \fd = 0$, then only the zero eigenvalue terms are relevant, i.e.\ $f_{10}^* \otimes e_{11}$ and $f_{10}^* \otimes e_{32}$, with weights $2\alpha_1 + \alpha_2$ and $4\alpha_1 + 2\alpha_2$ respectively.  Consequently, 
 \begin{align}
  X_1 &= f_{10} + t_1 e_{11} + t_2 e_{32} \in \ff^{-1}.
 \end{align} 
 Doing this similarly for $X_2 \in \ff^{-1}$ with $X_2 \equiv f_{11} \mod \fp$, we must have
  \begin{align}
  X_2 &= f_{11} + t_3 e_{10} + t_4 e_{31} \in \ff^{-1}.
  \end{align}
  
  Now use the residual exponential freedom $\exp(\fg_2) \subset P_+$ to normalize: using $\exp(x e_{21})$, we have
  \begin{align}
 (\widetilde{t_1},\widetilde{t_3}) =  (t_1 - 2x, t_3 + 2x).
  \end{align}
  Setting $x = \frac{t_1 - t_3}{4}$, we ensure $\widetilde{t_1} = \widetilde{t_3}$.  Dropping tildes, we may assume that \framebox{$t_1 = t_3 =: a$}.  Then:
  \begin{align}
  \ff^{-2} \ni X_3 &:= -\frac{1}{2} [X_1,X_2] + \frac{3a}{2} \sfH \,\,\,= f_{21} + \left(a^2 + \frac{t_2 + t_4}{2}\right) e_{21}, \label{E:ss2X3}\\
  \ff^{-3} \ni X_4 &:= +\frac{1}{3} [X_1,X_3] - \frac{2a}{3} X_1 = f_{31} + \frac{2t_2+t_4}{3} e_{11} + a\left(a^2 -\frac{t_2}{6} + \frac{t_4}{2} \right) e_{32}, \\
  \ff^{-3} \ni X_5 &:= -\frac{1}{3} [X_2,X_3] - \frac{2a}{3} X_2 = f_{32} + \frac{t_2 + 2 t_4}{3} e_{10} + a\left( a^2 + \frac{t_2}{2} - \frac{t_4}{6} \right) e_{31}.
  \end{align}
  
 \begin{lemma} We have \framebox{$t_2 = t_4 =:b$} with \framebox{$ab = 0$}, but $(a,b) \neq (0,0)$.  Moreover, $\ell = \langle X_3 \rangle \mod \ff^0$.
 \end{lemma}
 
 \begin{proof} Note $[X_1,X_5] + a X_3 - (t_2+t_4)\sfH = \frac{t_2 - t_4}{3} \sfZ_1 - \frac{2a(t_2 + t_4)}{3} e_{21} \in \ff^0$. But $\sfZ_1 \not\in \fs_0$ by hypothesis, so $t_2 = t_4 =: b$ and the above lies in $\ff^2 = 0$, which forces $ab = 0$.  Lemma \ref{L:E} and \eqref{E:ss2X3} imply $\ell = \langle X_3 \rangle \mod \ff^0$.

 Assume $a=b=0$.  Then $f_{10}, f_{11}, f_{21}, f_{31}, f_{32} \in \ff$ and $\ell = \langle f_{21} \rangle \mod \ff^0$.  Let $U = r \sfZ_1 + U_0 + U_+ \in \ff^0$, where $U_0 \in \langle f_{01}, \sfH, e_{01} \rangle \subset \fg_0$ and $U_+ \in \fg_+$.  Then $[U,f_{21}] = -2r f_{21} + [U_+,f_{21}]$.  But $\ff^0$-invariance of $\ell$ forces $[U_+,f_{21}] \in \ff^0$, while $\im(\ad_{f_{21}}|_{\fg_+}) = \langle f_{10}, f_{11}, \sfZ_1, e_{10}, e_{11} \rangle$, so we must have $[U_+,f_{21}] = 0$ (since $\sfZ_1 \not\in \fs_0$ and $\ff^1 = 0$).  Since $\ker(\ad_{f_{21}}|_{\fg_+}) = 0$, then $U_+ = 0$. Thus, we have an inclusion into the $\mathsf{M9}$ model, which contradicts maximality.
 \end{proof}
 
  
 \subsubsection{\framebox{$a= 0$, $b \neq 0$}} Use $\exp(t \sfZ_1)$ to normalize $1 = \widetilde{b} = b \exp(4t)$.  Dropping tildes:
  \begin{align}
 X_1 = f_{10} + e_{32}, \quad
 X_2 = f_{11} + e_{31}, \quad
 X_3 = f_{21} + e_{21}, \quad
 X_4 = f_{31} + e_{11}, \quad
 X_5 = f_{32} + e_{10}.
 \end{align}
 Note $[X_1,X_4] = 4e_{01}$ and $[X_2,X_5] = 4 f_{01}$, so $e_{01}, f_{01} \in \ff^0$.  Since $\sfZ_1 \not\in \fs_0$, then $\ff^0 = \langle f_{01}, \sfH, e_{01} \rangle$ and $\dim(\ff) = 8$.   The $\ff^0$-invariant line field is $\ell = \langle X_3 \rangle \mod \ff^0$.  We label this structure $\mathsf{M8}$. 
  

 \subsubsection{\framebox{$a\neq 0$, $b= 0$}} Use $\exp(x \sfZ_1)$ to normalize $1 = \widetilde{a} = a \exp(2x)$.  Dropping tildes:
 \begin{align}
 X_1 = f_{10} + e_{11}, \quad
 X_2 = f_{11} + e_{10}, \quad
 X_3 = f_{21} + e_{21}, \quad
 X_4 = f_{31} + e_{32}, \quad
 X_5 = f_{32} + e_{31}.
 \end{align}
 We know $\sfH \in \ff^0$, and $\sfH$-invariance of the line field forces $\ell = \langle X_3 \rangle \mod \ff^0$.
 
Recall that $\langle \sfH \rangle \subseteq \ff^0 \subseteq \langle f_{01}, \sfH, e_{01} \rangle$.   We have $\ff^0$-invariant $\cD = \ff^{-1} / \ff^0 = \langle X_1, X_2 \rangle \mod \ff^0$, but if $U = u_1 f_{01} + u_2 e_{01} \in \ff^0$, then $[U,X_1] - u_2 X_2 = -2u_2 e_{10}$ and $[U,X_2] - u_1X_1= -2 u_1 e_{11}$, which forces $u_1 = u_2 = 0$ since both lie in $\ff^1 = 0$.  Thus, $\ff^0 = \langle \sfH \rangle$, and $\dim(\ff^0) = 6$.  We label this structure $\mathsf{M6S}$.
 
 
 \subsubsection{Generic case: \framebox{$\sfZ_1 \not\in \fs_0$,\, $S_0 := \sfZ_1 + c\sfH \in \fs_0$ with $c^2 \in \bbC \backslash \{ 0, 1, 9 \}$}}  Via the $G_0$-action, we have either $\fs_0 = \langle S_0, f_{01} \rangle$ or $\fs_0 = \langle S_0 \rangle$, and define $\fs^\perp = \langle \sfZ_1, e_{01} \rangle$ or $\fs^\perp = \langle f_{01}, \sfZ_1, e_{01} \rangle$ respectively.
 By Lemma \ref{L:ss}, we may assume $S := \sfZ_1 + c\sfH \in \ff^0$.  The zero $S$-eigenvalue terms in $f_{10}^* \otimes \fs^\perp$ and $f_{11}^* \otimes \fs^\perp$ correspond to the admissible filtered deformations:
 \begin{align}
 X_1 &= f_{10} + t_{1/3} f_{01} + t_2 e_{31} \in \ff^{-1},\\
 X_2 &= f_{11} + t_{-1/3} e_{01} + t_{-2} e_{32} \in \ff^{-1},
 \end{align}
 where $t_i \in \delta^c_i \bbC := \begin{cases} \bbC, & c = i;\\ 0, & c \neq i. \end{cases}$  Since $t_i t_j = 0$ for $i \neq j$, then taking brackets yields
 \begin{align}
 X_3 = f_{21} \in \ff^{-2}, \quad
 X_4 = f_{31} - \frac{t_2}{3} e_{10} \in \ff^{-3}, \quad 
 X_5 = f_{32} - \frac{t_{-2}}{3} e_{11} \in \ff^{-3}.
 \end{align}
 Furthermore,
 \begin{align}
 \ff\ni [X_1,X_4] &= t_2 \left(\frac{5}{3} \sfZ_1 - 2\sfZ_2\right) - t_{1/3} f_{32},\\
 \ff\ni [X_2,X_5] &= t_{-2} \left( -\frac{1}{3} \sfZ_1 + 2\sfZ_2 \right) - t_{-1/3} f_{31}.
 \end{align}
 When $c = \pm 2$, these lie in $\ff^0$. We have $\fs_0 \subseteq \langle -\sfZ_1 + 4\sfZ_2, f_{01} \rangle$ when $c=2$, while $\fs_0 \subseteq \langle 3\sfZ_1 - 4\sfZ_2, f_{01} \rangle$ when $c=-2$.  Thus, $t_2 = t_{-2} = 0$.  Since $f_{21}, f_{11}, f_{10}$ have distinct $\ad_S$-eigenvalues, then $S$-invariance forces $\ell = \langle X_3 \rangle \mod \ff^0$.
 
 If $\dim(\ff^0) = 1$, there is a clear inclusion into the $\mathsf{M9}$ model.  If $\dim(\ff^0) = 2$, then the zero $S$-eigenvalue terms of $f_{01}^* \otimes \fs^\perp$ with positive homogeneity force
 \begin{align}
 N = f_{01} + n_{-1/3} e_{11} \in \ff^0,
 \end{align}
 where $n_{-1/3} \in \delta^c_{-1/3} \bbC$.
 But then $[N,X_1] - X_2 = -(t_{-1/3} + 3n_{-1/3}) e_{01} \in \ff^0$.  Since $e_{01} \not\in \fs_0$, then $t_{-1/3} = -3n_{-1/3}$.  Hence, $[N,X_2] =  n_{-1/3}\left( -4 \sfZ_1 + 9 \sfZ_2\right) \in \ff^0$.  Since $\fs_0 = \langle \frac{4}{3} \sfZ_1 - \frac{2}{3} \sfZ_2, f_{01} \rangle$ when $c = -\frac{1}{3}$, then $n_{-1/3} = 0$.  We again embed into $\mathsf{M9}$.
 
 
 \subsubsection{Non-generic cases}
 
 Let $c \in \{ \pm 1, \pm 3 \}$.  By Lemma \ref{L:ss}, assume $S := \sfZ_1 + c \sfH \in \ff^0$ (or a multiple thereof).
See Table \ref{F:1Diso} for a summary.

 \begin{table}[h]
  \[
 \begin{array}{|@{}c@{}|c|c|c|c|@{}c@{}|c|cc} \hline
 c & -1 & 1 & -3 & 3 \\ \hline\hline
 S \in \ff^0 & \sfZ_1 - \sfZ_2 & \sfZ_2 & 2\sfZ_1 - 3\sfZ_2 & -\sfZ_1 + 3\sfZ_2\\ \hline
 \mbox{Freedom in $P_+$} & \exp(\langle e_{11} \rangle) & \exp(\langle e_{10} \rangle) & \exp(\langle e_{32} \rangle) & \exp(\langle e_{31} \rangle) \\ \hline
 \begin{array}{c} N \in \ff^0 \\
 {\tiny \mbox{if $\dim(\ff^0) = 2$}}
 \end{array} & f_{01} + n_1 e_{10} + n_2 e_{21} + n_3 e_{32} &  f_{01} & f_{01} + n e_{31} & f_{01} \\ \hline
 X_1 \in \ff^{-1} & f_{10} + t_1 e_{01} & f_{10} + t_1\sfZ_1 + t_2 e_{10} & f_{10} & f_{10} + t e_{21}\\ \hline
 X_2 \in \ff^{-1} & f_{11} + t_2 \sfZ_1 + t_3 e_{11} & f_{11} + t_3 f_{01} & f_{11} + t e_{21} & f_{11} \\ \hline
 \begin{array}{c}
 {\tiny \mbox{Line field generator}}\\
 X_3 \in \ff^{-2} 
 \end{array}
 & f_{21} + t_4 f_{10} + t_5 e_{01} & f_{21} + t_4 f_{11} + t_5 f_{01} & f_{21} - t e_{11} & f_{21} - t e_{10}\\ \hline
 {\tiny \begin{tabular}{c}
 Normalization\\
 via freedom in $P_+$
 \end{tabular}} & t_4 = 0 & t_4 = 0 & t = 0 & t = 0\\ \hline
 \mbox{Maximal?}  & \mbox{No } (\mathsf{M9}) & \mbox{No } (\mathsf{M7}) & \mbox{No } (\mathsf{M9}) & \mbox{No } (\mathsf{M9})\\ \hline
 \end{array}
 \]
 \caption{$\dim(\ff^0) = 1$ cases with $\ell = \langle X_3 \rangle \mod \ff^0$ and $\cD = \langle X_1,X_2 \rangle \mod \ff^0$}
 \label{F:1Diso}
 \end{table}
 We proceed by classifying $S$-invariant filtered deformations as above.  This yields $S,X_1,X_2,X_3$ (and $N$ if $\dim(\ff^0) = 2$) as in Table \ref{F:1Diso}.  The residual freedom in $P_+$ stabilizing the normalization of $S$ is used to normalize the generator $X_3 \mod \ff^0$ of the line field $\ell \subset \ff^{-2} / \ff^0$.  When $c=\pm 3$, the model embeds into $\mathsf{M9}$.  (When $c=-3$ and $\ff^0 = \langle S,N \rangle$, we observe $n=0$ from $[N,X_3] = -n e_{10} \in \ff^1 = 0$.)  So it remains to examine $c = \pm 1$.
 
Suppose $\ff^0 = \langle S \rangle$.  There exists a $G_0$ element $\sigma$ that induces $(\sfZ_1,\sfH) \mapsto (\sfZ_1, -\sfH)$, which induces $c \mapsto -c$.  (In terms of $2\times 2$ matrices, $\sigma = \begin{pmatrix} 0 & 1\\ 1 & 0 \end{pmatrix}$, and $\sigma \diag(1+c,1-c) \sigma^{-1} = \diag(1-c,1+c)$.)  Thus, it suffices to examine $c=1$.  In this case, the model embeds into $\mathsf{M7}$ by enlarging $\ff^0$ by $f_{01}$.

Suppose $\ff^0 = \langle S, N \rangle$.  
\begin{itemize}
\item $c = 1$: we can WLOG take $t_3 = t_5 = 0$ by adding multiples of $N$.   Necessarily $t_2 \neq 0$, otherwise the model embeds into $\mathsf{M9}$.  Applying a rescaling by $\exp(\langle \sfZ_1 \rangle) \subset G_0$, we get $\langle X_1 \rangle \mapsto \langle \lambda^{-1} f_{10} + t_1 \sfZ_1 + \lambda t_2 e_{10} \rangle = \langle f_{10} + \widetilde{t}_1 \sfZ_1 + \widetilde{t}_2 e_{10} \rangle$, where $(\widetilde{t}_1,\widetilde{t}_2) = (\lambda t_1, \lambda^2 t_2)$.  Over $\bbC$, we normalize $t_2 = 1$, and write $t_1 =: \sfa$.  We label these structures $\mathsf{M7_a}$. 
 \item $c=-1$: take $S,N,X_1,X_2,X_3$ as in Table \ref{F:1Diso}.  Then
\begin{align}
 [N,X_3] - 2 t_5 S = -2 n_1 f_{11} + (n_2 - t_5) \sfZ_1 + (n_1 t_5 + n_3) e_{11} \in \ff^{-1}.
\end{align}
 Since $\ell = \langle X_3 \rangle \mod \ff^0$ is $\ff^0$-invariant, then $n_1 = 0$.  But then $[N,X_3] \in \ff^0$ and $\ff^1 = 0$ forces
 $n_2 = t_5$ and $n_3 = 0$.  Furthermore,
 \begin{align}
 [N,X_1]-X_2 - 2t_1 S = -(t_1 + t_2) \sfZ_1 - (t_3 + 2t_5) e_{11} \in \ff^0,
 \end{align}
 so $t_2 = -t_1$ and $t_3 = -2t_5$.  Moreover, $[N,X_2] = 4t_5 e_{10} + 2t_1 t_5 e_{21} + 6 t_5^2 e_{32} \in \ff^1$, so $t_5 = 0$ is forced.  This yields $(S,N,X_1,X_2,X_3,X_4,X_5) = (\sfZ_1 - \sfZ_2,f_{01},f_{10} + t_1 e_{01}, f_{11} - t_1 \sfZ_1, f_{21}, f_{31}, f_{32})$, which embeds into $\mathsf{M9}$.
 \end{itemize}

We have completed our study of all $\fs_0$ that contains a semisimple element.  From Table \ref{F:gl2-subalg}, it remains to consider $\ff^0 = \langle N \rangle$ with $N_0 = \tgr_0(N) =  f_{01} + r \sfZ_1 \in \fs_0$.
 
 \subsubsection{Jordan non-nilpotent case}
  Let $\ff^0 = \langle N \rangle$ with $N_0 = \tgr_0(N) =  f_{01} + r \sfZ_1 \in \fs_0$.
 Since $r \neq 0$, then $\ad_{N_0}|_{\fg_+}$ is injective, since it maps
 \begin{align}
(e_{10},e_{11}, e_{21}, e_{31}, e_{32}) \mapsto
 (r e_{10},re_{11} - e_{10}, 2r e_{21}, 3r e_{31}, 3r e_{32} + e_{31}).
 \end{align}
 Normalizing via the $P_+$-action, we may assume $N = N_0 \in \ff^0$.  Let $\ff^{-1} \ni X_2 \equiv f_{11} \mod \fp$.  Since $\tgr_{-1}([N,X_2]) = -rf_{11}$, then $[N,X_2] + rX_2 \in \ff^0 = \langle N \rangle$.  Evaluation yields $X_2 = f_{11} \mod \ff^0$, so WLOG $X_2 = f_{11}$.  Similarly, take $\ff^{-1} \ni X_1 \equiv f_{10} \mod \fp$. We find that $[S,X_1] + r X_1 - X_2 \in \ff^0 = \langle N \rangle$.  This forces $X_1 \equiv f_{10} \mod \ff^0$, so WLOG $X_1 = f_{10}$.  Taking brackets, we find $\ff^{-2} = \langle f_{21}, f_{11}, f_{10}, N \rangle$.  The only $N$-invariant line field is $\ell = \langle X_3 \rangle \mod \ff^0$ with $X_3 = f_{21} \in \ff^{-2}$.  This model embeds into $\mathsf{M9}$.

 \subsubsection{Nilpotent case}
 Let $\ff^0 = \langle N \rangle$ with $N_0 = \tgr_0(N) =  f_{01} \in \fs_0$.
 
 \begin{lemma} \label{L:nil} Let $N_0 = f_{01} \in \fs_0$.  Normalizing by the $P$-action, we may assume that 
 \begin{align}
 N = N_0 \in \ff^0, \quad \ell = \langle L \rangle \mod \ff^0,
 \end{align}
 where, for some $a \in \bbC$, we have
 \begin{align}
 \ff^{-2} \ni L \equiv f_{21} + a f_{11} \quad\mod \fp.
 \end{align}
 The element $\exp(x_1 e_{10} + x_2 e_{21} + x_3 e_{31}) \in P_+$ preserving $N = N_0$ acts via $\widetilde{a} = a - 2x_1$.
 \end{lemma}
 
 \begin{proof}
 Write $\ff^0 \ni N = f_{01} + n_1 e_{10} + n_2 e_{11} + n_3 e_{21} + n_4 e_{31} + n_5 e_{32}$. Since $[e_{11}, N_0] = e_{10}$ and $[e_{32},N_0] = -e_{31}$, then normalize $n_1 = n_4 = 0$ by exponentiating along $\langle e_{11}, e_{32} \rangle$.
  
 For the generator of the line field, write $\ff^{-2} \ni L \equiv f_{21} + a_1 f_{11} + a_2 f_{10} \mod \fp$.  
 Modulo $\fp$, we have $[N,L] \equiv 2 n_2 f_{10} + a_2 f_{11}$, so by $N$-invariance of $\ell$, we must have $n_2 = a_2 = 0$.  Since $[e_{21},f_{21}] = \sfZ_1$ and $[e_{32},f_{21}] = e_{11}$, then evaluation again yields
 \begin{align}
 \ff^0 \ni [N,L] \equiv n_3 \sfZ_1 + n_5 e_{11} \quad\mod \langle f_{01}, \sfH, e_{01}, e_{10}, e_{21}, e_{31}, e_{32} \rangle,
 \end{align}
 so $n_3 = n_5 = 0$.  The final statement follows since $[e_{10}, f_{21}] = -2 f_{11}$.
 \end{proof}

 So, assume $N = f_{01}$.  Next, find generators for $\ff$ and normalize.  Take
 \begin{align}
 \ff^{-1} \ni X_1 &= f_{10} + t_0 f_{01} + t_1 \sfZ_1 + t_2 \sfH + t_3 e_{01} + t_4 e_{10} + t_5 e_{11} + t_6 e_{21} + t_7 e_{31} + t_8 e_{32}.
 \end{align}
 Adding multiples of $N$, we may assume $t_0=0$.  Since
 \begin{align}
 [e_{10}, f_{10}] = 2\sfZ_1 - 3\sfZ_2, \quad 
 [e_{21}, f_{10}] = -2 e_{11}, \quad
 [e_{31}, f_{10}] = e_{21},
 \end{align}
 and $[e_{10},N] = [e_{21},N] = [e_{31},N] = 0$, then normalize \framebox{$t_1 = t_5 = t_6 = 0$} by exponentiating along $\langle e_{10}, e_{21}, e_{31} \rangle$.  Define $X_2 := [N,X_1] - 2 t_2 N \in \ff^{-1}$, so
 \begin{align} 
 X_1 &= f_{10} + t_2 \sfH + t_3 e_{01} + t_4 e_{10} + t_7 e_{31} + t_8 e_{32}, \label{E:nilX1}\\
 X_2 &= f_{11} - t_3 \sfH + t_8 e_{31}. \label{E:nilX2}
 \end{align}
 Next, we iteratively take brackets, and add appropriate linear combinations of higher homogeneity terms to produce $X_3,X_4,X_5 \in \ff$ with $(X_3,X_4,X_5) \equiv (f_{21}, f_{31}, f_{32})$, each modulo $\fp$.  Namely, 
 \begin{align}
  X_3 &:= -\frac{1}{2} [X_1,X_2] - \frac{t_2}{2} X_2 + t_3 X_1 - \frac{3}{2} t_4 N \in \ff^{-2}\\
  &= f_{21} + \frac{3}{2} t_2 t_3 \sfH + \frac{3}{2} t_3 t_4 e_{10} + t_8 e_{21} + \frac{3}{2} t_3 t_7 e_{31} \nonumber\\
  X_4 &:= \frac{1}{3} [X_1,X_3] + \frac{2}{3} t_4 X_2 + \frac{t_2 t_3}{2} X_1 \in \ff^{-3}\\
  &= f_{31} - t_3 \left( \frac{t_2^2}{2} + \frac{t_4}{3} \right) \sfZ_1 + t_3 \left( t_2^2 + \frac{t_4}{6} \right) \sfZ_2 - \frac{t_2 t_3^2}{2} e_{01} + \left( \frac{t_2 t_3 t_4}{2} - \frac{t_7}{3} \right) e_{10} \nonumber\\
  &\qquad + \left( t_8 - \frac{t_3^2 t_4}{2} \right) e_{11} - \frac{t_3 t_7}{2} e_{21} + \left( \frac{t_2 t_3 t_7}{2} - \frac{t_4 t_8}{3} \right) e_{31} + \frac{t_3^2 t_7}{2} e_{32} \nonumber\\
  X_5 &:= -\frac{1}{3} [X_2,X_3] + \frac{t_2 t_3}{2} X_2 + \frac{3}{2} t_3 t_4 N \in \ff^{-3}\\
  &= f_{32} - \frac{t_2 t_3^2}{2} \sfH + \left(t_8 - \frac{t_3^2 t_4}{2}\right) e_{10} - \frac{t_3^2 t_7}{2} e_{31} \nonumber
 \end{align}
 Thus, $\ff = \langle N, X_1, X_2, X_3, X_4, X_5 \rangle$, but there are further constraints on the parameters. 
 
 \begin{lemma} 
 We must have $t_4 = t_7 = 0$ and $t_8 = \frac{(t_2 t_3)^2}{4} \neq 0$.
 \end{lemma}
 \begin{proof} Let us evaluate the closure condition $[X_1,X_4] \in \ff$.  Modulo $\langle e_{01} \rangle \op \fg_+$, we find:
 \begin{align}
 [X_1,X_4] - t_2 X_4 - t_4 X_3 + t_3 \left( \frac{t_2^2}{2} + \frac{t_4}{3} \right) X_1 &\equiv \left( \frac{5 t_7}{3} + \frac{t_2 t_3 t_4}{2}  \right) \sfZ_1 - \left( t_2 t_3 t_4 + 2 t_7 \right) \sfZ_2.
 \end{align}
 Since $[X_1,X_4] \in \ff^0 = \langle N = f_{01} \rangle$, then $[X_1,X_4] = 0$, so $t_7 = 0 = t_2 t_3 t_4$.  Given this, we have
 \begin{align}
 \begin{split}
  [X_1,X_4] - t_2 X_4 - t_4 X_3 + t_3 \left( \frac{t_2^2}{2} + \frac{t_4}{3} \right) X_1 &\equiv \left(4 t_8 - (t_2 t_3)^2 - \frac{4}{3} t_3^2 t_4 \right) e_{01} -\frac{5}{6} t_3 t_4^2 e_{10} \\
  &\qquad \quad \mod\langle e_{11}, e_{21}, e_{31}, e_{32} \rangle,
  \end{split}
 \end{align}
 so $t_3 t_4 = 0$ and $t_8 = \frac{(t_2 t_3)^2}{4}$.
 
 If $t_3 \neq 0$, then $t_4 = 0$ follows from $t_3 t_4 = 0$.  Assume $t_3 = 0$.  Then $t_8 = 0$ as well, and 
 \begin{align}
 (N,X_1,X_2,X_3,X_4,X_5) = (f_{01}, f_{10} + t_2 \sfH + t_4 e_{10}, f_{11}, f_{21}, f_{31}, f_{32}).
 \end{align}
 From Lemma \ref{L:nil}, we know $\ell = \langle L \rangle \mod \ff^0$, where $L \equiv f_{21} + a f_{11} \mod \fp$.  Apply $p = \exp(x e_{10})$ so that the line field is $\widetilde{\ell} = \langle f_{21} \mod \widetilde\ff^0 \rangle$.  Then $\widetilde\ff = \Ad_p\ff$ is spanned by
 \begin{align}
 (\widetilde{N},\widetilde{X}_1,\widetilde{X}_2,\widetilde{X}_3,\widetilde{X}_4,\widetilde{X}_5) = (f_{01},f_{10} + \widetilde{u} \sfZ_1 + \widetilde{v} \sfH + \widetilde{w} e_{10},f_{11},f_{21},f_{31},f_{32}),
 \end{align}
 where $(\widetilde{u},\widetilde{v},\widetilde{w}) = (\frac{x}{2}, t_2 - \frac{3}{2} x, t_4 + x t_2 - x^2)$.  However, maximality is contradicted from:
 \begin{itemize}
 \item $\widetilde{w} = 0$: the model embeds into $\mathsf{M9}$.
 \item $\widetilde{w} \neq 0$: use $\exp(\langle \sfZ_1 \rangle)$ to normalize $\widetilde{w} = 1$, and then embed into $\mathsf{M7}$.
 \end{itemize}
 
 Next, assume $t_2 = 0$.  Then $\ff$ is spanned by $N = f_{01}$ and 
 \begin{align}
 X_1 = f_{10} + t_3 e_{01}, \quad X_2 = f_{11} - t_3 \sfH, \quad X_3 =  f_{21}, \quad X_4 = f_{31}, \quad X_5 = f_{32}.
 \end{align}
 For $\ell = \langle L \rangle\mod \ff^0$, where $L = X_3 + a X_2$, forcing $\cT \equiv 0$ (via \eqref{E:XXOtorsion}) for the XXO-structure $(\ff,\ff^0;\ell,\cD)$ yields $a=0$.  There is then an inclusion into the $\mathsf{M9}$ model.  Thus, $t_2 \neq 0$.
 \end{proof}
  
 Since $t_2 t_3 \neq 0$, then we normalize both using a $G_0$-rescaling.  From \eqref{E:nilX1} and \eqref{E:nilX2}, we find that $\exp(\langle \sfZ_1 \rangle)$ and $\exp(\langle \sfZ_2 \rangle)$ yield $(t_2,t_3) \mapsto (t_2 \lambda, t_3 \lambda\mu)$, where $\lambda, \mu \in \bbC^\times$.  Let us normalize $(t_2,t_3) = (2,1)$.  Thus, $\ff$ is spanned by $N = f_{01}$, and
 \begin{equation}
 \begin{aligned}
 X_1 &= f_{10} + 2\sfH + e_{01} + e_{32}, \quad
 X_2 = f_{11} - \sfH + e_{31}, \quad
 X_3 = f_{21} + 3\sfH + e_{21}, \label{E:M6N-123}\\
 X_4 &= f_{31} + 2\sfH - e_{01} + e_{11}, \quad
 X_5 = f_{32} - \sfH + e_{10}.
 \end{aligned}
 \end{equation}
 Any $\ff^0$-invariant line field $\ell = \langle L \rangle \mod \ff^0$ has generator of the form 
 \begin{align}
 L = X_3 + a X_2, \quad a \in \bbC.
 \end{align}
 However, using \eqref{E:XXOtorsion} and forcing $\cT \equiv 0$ for the XXO-structure $(\ff,\ff^0;\ell,\cD)$ implies
 \begin{align} \label{E:M6Ntorsion}
 a(a-3)(a-4) = 0.
 \end{align}
 When $a=0$, there is an obvious embedding into the $\mathsf{M8}$ model.  When $a=4$, let $x = 2 e_{10} + e_{21} + e_{31}$ and $\widetilde\ff = \Ad_{\exp(x)} \ff$, which contains 
 \begin{align}
 \begin{array}{l}
 \widetilde{X}_1 := f_{10} + 2(\sfZ_1 - \sfZ_2) + e_{01}  \in \widetilde\ff^{-1}, \\
 \widetilde{X}_2 := f_{11} + \sfZ_1 - 2\sfZ_2 \in \widetilde\ff^{-1},
 \end{array} \quad 
 \begin{array}{l}
 \widetilde{N} := f_{01} \in \widetilde\ff^0,\\
 \widetilde{L} := f_{21} + 2(\sfZ_1 - \sfZ_2)  \in \widetilde\ff^{-2},
 \end{array}
 \end{align}
 where $\widetilde{\ell} = \langle \widetilde{L} \rangle \mod \widetilde\ff^0$, and $\widetilde{X}_4,\widetilde{X}_5$ are induced via brackets. Clearly, the model embeds into $\mathsf{M9}$.
 
 \begin{prop} The $\mathsf{M6N}$ model ($a=3$) does not embed into a model in Table \ref{F:XXO} of larger dimension (or a $P$-translate thereof).
 \end{prop}

 \begin{proof}
 When $a=3$, we have $(N,L) = (f_{01},f_{21} + 3 f_{11} + e_{21} + 3 e_{31})$.  Normalize this via the $P_+$-action.  Letting $x = x_1 e_{10} + x_2 e_{11} + x_3 e_{21} + x_4 e_{31} + x_5 e_{32} \in \fp_+$, and $\widetilde\ff = \Ad_{\exp(x)} \ff$, the condition $\widetilde\ff^0 \subset \fg_0$ forces $x_2 = x_5 = 0$, and the condition $\widetilde{L} \equiv f_{21} \mod \fp$ then forces $x_1 = \frac{3}{2}$.  Then, $\widetilde{N} = f_{01}$ and 
 \begin{align}
 \widetilde{L} \equiv f_{21} + x_3 \sfZ_1 + \left( \frac{9}{4} x_3 - x_4 \right) e_{10} + (1 - x_3^2) e_{21} + \left( 3 x_3 \left( \frac{9}{4} x_3 - x_4\right) - \frac{3}{2} \right) e_{31} \quad\mod \widetilde\ff^0.
 \end{align}
 It is clear that the $e_{10}$ and $e_{31}$ coefficients cannot both be removed via the $P_+$-action, nor any further $G_0$-action.
 From Table \ref{F:XXO}, all models of greater dimension are presented with line field generators of the form $L \equiv f_{21} \mod \ff^0$ or $L \equiv f_{21} + e_{21} \mod \ff^0$, so the conclusion follows.
 \end{proof}


 \subsection{Real forms}
 \label{S:RXXO}
 

 Given the classification of complex admissible $(\ff,\ff^0;\ell,\cD)$ from Section \ref{S:CXXO}, we now give the abstract classification of real forms.  Recall that an anti-involution of $\ff$ is a conjugate-linear, bracket-preserving map $\varphi : \ff \to \ff$ such that $\varphi^2 = \id_\ff$.
 
 \begin{defn} For a complex admissible model $(\ff,\ff^0;\ell,\cD)$,
 \begin{enumerate}
 \item an {\sl automorphism} is a map $T \in \operatorname{Aut}(\ff)$ such that $T(\ff^i) = \ff^i$, $\forall i$, and which induces maps that preserve $\ell$ and $\cD$;
 \item a {\sl real form} is an anti-involution $\varphi: \ff \to \ff$ such that $\varphi(\ff^i) = \ff^i$, $\forall i$, and which induces maps that preserve $\ell$ and $\cD$. 
 \end{enumerate}
Two real forms $\varphi,\widetilde\varphi$ are {\sl equivalent} if there exists an automorphism $T$ such that $\widetilde\varphi = T \circ \varphi \circ T^{-1}$.
 \end{defn}
 
 Each anti-involution $\varphi$ yields a real admissible model $(\ff_\bbR,\ff^0_\bbR;\ell_\bbR,\cD_\bbR)$ determined by the (real) fixed point sets, e.g.\ $\ff_\bbR := \ff^\varphi = \{ x \in \ff : \varphi(x) = x \}$.
 
 \begin{example} For the $\mathsf{M8.2}$ model, the anti-involution $\varphi$ given in Table \ref{F:RXXO-AI} yields
 \begin{align}
 \ff_\bbR := \ff^\varphi = \langle (1+i) X_1, (1+i) X_2, i X_3, (1-i) X_4, (1-i) X_5,\sfH, e_{01}, f_{01} \rangle.
 \end{align}
 For the $\mathsf{M6S.3}$ model, the anti-involution $\varphi$ given in Table \ref{F:RXXO-AI} yields
 \begin{align}
 \ff_\bbR := \ff^\varphi = \langle X_1 - X_2, i(X_1 + X_2), i X_3, X_4 - X_5, i(X_4 + X_5),i\sfH \rangle.
 \end{align}
 \end{example}

 Our main result is the classification of real forms up to equivalence:

 \begin{theorem} \label{T:AI}
 The complete classification of real forms of complex admissible models $(\ff,\ff^0;\ell,\cD)$ is given in Tables \ref{F:RXXO} and \ref{F:RXXO-AI}.
 \end{theorem}

 \begin{table}[h]
 \[
 \begin{array}{|c|c|c|c|c|c|} \hline
 \mbox{$\bbC$-model} & \mbox{Restrictions} & \mbox{Real forms} & \mbox{Real labels} & \mbox{Symmetry algebra $\ff \subset G_2$}\\ \hline\hline
 \mathsf{M9} & - & 1 & \mathsf{M9} & \fp_1\\ \hline
 \multirow{2}{*}{$\mathsf{M8}$} &  \multirow{2}{*}{$-$} &  \multirow{2}{*}{$2$} & \mathsf{M8.1} & \fsl(3,\bbR)\\
 &&& \mathsf{M8.2} & \fsu(1,2)\\ \hline
 \multirow{2}{*}{$\mathsf{M7}_\sfa$} & \sfa \in \bbR^\times \cup i\bbR^\times & 1 & \mathsf{M7}_\sfa & \multirow{2}{*}{$\bbR^2 \ltimes \mathfrak{heis}_5$}\\
 & \sfa=0 & 2 & \mathsf{M7}_0^\pm & \\ \hline
 \multirow{3}{*}{$\mathsf{M6S}$} &  \multirow{3}{*}{$-$} &  \multirow{3}{*}{$3$} & \mathsf{M6S}.1 & \fsl(2,\bbR) \times \fsl(2,\bbR)\\
 &&& \mathsf{M6S}.2 & \fsl(2,\bbR) \times \fsl(2,\bbR) \\
 &&& \mathsf{M6S}.3 & \fso(3) \times \fso(3) \\ \hline
 \mathsf{M6N} & - & 1 & \mathsf{M6N} & \mathfrak{aff}(2,\bbR)\\ \hline
 \end{array}
 \]
 \caption{Overview of real forms of complex admissible models $(\ff,\ff^0;\ell,\cD)$}
 \label{F:RXXO}
 \end{table}

 \begin{table}[h]
 \[
 \begin{array}{|c|c|c|c|c|c|} \hline
 \mbox{$\bbC$-model} & \mbox{Basis} & \mbox{$\bbR$-model} & \mbox{Anti-involutions} \\ \hline\hline
 \mathsf{M9} & \begin{array}{c} X_1, X_2, X_3, X_4, X_5,\\ \sfH, e_{01}, f_{01}, \sfZ_1 \end{array} & \mathsf{M9} & \id\\ \hline
 \multirow{2}{*}{$\mathsf{M8}$} & \multirow{2}{*}{$\begin{array}{c} X_1, X_2, X_3, X_4, X_5,\\ \sfH, e_{01}, f_{01} \end{array}$} & \mathsf{M8.1} & \id \\
& & \mathsf{M8.2} & \diag(i,i,-1,-i,-i,1,1,1) \\ \hline
 \mathsf{M7}_\sfa & \begin{array}{c} X_1, X_2, X_3, X_4, X_5,\\ \sfZ_2, f_{01} \end{array} & \begin{array}{l} \mathsf{M7}_{\sfa\neq 0} \\ \mathsf{M7}_{\sfa=0}^\pm \end{array} & \begin{cases}
 \id, & \sfa\in \bbR\\
 \diag(-1,-1,1,-1,-1,1,1), & \sfa\in i\bbR
 \end{cases}\\ \hline
 \mathsf{M6S}.1 &  & \mathsf{M6S}.1 & \id\\
 \mathsf{M6S}.2 & \multirow{2}{*}{$\begin{array}{c} X_1, X_2, X_3, X_4, X_5,\\ \sfH \end{array}$}& \mathsf{M6S}.2 & \diag\left( \begin{pmatrix} 0 & 1 \\ 1 & 0\end{pmatrix}, -1, \begin{pmatrix} 0 & 1 \\ 1 & 0\end{pmatrix}, -1 \right)\\
 \mathsf{M6S}.3 & & \mathsf{M6S}.3 & \diag\left( \begin{pmatrix} 0 & -1 \\ -1 & 0\end{pmatrix}, -1, \begin{pmatrix} 0 & -1 \\ -1 & 0\end{pmatrix}, -1 \right)\\ \hline
 \mathsf{M6N} & \begin{array}{c} X_1, X_2, X_3, X_4, X_5,\\ f_{01} \end{array} & \mathsf{M6N} & \id\\ \hline
 \end{array}
 \]
 \caption{Anti-involutions of complex admissible models $(\ff,\ff^0;\ell,\cD)$}
 \label{F:RXXO-AI}
 \end{table}

The proof of Theorem \ref{T:AI} is a straightforward, but tedious computation, and do not provide details here.  We take advantage of the fact that any anti-involution $\varphi$ of $(\ff,\ff^0;\ell,\cD)$ is bracket-preserving, so ideals, derived \& lower central series are preserved by $\varphi$, in addition to being filtration and structure preserving.  We also note that a source of automorphisms generating equivalences of anti-involutions is $\exp(\ff^0) \subset \Aut(\ff)$, so this can be used for some initial normalization of $\varphi$.
 

 \section{Conformal structures}
 

 For each (real or complex) admissible XXO-structure $(\ff,\ff^0;\ell,\cD)$ in our classification (Table \ref{F:XXO}), quotienting by $\ell$ yields a well-defined $\ff^0$-invariant conformal structure on $\ff/\ff^0$.  Integration yields a local coordinate model for a conformal structure whose conformal symmetry algebra acts multiply-transitively, is isomorphic to $\ff$, and has isotropy subalgebra isomorphic to $\ff^0$ at a generic point.

 In this section, for each of these conformal structures, we:
 \begin{enumerate}
 \item establish the Petrov classification of the (SD \& ASD parts of the) Weyl tensor;
 \item give a local coordinate model;
 \item describe the model Cartan-theoretically;
 \item calculate the conformal holonomy;
 \item assess the existence of an Einstein representative.
 \end{enumerate}
 The Petrov classification is confirmed in three independent manners: Lie-theoretically, in a coordinate XXO-fashion, and Cartan-theoretically.  A fourth (more traditional) method (which we do not discuss) is to directly compute the Weyl tensor of $[\sfg]$, decompose it into SD \& ASD parts, and write each in a spinorial fashion so that their descriptions as binary quartics manifest themselves.


 \subsection{Petrov types}
 
 
  The fundamental tensor for 4D split-conformal structures is the Weyl tensor, which decomposes into SD and ASD parts.  Each can be viewed as a binary quartic, so each has a {\sl Petrov type} corresponding to the root type multiplicities: $\sfO$ (quartic vanishes), $\sfN$ (quadruple root), $\sfD$ (two double roots), $\mathsf{III}$ (one multiplicity 3 root, one simple root), $\mathsf{II}$ (one double root, two other simple roots), $\mathsf{I}$ (four distinct simple roots).  In the real case, one has further refinements, e.g.\ $\sfD^+$ (two real double roots) and $\sfD^-$ (two complex conjugate double roots).  We will write the combined Petrov type as $A.B$, where $A$ and $B$ refer to the SD and ASD Petrov types, respectively.
  
 \begin{theorem} \label{T:Petrov-Lie} The Petrov types for the conformal structures arising from Table \ref{F:XXO} are given below. 
 \[
 \begin{array}{|c|c|c|c|c|} \hline
 \bbC\mbox{-model} & \mbox{Petrov type} & \mbox{Remarks on real forms}\\ \hline\hline
 \mathsf{M9} & \sfN.\sfO & \\
 \mathsf{M8} & \sfD.\sfO & \begin{cases} \mathsf{M8.1}: \sfD^+.\sfO, \,\,\, \ff \cong \fsl(3,\bbR) \\ \mathsf{M8.2}: \sfD^-.\sfO, \,\,\, \ff \cong \fsu(1,2)\end{cases} \\
 \begin{array}{@{}c@{}} 
 \mathsf{M7}_\sfa \\
 \mbox{\tiny $(\sfa \in \bbC)$}
 \end{array} & \begin{cases} \sfN.\sfN, & \sfa^2 \neq \frac{4}{3}\\ \sfN.\sfO, & \sfa^2 = \frac{4}{3} \end{cases} & 
 \begin{tabular}{c}
 One real form $\mathsf{M7_a}$ exists when $\sfa^2 \in \bbR^\times$;\\
 two real forms $\mathsf{M7_0^\pm}$ exist when $\sfa = 0$.
 \end{tabular}\\
 \mathsf{M6S} & \sfD.\sfD & 
 \begin{cases} 
 \mathsf{M6S}.1: \sfD^+.\sfD^+, \,\,\, \ff \cong \fsl(2,\bbR) \times \fsl(2,\bbR)\\ 
 \mathsf{M6S}.2: \sfD^-.\sfD^-, \,\,\, \ff \cong \fsl(2,\bbR) \times \fsl(2,\bbR)\\ 
 \mathsf{M6S}.3: \sfD^-.\sfD^-, \,\,\, \ff \cong \fso(3) \times \fso(3)
 \end{cases} \\
 \mathsf{M6N} & \mathsf{III}.\sfO & \\ \hline
 \end{array}
 \]
 \end{theorem}
  
 We now indicate how Theorem \ref{T:Petrov-Lie} was obtained from Lie-theoretic data in Table \ref{F:XXO}.  For each model there, $\ell = \langle L \rangle \mod \ff^0$ and $\cD = \langle X_1, X_2 \rangle \mod \ff^0$.  Define $\fk := \langle L \rangle + \ff^0$, so that $\ff / \fk$ is 4-dimensional, and carries a $\fk$-invariant conformal structure $[\sfg]$, where $\sfg \in S^2(\ff/\fk)^*$.  Recall that $\fk$-invariance is expressed in matrix form via $ A^\top \sfg + \sfg A = \lambda(A) \sfg, \,\, \forall A \in \fk$.  We easily confirm:
 
 \begin{lemma} \label{L:fk} Define $\we_i \in \ff$ below, and define a basis of $\ff / \fk$ given by $e_i := \we_i \mod \fk$.
 \[
 \begin{array}{|c|cccc|c|} \hline
 \mbox{Model} & \we_1 & \we_2 & \we_3 & \we_4 \\ \hline\hline
 \mathsf{M9} & X_1 & X_5 & X_4 & X_2 \\ 
 \mathsf{M8} & X_1 & X_5 & X_4 & X_2 \\ 
 \mathsf{M7}_\sfa & X_1 & X_5 & X_4 - \frac{1}{3} X_2 & X_2 \\
 \mathsf{M6S} & X_1 & X_5 + \frac{2}{3} X_2 & X_4 & X_2 \\
 \mathsf{M6N} & X_1 & X_5 & X_4 + 2 X_2 & X_2 \\ \hline
 \end{array}
 \] 
 Letting $\{ \theta^i \}$ be the dual basis to $\{ e_i \}$, the unique $\fk$-invariant $[\sfg]$ on $\ff / \fk$ is given by $\sfg = 2(\theta^1 \theta^2 + \theta^3 \theta^4)$.
\end{lemma}

 Letting $\vol_\sfg = \theta^1 \wedge \theta^2 \wedge \theta^3 \wedge \theta^4$, the two families of totally null 2-planes for $[\sfg]$ are:
 \begin{itemize} 
 \item SD: \,\,$\langle s e_1 + t e_3, s e_4 - t e_2 \rangle$, parametrized by $[s:t] \in \bbP^1$.  Write $\Pi^{\mbox{\tiny SD}}_\xi := \langle e_1 + \xi e_3, e_4 - \xi e_2 \rangle$ and $\Pi^{\mbox{\tiny SD}}_{\xi=\infty} := \langle e_3, e_2 \rangle$.
 \item ASD: $\langle s e_1 + t e_4, s e_3 - t e_2 \rangle$, parametrized by $[s:t] \in \bbP^1$.   Write $\Pi^{\mbox{\tiny ASD}}_\eta := \langle e_1 + \eta e_4, e_3 - \eta e_2 \rangle$ and $\Pi^{\mbox{\tiny ASD}}_{\eta=\infty} := \langle e_4, e_2 \rangle$.
 \end{itemize} 
 Let $\fk^{\mbox{\tiny SD}}_\xi$ and $\fk^{\mbox{\tiny ASD}}_\eta$ be the subalgebras of $\fk$ preserving $\Pi^{\mbox{\tiny SD}}_\xi$ and $\Pi^{\mbox{\tiny ASD}}_\eta$, respectively.

 Fixing $\xi$ and $\eta$, the 5-dimensional quotients $\ff / \fk^{\mbox{\tiny SD}}_\xi$ and $\ff / \fk^{\mbox{\tiny ASD}}_\eta$ are infinitesimal models that carry the associated SD and ASD twistor XXO-structures $(\ell,\cD)$.  Let us describe the XXO-data $(\ell,\cD)$ on them.  Clearly, $\fk / \fk^{\mbox{\tiny SD}}_\xi \subset \ff / \fk^{\mbox{\tiny SD}}_\xi$ and $\fk / \fk^{\mbox{\tiny ASD}}_\eta \subset \ff / \fk^{\mbox{\tiny ASD}}_\eta$ generate distinguished lines $\ell$, and choosing a generator for $\ell$, we can express it as a vector field on $\xi$-space or $\eta$-space, respectively.  In the SD case, the 3-plane $\cH = \ell \oplus \cD$ is the pullback under the quotient $\ff / \fk^{\mbox{\tiny SD}}_\xi \to \ff / \fk$ of $\Pi^{\mbox{\tiny SD}}_\xi$, and similarly for the ASD case.  Inside $\cH$, the twistor 2-plane lift $\cD$ is uniquely determined by $[\cD,\cD] \subset \cH$.  Explicitly,
 \begin{itemize}
 \item SD:  $\cD = \langle \we_1 + \xi \we_3 + A V, \we_4 - \xi \we_2 + BV\rangle \mod \fk^{\mbox{\tiny SD}}_\xi$, where $0 \neq V \in \fk / \fk^{\mbox{\tiny SD}}_\xi$;
 \item ASD:  $\cD = \langle \we_1 + \eta \we_4 + A V, \we_3 - \eta \we_2 + BV\rangle \mod \fk^{\mbox{\tiny ASD}}_\eta$, where $0 \neq V \in \fk / \fk^{\mbox{\tiny ASD}}_\eta$.
 \end{itemize}
  Then $\cW^{\pm}$ are determined (up to scale) by the relative invariant $\cI : \bigwedge^2 \cD \to \ell$, $(x,y) \mapsto \proj_\ell([x,y])$. 
  
 \begin{example}[$\mathsf{M6N}$]  For $\mathsf{M6N}$, define $X_1,...,X_5,L$ as in Table \ref{F:XXO}, and $\we_i,e_i,\theta^i$ as in Lemma \ref{L:fk}.  We have $\ff^0 = \langle f_{01} \rangle$ and $\fk = \langle L, f_{01} \rangle$. The isotropy rep $\rho: \fk \to \fgl(\ff / \fk)$ in the $\{ e_1, e_2, e_3, e_4 \}$ basis is:
 \begin{align}
 \rho(c_1 L + c_2 f_{01}) = \begin{pmatrix} 
 -3 c_1 & 0 & 0 & 0\\
 0 & 0 & 6 c_1 - c_2 & 3 c_1\\
 -3 c_1 & 0 & 0 & 0\\
 -6 c_1 + c_2 & 0 & 0 & -3 c_1
 \end{pmatrix}. 
 \end{align}
 The unique $\fk$-invariant conformal structure $[\sfg]$ on $\ff / \fk$ is given by
 $g = 2\theta^1 \theta^2 + 2\theta^3 \theta^4 \in S^2(\ff / \fk)^*$.

 \underline{SD case}: $c_1 L + c_2 f_{01} \in \fk$ acts on $\Pi^{\mbox{\tiny SD}}_\xi$ via the infinitesimal change $\xi \mapsto 3c_1 (\xi-1)$, i.e.\ $3c_1(\xi-1) \partial_\xi$:
 \begin{align}
 e_1 + \xi e_3 \mapsto e^{t\rho(c_1 L + c_2 f_{01})} (e_1 + \xi e_3) &= e_1 + \xi e_3 + t \rho(c_1 L + c_2 f_{01}) + O(t^2)\\
 &\equiv (1- 3 c_1 t) e_1 + (\xi - 3 c_1 t) e_3 + O(t^2) \,\,\mod \{ e_2, e_4 \}
 \end{align}
 Thus, $\xi \mapsto \frac{\xi - 3 c_1 t}{1- 3 c_1 t} + O(t^2)$.  Differentiating at $t=0$ yields $\xi \mapsto 3c_1 (\xi-1)$.  So $\fk^{\mbox{\tiny SD}}_\xi = \langle f_{01} \rangle = \ff^0$, and $L$ quotients to a generator of $\ell = \fk / \fk^{\mbox{\tiny SD}}_\xi$, acting via $3(\xi-1)\partial_\xi$.  The condition 
 \begin{align}
 v := [\we_1 + \xi \we_3 + A L, \we_4 - \xi \we_2 + BL] \in \langle \we_1 + \xi \we_3, \we_4 - \xi \we_2, L, f_{01} \rangle
 \end{align}
 uniquely determines the twistor distribution $\cD = \langle \we_1 + \xi \we_3 + A L, \we_4 - \xi \we_2 + BL \rangle \mod \ff^0$, and pins down $(A,B) = (-2\xi,0)$.  The components of $v$ in the basis
 \begin{align}
 \we_1 + \xi \we_3 - 2\xi L, \quad \we_4 - \xi \we_2, \quad L, \quad f_{01}
 \end{align}
 are $(-2\xi + 2, 4, -2(\xi-1)^2, 16\xi^2)$.  The projection $\bigwedge^2 \cD \to \ell$ has image a multiple of $-2(\xi-1)^2 L$.  But we saw that $L$ acts via $3(\xi-1)\partial_\xi$, so $\cW^+$ is a multiple of $(\xi-1)^3$.  Viewing this as a quartic on $\xi \in \bbP^1$, it is Petrov type $\mathsf{III}$, with a triple root at $\xi = 1$ and a simple root at $\xi = \infty$.  At these exceptional values, we get integrable 2-planes $\langle \we_1+\we_3,\we_4 - \we_2 \rangle \mod \fk^{\mbox{\tiny SD}}_{\xi=1}$ and $\langle \we_3,\we_2 \rangle \mod  \fk^{\mbox{\tiny SD}}_{\xi=\infty}$.
 
 \underline{ASD case}: $c_1 L + c_2 f_{01} \in \fk$ acts on $\Pi^{\mbox{\tiny ASD}}_\eta$ via $(c_2-6c_1) \partial_\eta$, so $\fk^{\mbox{\tiny ASD}}_\eta = \langle L+6 f_{01} \rangle$, and $L$ quotients to a generator of $\ell = \fk / \fk^{\mbox{\tiny ASD}}_\eta$, acting via $-6\partial_\eta$.  We now proceed similarly as above.
 \end{example}

 Data associated to all our (complex) models is compiled below.
 \[
 \begin{array}{|c|c|c|c|c|c|c|c|} \hline
 \mbox{Model} & \fk^{\mbox{\tiny SD}}_\xi & \multicolumn{2}{c|}{\fk / \fk^{\mbox{\tiny SD}}_\xi} & A & B & \cW^+ & \mbox{Type} \\ \hline\hline
 \mathsf{M9} & -2 \xi L+3\sfZ_1,\,e_{01},\,\sfH,\, f_{01} & L & -3 \partial_\xi & 0 & 0 & & \sfN\\
 \mathsf{M8} & e_{01},\, \sfH, \,f_{01} & L & 3(\xi^2-1) \partial_\xi & 0 & 0 & 6(\xi^2-1)^2 & \sfD\\
 \mathsf{M7}_\sfa & \xi L-3 \sfZ_2,\, f_{01} & L & -3 \partial_\xi & -\frac{2\sfa\xi}{3} & 0 & 6 & \sfN\\
 \mathsf{M6S} & \sfH & L & (\xi + 3)(\xi - 1)\partial_\xi & 0 & 0 & \frac{2}{3} (\xi + 3)^2(\xi - 1)^2 & \sfD\\
 \mathsf{M6N} & f_{01} & L & 3(\xi-1) \partial_\xi & -2\xi & 0 & -6(\xi-1)^3 & \mathsf{III}\\ \hline
 \end{array}
 \]
 
  \[
 \begin{array}{|c|c|c|c|c|c|c|c|} \hline
 \mbox{Model} & \fk^{\mbox{\tiny ASD}}_\eta & \multicolumn{2}{c|}{\fk / \fk^{\mbox{\tiny ASD}}_\eta} & A & B & \cW^- & \mbox{Type}\\ \hline\hline
 \mathsf{M9} & L, \,\sfZ_1,\, 2 e_{01}-\eta\sfH,\,\sfH+2\eta f_{01} & f_{01} & \partial_\eta & 0 & 0 & 0 & \sfO\\
 \mathsf{M8} & L, \,\,\sfH+2\eta f_{01},\,\, \eta\sfH - 2e_{01} & f_{01} & \partial_\eta & 0 & 0 & 0 & \sfO\\
 \mathsf{M7}_\sfa & L-f_{01}, \sfZ_2+\eta f_{01} & f_{01} & \partial_\eta & 0 & -\frac{2\sfa}{3} & -2(\sfa^2 - \frac{4}{3}) & \sfN \mbox{ or } \sfO\\
 \mathsf{M6S} & L+\sfH & L & 2\eta\partial_\eta & 0 & 0 & \frac{32}{3} \eta^2 & \sfD\\
 \mathsf{M6N} & L + 6 f_{01} & L & -6\partial_\eta & \frac{\eta(\eta-4)}{6} & -\frac{\eta^2}{6} - \frac{2}{3} & 0 & \sfO\\ \hline
 \end{array}
 \]

Now consider Petrov types for real forms (Table \ref{F:RXXO-AI}).  The types are necessarily the same for $\mathsf{M9}$ ($\sfN.\sfO$), $\mathsf{M7}_\sfa$ ($\sfN.\sfN$ or $\sfN.\sfO$), and $\mathsf{M6N}$ ($\mathsf{III}.\sfO$).  For $\mathsf{M8.1}$ and $\mathsf{M6S}.1$, computations are the same as in the complex case, so we get real roots, and the claimed types: $\sfD^+.\sfO$ and $\sfD^+.\sfD^+$ respectively.  Below are details for $\mathsf{M8.2}$ ($\sfD^-.\sfO$) and $\mathsf{M6S}.2 \& 3$ ($\sfD^-.\sfD^-$).
 
    \[
 \begin{array}{|c|c|c|c|c|c|} \hline
 \mbox{Model} & \we_1 & \we_2 & \we_3 & \we_4 & L\\ \hline\hline
 \mathsf{M8.2} & (1+i)X_1 & (1-i)X_5 & (1-i)X_4 & (1+i)X_2 & iX_3\\ 
 \mathsf{M6S}.2 & X_1 + X_2 & X_4 + X_5 + \frac{1}{3} (X_1 + X_2) & i(X_1-X_2) & i(X_4 - X_5) + \frac{1}{3} i(X_1 - X_2) & iX_3\\ 
 \mathsf{M6S}.3 & X_1 - X_2 & X_4 - X_5 + \frac{1}{3} (X_1 - X_2) & i(X_1+X_2) & i(X_4 + X_5) + \frac{1}{3} i(X_1 + X_2) & iX_3\\ \hline
 \end{array}
 \]

 \[
 \begin{array}{|c|c|c|c|c|c|c|c|} \hline
 \mbox{Model} & \fk^{\mbox{\tiny SD}}_\xi & \multicolumn{2}{c|}{\fk / \fk^{\mbox{\tiny SD}}_\xi} & A & B & \cW^+ & \mbox{Type} \\ \hline\hline
 \mathsf{M8.2} & e_{01},\, \sfH, \,f_{01} & L & 3(\xi^2+1) \partial_\xi & 0 & 0 & -12(\xi^2+1)^2 & \sfD^-\\
 \mathsf{M6S}.2 & i\sfH & L & -(\xi^2+1)\partial_\xi & 0 & 0 & -\frac{10}{3}(\xi^2 + 1)^2 & \sfD^-\\
 \mathsf{M6S}.3 & i\sfH & L & -(\xi^2+1)\partial_\xi & 0 & 0 & +\frac{10}{3} (\xi^2 + 1)^2 & \sfD^-\\ \hline
 \end{array}
 \]

 \[
 \begin{array}{|c|c|c|c|c|c|c|c|} \hline
 \mbox{Model} & \fk^{\mbox{\tiny ASD}}_\eta & \multicolumn{2}{c|}{\fk / \fk^{\mbox{\tiny ASD}}_\eta} & A & B & \cW^- & \mbox{Type}\\ \hline\hline
 \mathsf{M8.2} & iX_3, \,\,\sfH+2\eta f_{01},\,\, \eta\sfH - 2e_{01} & f_{01} & \partial_\eta & 0 & 0 & 0 & \sfO\\
 \mathsf{M6S}.2 & i\sfH & L & -\frac{4\eta^2 + 9}{3} \partial_\eta & 0 & 0 & -\frac{64}{27} (\eta^2 + \frac{9}{4})^2 & \sfD^-\\
 \mathsf{M6S}.3 & i\sfH & L & -\frac{4\eta^2 + 9}{3} \partial_\eta & 0 & 0 & +\frac{64}{27} (\eta^2 + \frac{9}{4})^2 & \sfD^-\\ \hline
 \end{array}
 \]


\subsection{Coordinate models}
We will present local coordinate models of the conformal structures and the corresponding XXO-structures appearing in our classification. In the $\mathsf{M9}$, $\mathsf{M8}$ and $\mathsf{M6S}$ cases there were known metrics realizing the abstract models. The $\mathsf{M7_a}$   and $\mathsf{M6N}$ models are new and the coordinate models have been obtained via integration of the structure equations of the  XXO-structures. Null coframes $(\theta^1,\theta^2,\theta^3,\theta^4)$  in which the representative metric has the form $\sfg=2(\theta^1\theta^2+\theta^3\theta^4)$ are summarized in Table \ref{table-coframe}.

\begin{table}[h]\label{table-coframe}
 \[
 \begin{array}{|l|l|}  \hline
 \mathsf{M9} & 
 \begin{array}{l@{\,}l}
 \theta^1=\mathsf{d}x,\quad \theta^2=\tfrac12 \mathsf{d}u,\quad  \theta^3=\mathsf{d}y+x^2\mathsf{d}v,\quad \theta^4=\tfrac12\mathsf{d}v

 \end{array} 
  \\ \hline
 \mathsf{M8.1} &
 \begin{array}{l@{\,}l}
  \theta^1=\tfrac{1}{v+ux-y}\mathsf{d}x,\quad \theta^2=\tfrac{y-v}{v+ux-y}\mathsf{d}u+\tfrac{u}{v+ux-y}\mathsf{d}v,\\
  \theta^3=\tfrac{x}{v+ux-y}\mathsf{d}u+\tfrac{1}{v+ux-y}\mathsf{d}v,\quad \theta^4=\tfrac{-1}{v+ux-y}\mathsf{d}y
 \end{array}  \\ \hline
\mathsf{M8.2} &
 \begin{array}{l@{\,}l}
  \theta^1=\tfrac{\sqrt{x^2+y^2+2u}+x}{x^2+y^2+2u}\mathsf{d}x+\tfrac{y}{x^2+y^2+2u}\mathsf{d}y+\tfrac{1}{x^2+y^2+2u}\mathsf{d}u\\
 \theta^{2}=-\tfrac{\sqrt{x^2+y^2+2u}-x}{x^2+y^2+2u}\mathsf{d}x+\tfrac{y}{x^2+y^2+2u}\mathsf{d}y+\tfrac{1}{x^2+y^2+2u}\mathsf{d}u\\
\theta^3=-\tfrac{y}{x^2+y^2+2u}\mathsf{d}x+\tfrac{\sqrt{x^2+y^2+2u}+x}{x^2+y^2+2u}\mathsf{d}y+\tfrac{1}{x^2+y^2+2u}\mathsf{d}v\\
 \theta^4=-\tfrac{y}{x^2+y^2+2u}\mathsf{d}x-\tfrac{\sqrt{x^2+y^2+2u}-x}{x^2+y^2+2u}\mathsf{d}y+\tfrac{1}{x^2+y^2+2u}\mathsf{d}v
 \end{array}  \\ \hline
 \begin{array}{@{}l@{}} \mathsf{M7_a} \\ \mbox{\tiny $(\sfa^2 \in \bbR^\times)$}\\
 \mathsf{M7_0^\pm}
 \end{array} &
 \begin{array}{l@{\,}l} \theta^1=\tfrac{9}{2}(2r^2\epsilon+4r x+y^2-1)\mathsf{d}u+6\epsilon (r y + 2x)\mathsf{d}v-6\epsilon \mathsf{d}x,\quad \theta^2=\mathsf{d}u,\\ \theta^3=(6 r^2 - 3y - 5\epsilon)\mathsf{d}v+3\mathsf{d}y,\quad \theta^{4}=\mathsf{d}v, \\
  \mbox{where}\quad  r=\vert \sfa\vert\geq 0,\quad \epsilon=\begin{cases}
\mathrm{sgn}(\sfa^2), & \sfa \neq 0;\\
 \pm 1, & \sfa = 0
\end{cases}
 \end{array} 
   \\ \hline
 \begin{array}{@{}c} \mathsf{M6S}.1\\ \mathsf{M6S}.2\\ \mathsf{M6S}.3\end{array} & \begin{array}{@{}l@{\,}l}\begin{array}{l@{\,}l}
 \theta^1=\frac{1}{\kappa(x^2+\epsilon y^2)+1}\mathsf{d}x+ \frac{1}{9\kappa (u^2+\epsilon v^2)+1}\mathsf{d}u\\
 \theta^2=\frac{1}{\kappa (x^2+\epsilon y^2)+1}\mathsf{d}x-\frac{1}{9\kappa (u^2+\epsilon v^2)+1}\mathsf{d}u\\
 \theta^3=\frac{\epsilon}{\kappa(x^2+\epsilon y^2)+1}\mathsf{d}y+\frac{1}{9\kappa (u^2+\epsilon v^2)+1}\mathsf{d}v\\
 \theta^4=\frac{1}{\kappa (x^2+\epsilon y^2)+1}\mathsf{d}y-\frac{\epsilon}{9\kappa (u^2+\epsilon v^2)+1}\mathsf{d}v\end{array}&\quad\mbox{where}\quad
 \begin{array}{c|ccc} & \mathsf{M6S}.1 & \mathsf{M6S}.2 & \mathsf{M6S}.3\\ \hline (\kappa,\epsilon) & (1,-1) & (-1,1) & (1,1) \end{array}
 \end{array} \\
 
\hline
 \mathsf{M6N} & 
 \begin{array}{l@{\,}l} \theta^1=\mathsf{d}v,\quad \theta^2=\mathsf{d}x,\quad \theta^3=2e^{2v}\mathsf{d}x-2e^{2v}u\mathsf{d}y+\mathsf{d}u-u\mathsf{d}v, \quad \theta^4=\mathsf{d}y \end{array}   \\ \hline
 \end{array}
 \]
 \caption{Here $(\theta^1,\theta^2,\theta^3,\theta^4)$ is a positively oriented coframe in $(x,y,u,v)$-space  such that a representative metric is of the form $\sfg=2(\theta^1\theta^2+\theta^3\theta^4)$.}
 \end{table}

 The  $\mathsf{M9}$ model is represented by the split signature Ricci-flat pp-wave metric $\sfg=\mathsf{d}x\mathsf{d}u+\mathsf{d}y\mathsf{d}v+x^2\mathsf{d}v^2$. The $\mathsf{M8}$ model is realized by the $\mathrm{SL}(3,\mathbb{C})$-invariant holomorphic metric
 \begin{align}
\sfg= \frac{z^2\mathsf{d}z^1\mathsf{d}w^2-(w^1+w^2)\mathsf{d}z^1\mathsf{d}z^2+z^1\mathsf{d}w^1\mathsf{d}z^2+\mathsf{d}w^1\mathsf{d}w^2}{(w^1+w^2+z^1z^2)^2},
 \end{align}
from which we obtain the following representative metrics of the real forms $\mathsf{M8.1}$ and $\mathsf{M8.2}$:
 \begin{itemize}
 \item  $z^1=u$, $z^2=x$, $w^2=-y$, $w^1=v$  gives a metric realizing the $\mathsf{M8.1}$ model, which is a para-K\"ahler-Einstein  metric on the symmetric space $\mathrm{SL}(3,\bbR)/\mathrm{GL}(2,\bbR)$ and was studied geometrically under the name {\sl dancing metric} in \cite{BLN2018}, and
 \item  $z^1=x+i y$, $z^2=x-i y$, $w^2=u-i v$, $w^1=u+i v$ gives a metric realizing the  $\mathsf{M8.2}$ model, which is a pseudo-K\"ahler-Einstein metric on the symmetric space $\mathrm{SU}(1,2)/\mathrm{GL}(2,\bbR)$.
 \end{itemize}
 The real forms of the $\mathsf{M6S}$ model can all be realized  via variants of the rolling construction. The conformal structures can  be represented by  products of the form $(\Sigma_1\times \Sigma_2,[\sfg_1\times -\sfg_2])$, where $(\Sigma_1,\sfg_1)$ and $(\Sigma_2,\sfg_2)$ are $2$-dimensional constant curvature spaces of the same signature whose curvatures have ratio $9:1$. In the $\mathsf{M6S.3}$ case, these are two spheres, in the  $\mathsf{M6S.2}$ case  two hyperbolic spaces, and in the $\mathsf{M6S.1}$ case these are two de Sitter spaces of signature $(1,1)$.


Representative metrics $\sfg$  for the $\mathsf{M7_a}$ models are given in Table \ref{table-coframe}. Real forms exist for $\sfa\in\bbR\cup i\bbR$, but since $\pm \sfa$ yield the same complex model, it suffices to restrict to $\sfa\in\bbR_{\geq 0}\cup i\bbR_{\geq 0}$.  This is equivalently encoded by the pair $(r,\epsilon)$, where $r=\vert \sfa\vert \geq 0$ and $\epsilon=\mathrm{sgn}(\sfa^2)$ when $\sfa \neq 0$, while $\epsilon = \pm 1$ when $\sfa = 0$.  Any two metrics corresponding to $(r,\epsilon)$ and $(r',\epsilon')$ such that $(r,\epsilon)\neq (r',\epsilon')$ are conformally inequivalent. The metrics are of Petrov type $\sfN.\sfN$ {\em except} if $\sfa^2 = \frac{4}{3}$ (or $(r,\epsilon) = (\frac{2}{\sqrt{3}}, 1)$), in which  case the Petrov type is $\sfN.\sfO$. In fact, in this special case one can choose coordinates $(X,Y,U,V)$ such that the metric is conformal to the pp-wave metric
 \begin{align} \label{E:N7-HF}
 \mathsf{d}X\mathsf{d}U+\mathsf{d}Y\mathsf{d}V-\tfrac{3}{2}X^{1/2}\mathsf{d}V^2.
 \end{align}
All $\mathsf{M7_a}$ models are conformally Ricci-flat. The Ricci-flat representative metric $\sfg_E$ is related to $\sfg$ via
 \begin{align} \label{E:M7-Einstein}
 \sfg_E= \  \frac{\mathrm{sech}^2(\frac{u}{2}\sqrt{20\epsilon+r^2})}{\mathrm{exp}(3r\epsilon+2v)}\ \sfg.
 \end{align}
It would be nice to have realizations of the $\mathsf{M7_a}$ and $\mathsf{M6N}$ models in terms of non-holonomic systems.

Given a metric $\sfg=2(\theta^1\theta^2+\theta^3\theta^4)$ expressed in a null coframe, the associated SD twistor distribution  is given by the common kernel of the $1$-forms $\omega^1=\theta^2+\xi\theta^4$ and $\omega^2=\theta^3-\xi\theta^1$ as in \eqref{om12} and $\omega^3$, which is uniquely determined by the equations \eqref{DfromH}. By Lemma \ref{lemm-Weyl}, the Petrov type of the SD Weyl tensor can be alternatively obtained by calculating $\mathsf{d}\omega^3\wedge\omega^1\wedge\omega^2\wedge\omega^3$, see \eqref{W(xi)}, and analyzing the root type of $\mathcal{W}^+(\xi)$, as illustrated in the following example. (The ASD twistor distribution and the Petrov type of the ASD Weyl tensor can be obtained in an analogous manner.) 
\begin{example}
Let us consider the $\mathsf{M6N}$ model.  Here, a null coframe is given by 
 \begin{align}
 \begin{array}{l@{\,}l} 
 \theta^1=\mathsf{d}x,\quad \theta^2=\mathsf{d}v,\quad \theta^3=2e^{2v}\mathsf{d}x-2e^{2v}u\mathsf{d}y+\mathsf{d}u-u\mathsf{d}v, \quad \theta^4=\mathsf{d}y.
 \end{array}
 \end{align} 
Then the twistor distribution is of the form $\cD =\mathsf{ker}(\omega^1,\omega^2,\omega^3)$, where
 \begin{align}
 \omega^1 &= 2e^{2v}\xi\mathsf{d}x-2e^{2v}\xi u\mathsf{d}y+\xi\mathsf{d}u-(\xi u-1)\mathsf{d}v,\quad
 \omega^2 = \mathsf{d}y-\xi\mathsf{d}x,\\
 \omega^3 &= \mathsf{d}\xi-4e^{2v}\xi^2(\xi u-1)\mathsf{d}x.
 \end{align}
One confirms that the SD Weyl tensor $\cW^+$ has Petrov type $\mathsf{III}$ from
 \begin{align}
 \mathsf{d}\omega^3\wedge\omega^1\wedge\omega^2\wedge\omega^3=-12\xi^3 e^{2v}(\xi u-1)\mathsf{d}x\wedge\mathsf{d}y\wedge\mathsf{d}u\wedge\mathsf{d}v\wedge\mathsf{d}\xi.
 \end{align}
One can also directly compute that the symmetries of the conformal structure are given by
 \begin{align}
 &E= -xy\partial_{x}-y^2\partial_{y}+(yu-x)\partial_{u}+y\partial_{v},\quad F=\partial_{y},\quad
H=x\partial_{x}+2y\partial_{y}-u\partial_{u}-\partial_{v},\\\
 & R=-x\partial_{x}+u\partial_{u},\quad
  S=y\partial_{x}+\partial_{u},\quad
  T=\partial_{x}
 \end{align}
 with non-trivial Lie brackets   
 \begin{align}
  &[E,F]= H,\quad [H,E]=2E,\quad
   [H,F]=-2F,\quad  [H,S]=S,\quad
 [H,T]=-T,\\&  [E,T]=S,\quad
 [F,S]=T,\quad [R,S]=S,\quad
[R,T]=T, 
 \end{align}
 which confirms that the symmetry algebra is
$\mathfrak{aff}(2) \cong \fgl(2,\bbR) \ltimes \bbR^2$. 
\end{example}

 \subsection{Half-flatness and pairs of 2nd order ODE}
  
  
 Suppose that $(M,[\sfg])$ is SD, i.e.\ $\cW^- = 0$.  This is equivalent to integrability of the ASD twistor XXO-structure $(\ell,\cD)$ on $\bbT^-(M)$.  As described in Section \ref{S:XXO}, such $(\ell,\cD)$ can be encoded by a 2nd order ODE pair \eqref{E:2ODE}, or equivalently as \eqref{E:2ODE-LD}.

 \begin{example}[$\mathsf{M7}_\sfa$, $\sfa^2 = \frac{4}{3}$]  On $(x,y,u,v)$-space $M$, consider the half-flat conformal structure $[\sfg]$ of Petrov type $\sfN.\sfO$ given by $\sfg = 2(\theta^1 \theta^2 + \theta^3 \theta^4)$ from \eqref{E:N7-HF}, where
 \begin{align}
 \theta^1 = dx, \quad \theta^2 = du, \quad \theta^3 = dy - \frac{3}{2} x^{1/2} dv, \quad \theta^4 = dv.
 \end{align}
 Let $e_1,e_2,e_3,e_4$ be the dual framing, and an ASD totally null 2-plane is given by $\langle e_4 - \xi e_1, e_2 + \xi e_3 \rangle$.  Lifting to $N = \bbT^-(M)$, we get an integrable XXO-structure $(\ell,\cD)$ given by
 \begin{align}
 \ell = \langle \partial_\xi \rangle, \quad \cD = \left\langle \partial_u + \xi \partial_y, \partial_v - \xi \partial_x + \frac{3}{2} \sqrt{x} \,\partial_y \right\rangle.
 \end{align}
 Invariants for $\cD$ can be easily found using {\tt Maple}:
 \begin{verbatim}
 restart: with(DifferentialGeometry): with(GroupActions):
 DGsetup([x,y,u,v,xi],N):
 dist:=evalDG([D_u+xi*D_y,D_v-xi*D_x+3/2*sqrt(x)*D_y]):
 InvariantGeometricObjectFields(dist,[1],output="list");
 \end{verbatim}
 This yields the invariants $\xi, \,\, v+\frac{x}{\xi}, \,\, u - \frac{x^{3/2}}{\xi^2} - \frac{y}{\xi}$.
 Defining the following coordinate system
 \begin{align}
 (T,X,Y,P,Q) =\left( \frac{1}{\xi}, \,\, v+\frac{x}{\xi},\,\, -\frac{u}{2} + \frac{x^{3/2}}{2\xi^2} + \frac{y}{2\xi}, \,\,
x, \,\, \frac{x^{3/2}}{\xi} +\frac{y}{2} \right),
 \end{align}
 we can read off the 2nd order ODE pair from $(\ell,\cD)$ in these new coordinates:
 \begin{align}
 \ell = \left\langle \partial_T + P\partial_X + Q \partial_Y + P^{3/2} \partial_Q \right\rangle, \quad \cD = \langle \partial_P, \partial_Q \rangle.
 \end{align}
 \end{example}

 We do this similarly for all other SD conformal structures in our classification to obtain Table \ref{F:2ODE}.    The $\mathsf{M9}$ and $\mathsf{M8.1}$ cases were given in \cite{CDT2013}, while the $\mathsf{M8.2}$ case was only recently stated in \cite{KM2023}.

 \begin{table}[H]
 \[
 \begin{array}{|c|c|c|c|} \hline
 \mbox{Model} & \mbox{Petrov type} & \mbox{2nd order ODE pair} \\ \hline\hline
 \mathsf{M9} & \sfN.\sfO & x'' = 0, \quad y'' = (x')^3\\
 \mathsf{M8.1} & \sfD^+.\sfO & x'' = 0, \quad y'' = \frac{2x' (y')^2}{yx'-1} \\
 \mathsf{M8.2} & \sfD^-.\sfO & x'' = \frac{((y')^2+1)(x'y' - yy' - t)}{y't + x'-y}, \quad y'' = \frac{((y')^2+1)^2}{y't + x'-y} \\
 \mathsf{M7_a},\, \sfa^2=\frac{4}{3} & \sfN.\sfO & x'' = 0, \quad y'' = (x')^{3/2} \\
 \mathsf{M6N} & \mathsf{III}.\sfO & x'' = 0, \quad y'' = \frac{((3y^2+4y')f-3y)(y')^2}{2(y(y^2+2 y')f-y^2-y')}, \quad f = \frac{-y\pm \sqrt{y^2+4y'}}{2y'}
\\ \hline
 \end{array}
 \]
 \caption{2nd order ODE pairs associated to half-flat conformal structures}
 \label{F:2ODE}
 \end{table}

 
 \subsection{Cartan-theoretic descriptions}
 \label{S:CarTh}
 
  
 Beyond coordinate and Lie-theoretic descriptions, an equivalent manner of presenting homogeneous structures what we refer to as a {\sl Cartan-theoretic} description \cite{CS2009, The2022}.  We will give such descriptions for our homogeneous 4D-split conformal models in our classification.  Two key features of such descriptions is the ability to efficiently compute the conformal holonomy of these models and algebraically assess the existence of an Einstein metric in the conformal structure.  This will be illustrated in the next subsection.
 
 Any (regular, normal) Cartan geometry $(\cG \to M, \omega)$ of type $(\widetilde\fg,\widetilde{Q})$ that is homogeneous for the action of a Lie group $F$ can be encoded by:
  
  \begin{defn} \label{D:alg-model}
 A {\sl Cartan-theoretic model $(\ff;\widetilde\fg,\widetilde\fq)$} is a Lie algebra $(\ff,[\cdot,\cdot]_\ff)$ such that:
 \begin{enumerate}
 \item[\rm (C1)] $\ff \subset \widetilde\fg$ is a filtered subspace, with filtrands $\ff^i := \ff \cap \widetilde\fg^i$, and $\ff / \ff^0 \cong \widetilde\fg / \widetilde\fq$.
 \item[\rm (C2)] $\ff^0$ inserts trivially into the {\sl curvature} $\kappa(x,y) := [x,y] - [x,y]_\ff$.
 \item[\rm (C3)] $\kappa \in \bigwedge^2(\ff/\ff^0)^* \otimes \widetilde\fg \cong \bigwedge^2 (\widetilde\fg/\widetilde\fq)^* \otimes \widetilde\fg$ is regular and normal, i.e.\ $\kappa \in \ker(\partial^*)^1$.
 \end{enumerate}
 The {\sl harmonic curvature} is $\kappa_H := \kappa \,\mod \im(\partial^*) \in H_2(\widetilde\fq_+,\widetilde\fg)^1$.
 \end{defn}
 
 We note that $\widetilde{Q}$ acts via the adjoint action on $\wfg$ and this induces an action on $\ff$, i.e. $\ff \mapsto \Ad_q \ff$ for $q \in \widetilde{Q}$.  Such Cartan-theoretic descriptions are regarded as equivalent.
 
 For 4D split-conformal structures, we should take $\wfg = \fsl(4,\bbR)$, but for our purposes in Section \ref{S:CHE}, it will instead be sufficient to work with complexified models, so we take $\wfg = \fsl(4,\bbC)$.  We have $\kappa_H$ is valued in the SD and ASD parts of the Weyl tensor module.  In Table \ref{F:harmonic}, we presented bases $\phi_0,...,\phi_4$ and $\psi_0,...,\psi_4$ for these respective modules, realized as harmonic 2-cochains (elements of $\Lambda^2 \wfg_-^* \otimes \wfg$).  This is the homogeneity $+2$ part of $\kappa$ and in principle $\kappa$ could have a homogeneity $+3$ part, corresponding to the so-called Cotton tensor.  While we could have generated a basis for this homogeneity $+3$ module (by applying raising operators to the basis above), this turns out to be unnecessary for our purposes.  Namely, in Table \ref{F:CT-realize} we present Cartan-theoretic realizations of all of our models, and in these realizations $\kappa$ is ``purely harmonic''.   (Indeed, $\kappa_H$ can always be viewed as the lowest homogeneous component of $\kappa$, see \cite[Theorem 3.1.12]{CS2009}.)
 
 These descriptions were found with the aid of {\tt Maple}, mostly via appropriately setting undetermined coefficients, and by normalizing the form of $\kappa_H$ to make them as simple as possible.
 
 \begin{table}[H]
 \[
 \begin{array}{|c|c|c|c|} \hline
 \mbox{Label} & \mbox{$\fsl(4,\bbC)$-matrix $\bT$} & \mbox{Curvature $\kappa$} & \mbox{Petrov type}\\ \hline\hline
 \mathsf{M9} & \begin{psm}
 2 z & 0 & 0 & 0\\
 a_3 & 0 & 0 & 0\\
 a_4 & a_1 & -z + h & s\\
 a_5 & a_2 & t & -z - h
 \end{psm} & \phi_0 & \sfN.\sfO\\ \hline
 \mathsf{M8} & \begin{psm}
 s & 0 & \frac{a_2}{2} & -\frac{a_1}{2}\\
 0 & -s & -\frac{a_4}{2} & \frac{a_3}{2}\\
 a_3 & a_1 & t_1 & t_2\\
 a_4 & a_2 & t_3 & -t_1
 \end{psm} & \begin{array}{l}
 \phi_2
 \end{array} & \sfD.\sfO\\ \hline
 \mathsf{M7}_\sfa & \begin{array}{c} \begin{psm}
 z - g a_1& 0 & 0 & 0 \\ 
 a_3 & g a_1 & f a_1 & 0\\ 
 a_4 & a_1 & g a_1 & 0\\
 a_5 & a_2 & t  & -z-ga_1  \end{psm}\\
 \mbox{where } (f,g) = (\frac{\sfa^2}{4} + 5, -\frac{\sfa}{2})
 \end{array}
  &
 \begin{array}{c}
  \phi_0 + 12(\sfa^2 - \frac{4}{3}) \psi_0
 \end{array} & \begin{cases} \sfN.\sfN, & \sfa^2 \neq \frac{4}{3}\\ \sfN.\sfO, & \sfa^2 = \frac{4}{3} \end{cases}\\ \hline
 \mathsf{M6S} & 
 \begin{array}{c}
 \begin{psm}
 s_1 & 0 & \frac{11}{6} a_2 & \frac{19}{6} a_1 \\
 0 & -s_1 & \frac{19}{6} a_4 & \frac{11}{6} a_3 \\
 a_3 & a_1 & s_2 & 0 \\
 a_4 & a_2 & 0 & -s_2
 \end{psm}
 \end{array} &
 -\frac{4}{3} \phi_2 + \frac{4}{3} \psi_2 & \sfD.\sfD\\ \hline
 \mathsf{M6N} & \begin{psm}
 \frac{3z}{2} - \frac{b_1}{2} & 0 & 0 & 0\\
 \frac{b_3}{4} & -\frac{z}{2} + \frac{b_1}{2} & -\frac{5b_1}{48} - 5 b_2 & 5 b_1\\
 b_3 & b_1 & -\frac{b_1}{6} - 16 b_2 - \frac{z}{2} & 16 b_1\\
 b_4 & b_2 & t & \frac{b_1}{6} + 16 b_2 - \frac{z}{2}
 \end{psm} & \phi_1 & \mathsf{III}.\sfO\\ \hline
  \end{array}
 \]
 \caption{Cartan-theoretic realizations of 4D split-conformal models in our classification}
 \label{F:CT-realize}
 \end{table}
 
 \begin{table}[H]
 \[
 \begin{array}{|c|l|c|c|} \hline
 \mbox{Label} & \mbox{Embedding} \\ \hline\hline
 \mathsf{M9} &   \begin{array}{@{}ll}
 \begin{array}{l@{}l}
 (X_1,X_2,X_3) &\mapsto (\bT_{a_1}, \bT_{a_2}, \frac{1}{2}\bT_{a_3})\\
 (X_4,X_5) & \mapsto (\frac{1}{6} \bT_{a_4}, -\frac{1}{6} \bT_{a_5})\\
 (\sfZ_1, e_{01}, \sfH, f_{01}) &\mapsto (\bT_z, \bT_s, \bT_h, \bT_t)\\
 \end{array}
 \end{array} \\ \hline
 \mathsf{M8} & 
 \begin{array}{ll}
 (X_1,X_2,X_3) & \mapsto (\bT_{a_1} + \bT_{a_3}, 2(\bT_{a_2} + \bT_{a_4}),3 \bT_s)\\
 (X_4,X_5) & \mapsto (\bT_{a_3} - \bT_{a_1}, 2(\bT_{a_2} - \bT_{a_4}))\\
 (\sfH,e_{01},f_{01}) &\mapsto (\bT_{t_1}, \frac{1}{2} \bT_{t_2}, 2 \bT_{t_3})
 \end{array} \\ \hline
 \mathsf{M7}_\sfa &   \begin{array}{ll} 
 (X_1,X_2,X_3) &\mapsto (\bT_{a_1} + \frac{3\sfa}{2} \bT_z,\bT_{a_2} + \frac{\sfa}{2} \bT_t, \frac{1}{2}\bT_{a_3} + \bT_t)\\
 (X_4,X_5) &\mapsto (\frac{1}{6} \bT_{a_4} - \frac{\sfa}{12} \bT_{a_3} + \frac{1}{3} \bT_{a_2} + \frac{5\sfa}{6} \bT_t, -\frac{1}{6} \bT_{a_5})\\
 (\sfZ_2, f_{01}) &\mapsto (\bT_z, \bT_t)
 \end{array} \\ \hline
 \mathsf{M6S} & \begin{array}{ll} 
 (X_1,X_2,X_3) &\mapsto (\bT_{a_1} + \bT_{a_3}, \bT_{a_2} + \bT_{a_4}, -2 \bT_{s_1} - \bT_{s_2})\\
 (X_4, X_5, \sfH) &\mapsto (\frac{1}{3} \bT_{a_1} - \bT_{a_3}, -\bT_{a_2} + \frac{1}{3} \bT_{a_4}, \bT_{s_2})
 \end{array} \\ \hline
 \mathsf{M6N} & \begin{array}{ll}
 X_1 &\mapsto -\frac{24}{5} \bT_{b_1} + \frac{3}{40} \bT_{b_2} + 32 \bT_{b_3} - \frac{3}{160} \bT_t\\
 X_2 &\mapsto \frac{1}{16} \bT_{b_2} - \frac{5}{12} \bT_{b_4} + \frac{1}{16} \bT_t\\
 X_3 &\mapsto -\frac{3}{16} \bT_{b_2} + \frac{5}{4} \bT_{b_4} + \frac{3}{2} \bT_{b_5} - \frac{7}{64} \bT_t\\
 X_4 &\mapsto \frac{24}{5} \bT_{b_1} - \frac{1}{5} \bT_{b_2} + \frac{5}{6} \bT_{b_4} + 3 \bT_{b_5} - \frac{13}{240} \bT_t\\
 X_5 &\mapsto \frac{1}{16} \bT_{b_2} + \frac{1}{96} \bT_t\\
 f_{01} &\mapsto -\frac{5}{384} \bT_t
 \end{array} \\ \hline
 \end{array}
 \]
 \caption{Isomorphisms from Lie-theoretic to Cartan-theoretic realizations}
 \end{table}   
  
 
 \subsection{Conformal holonomy and Einstein representatives}
 \label{S:CHE}
 

 Given a Cartan-theoretic realization $(\ff;\wfg,\wfq)$, the infinitesimal holonomy algebra $\hol$ of the associated homogeneous Cartan geometry is easily computed via the following increasing sequence of subspaces of $\widetilde\fg$:
 \begin{align} \label{E:hol}
 \hol^0 := \langle \kappa(x,y) : x,y \in \ff \rangle, \qquad
 \hol^i := \hol^{i-1} + [\ff,\hol^{i-1}], \,\,\forall i \geq 1.
\end{align}
Since $\dim(\widetilde\fg)$ is finite, the sequence stabilizes to some $\hol^\infty$, and we have $\hol = \hol^\infty$.  More precisely, since we are working with the model data $(\wfg,\wfq)$ for conformal geometry, the above computes its {\sl conformal holonomy}.

\begin{example}[$\mathsf{M7_a}$]  From Table \ref{F:CT-realize}, we begin with $\kappa = \phi_0 + 12(\sfa^2 - \frac{4}{3}) \psi_0$, where $\phi_0$ and $\psi_0$ were defined in Table \ref{F:harmonic}.  If $\sfa^2 \neq \frac{4}{3}$, then $\hol^0 = \langle E_{21}, E_{43} \rangle$, and we calculate $\hol = \hol^1 = \langle E_{31}, E_{41}, E_{42} \rangle \mod \hol^0$.  If $\sfa^2 = \frac{4}{3}$, then $\hol^0 = \langle E_{21} \rangle$, and $\hol = \hol^1 = \langle E_{31}, E_{41} \rangle \mod \hol^0$.
\end{example}

 Notably, $\hol \cong \fsp(4,\bbC) \cong \fso(5,\bbC)$ in the $\mathsf{M8}$ case.  We summarize:

\begin{theorem}
The (complexified) conformal holonomy algebras for the models in our classification are given in Table \ref{F:CHol}.
\end{theorem}

Given any conformal structure $[\sfg]$, and a metric representative $\sfg \in [\sfg]$, one may ask if $\sigma^{-2} \sfg$ is an Einstein metric for some choice of positive function $\sigma$.  There is a now well-established reformulation of this problem so that $\sigma$ (in fact a density) must equivalently lie in the kernel of a certain second-order linear partial differential operator called the {\sl almost-Einstein equation}.  (``Almost'' refers to the fact that PDE solutions $\sigma$ can have a non-trivial zero set.)  Moreover, $\sigma$ can be prolonged to a section $s$ of the {\sl standard tractor bundle} of the Cartan bundle of the conformal geometry, and  it is parallel with respect to the standard tractor connection $\nabla$.  For details, we refer to \cite{BEG1994}.  The infinitesimal holonomy algebra of $\nabla$ is precisely the conformal holonomy algebra discussed above, and this obstructs the existence of Einstein scales $\sigma$.

For homogeneous structures, the existence of such parallel objects can be assessed purely algebraically.  To do so, we make use of the isomorphism $\fsl(4,\bbR) \cong \fso(3,3)$ and that the standard $\fso(3,3)$-rep becomes the $\fsl(4,\bbR)$-rep $\Lambda^2 \bbR^4$.  Of interest are the $\hol$-annihilated elements in this representation.  The complexified results are given in Table \ref{F:CHol}.

 \begin{table}[H]
 \[
 \begin{array}{|c|c|c|} \hline
 \mbox{Label} & \mathfrak{hol}, \mbox{ represented on $\bbC^4 = \langle e_1, e_2, e_3, e_4 \rangle$} & (\Lambda^2 \bbC^4)^\mathfrak{hol}\\ \hline\hline
 \mathsf{M9} & \begin{psm} 
 0 & 0 & 0 & 0\\
 * & 0 & 0 & 0\\
 * & 0 & 0 & 0\\
 * & 0 & 0 & 0
 \end{psm} & \begin{array}{l} e_2 \wedge e_3,\\ e_2 \wedge e_4,\\ e_3 \wedge e_4 \end{array}\\ \hline
 \mathsf{M8} & \fsp(4,\bbC): \,\,\begin{psm}
 s_1 & s_2 & \frac{c_4}{2} & -\frac{c_2}{2}\\
 s_3 & -s_1 & -\frac{c_3}{2} & \frac{c_1}{2}\\
 c_1 & c_2 & t_1 & t_2\\
 c_3 & c_4 & t_3 & -t_1\\
 \end{psm} & e_1 \wedge e_2 - 2 e_3 \wedge e_4 \\ \hline
 \multirow{4}{*}{$\mathsf{M7}_\sfa$} & \sfa^2 \neq \frac{4}{3}: \quad\begin{psm} 
 0 & 0 & 0 & 0\\
 * & 0 & 0 & 0\\
 * & 0 & 0 & 0\\
 * & * & * & 0
 \end{psm} & \begin{array}{l} e_2 \wedge e_4,\\ e_3 \wedge e_4 \end{array}\\ \cline{2-3}
 & \sfa^2 = \frac{4}{3}: \quad \begin{psm} 
 0 & 0 & 0 & 0\\
 * & 0 & 0 & 0\\
 * & 0 & 0 & 0\\
 * & 0 & 0 & 0
 \end{psm} & \begin{array}{l} e_2 \wedge e_3, \\ e_2 \wedge e_4,\\ e_3 \wedge e_4 \end{array} \\ \hline
 \mathsf{M6S} & \fsl(4,\bbC) & 0\\ \hline
 \mathsf{M6N} & \begin{psm}
 * & 0 & 0 & 0 \\
 * & * & * & * \\
 * & * & * & * \\
 * & * & * & *
 \end{psm} & 0\\ \hline
 \end{array}
 \]
 \caption{Complexified conformal holonomy algebras and $\hol$-annihilated elements}
 \label{F:CHol}
 \end{table}

 We find that there are no non-trivial $\hol$-annihilated elements in the $\mathsf{M6S}$ and $\mathsf{M6N}$ cases, so:

 \begin{theorem} \label{T:AE}
 The conformal structures associated to $\mathsf{M9},\mathsf{M8},\mathsf{M7}$ admit Einstein representatives, while $\mathsf{M6S}$ and $\mathsf{M6N}$ do not.
 \end{theorem}
 
 In particular, the $\mathsf{M9},\mathsf{M8.1}, \mathsf{M8.2}$ models in Table \ref{F:main} are Einstein, as are \eqref{E:M7-Einstein} in the $\mathsf{M7}$ cases.

\section*{Acknowledgments}

 We would like to thank Mike Eastwood for many discussions on the topic, and we acknowledge the use of the {\tt DifferentialGeometry} package \cite{AT-DG} in {\tt Maple}.  The research was funded from the Norwegian Financial Mechanism 2014–2021 with project registration number 2019/34/H/ST1/00636.  DT also received funding from the Troms\o{} Research Foundation (project ``Pure Mathematics in Norway''), the UiT Aurora project MASCOT, and this article/publication is based upon work from COST Action CaLISTA CA21109 supported by COST (European Cooperation in Science and Technology), \href{https://www.cost.eu}{https://www.cost.eu}.
 
  \appendix  

 \section{The obstruction to descending to a conformal structure}
 \label{S:torsion}
 
 
 In this section, we present a formula for the homogeneity 2 torsion $\cT$ for an XXO structure that completely locally obstructs the descent to a conformal structure.

 Let $(N;\ell,\cD)$ be an XXO structure.  Consider a framing of $TN$ adapted to $(\ell,\cD)$, namely 
 \begin{align}
 X_1, \quad X_2, \quad L, \quad X_1' := [L,X_1],\quad X_2' := [L,X_2],
 \end{align}
 where $\cD = \langle X_1, X_2 \rangle$ and $\ell = \langle L \rangle$.   Since $L$ is nowhere vanishing, we can always introduce coordinates to rectify it.  (So below, there is no loss of generality in thinking of $L$ as $\partial_t$.)  Write
 \begin{align}
 X_i'' := [L,X_i'] = \alpha_i{}^j X_j' + \beta_i{}^jX_j + \gamma_i L,
 \end{align}
 or more succinctly,
 \begin{align} \label{E:Xpp}
 X'' = \alpha X' + \beta X + \gamma L. 
 \end{align}
 This resembles a linear 2nd order ODE system, and techniques from that study can be used here.  Namely, a first step is to bring it to {\sl Laguerre--Forsyth canonical form}.
 
 Let $\bar{X} := A X$ be a local frame change in $\cD$, so $A$ is $\GL(2,\bbR)$-valued.  Then,
 \begin{align}
 \bar{X}' &:= [L,\bar{X}] = L(A) X + A X'\\
 \bar{X}'' &:= [L,\bar{X}'] = L^2(A) X + 2L(A) X'+ A X''\\
 &= L^2(A) X + 2L(A) X'+ A (\alpha X' + \beta X + \gamma L) \nonumber\\
 &= (2L(A) + A \alpha  ) X' + (L^2(A) + A \beta ) X + A  \gamma L \nonumber\\
 &= (2L(A) + A \alpha  ) A^{-1} (\bar{X}' - L(A) A^{-1} \bar{X})+ (L^2(A) + A \beta ) A^{-1} \bar{X} + A  \gamma L \nonumber
 \end{align}
 The coefficient in front of $\bar{X}'$ is $\bar\alpha :=  (2L(A) + A \alpha  ) A^{-1}$.  Normalizing $\bar\alpha = 0$ is equivalent to solving a first-order matrix ODE system, $L(A) = -\frac{1}{2} A \alpha$, which can always be done.  Substitution yields:
 \begin{align}
 \bar{X}'' 
 &= \left(-\frac{1}{2} L(A)\alpha - \frac{1}{2} A L(\alpha) + A \beta \right) A^{-1} \bar{X} + A  \gamma L\\
 &= A\left(\frac{1}{4} \alpha^2 - \frac{1}{2} L(\alpha) + \beta \right) A^{-1} \bar{X} + A  \gamma L \nonumber
 \end{align}
 
 Now consider a local frame change $\bar{L} := \lambda L$ in $\ell$, where $\lambda$ is nowhere vanishing. Modulo $\ell$,
 \begin{align}
 \dot{\bar{X}} &:= [\bar{L},\bar{X}] \equiv \lambda \bar{X}'\\
 \ddot{\bar{X}} &:= [\bar{L},\dot{\bar{X}}] \equiv [\bar{L},\lambda \bar{X}'] 
 = \bar{L}(\lambda) \bar{X}' + \lambda [\bar{L},\bar{X}']
 \equiv \lambda L(\lambda) \lambda^{-1} \dot{\bar{X}} + \lambda^2 \bar{X}''\\
 &\equiv \lambda L(\lambda) \lambda^{-1} \dot{\bar{X}} + \lambda^2 A\left(\frac{1}{4} \alpha^2 - \frac{1}{2} L(\alpha) + \beta \right) A^{-1} \bar{X} \nonumber
 \end{align}
 We maintain the analogous earlier normalization of having no $\dot{\bar{X}}$ coefficient by using $\lambda$ satisfying $L(\lambda) = 0$.  The coefficient of $\bar{X}$ now yields an element of $\tEnd(\cD)$ (so tensorial with respect to frame changes on $\cD$), and of weight 2 with respect to rescalings along $\ell$.  Its trace-free part is:
 \begin{align} \label{E:XXOtorsion}
 \cT := \Xi - \frac{1}{2} \tr(\Xi) I, \qbox{where}  \Xi := \frac{1}{4} \alpha^2 - \frac{1}{2} L(\alpha) + \beta.
 \end{align}
 
 From Table \ref{F:KH}, the homogeneity 2 XXO harmonic curvature is a map $\cT : \ell \times TN / \cH \to \cD$ that is trace-free, when regarded as $\cT \in (\ell^*)^{\otimes 2} \otimes \tEnd(\cD)$ in view of the natural identification $TN / \cH \cong \ell \otimes \cD$.  There are no other components of this type in the homogeneity 2 part of the curvature.  Since $\cT$ in \eqref{E:XXOtorsion} is naturally covariant for the XXO structure and of the correct type, it must be the claimed harmonic curvature.
 
 \begin{example} \label{X:M6N-torsion}
 Let us calculate $\cT$ in the $\mathsf{M6N}$ case to establish \eqref{E:M6Ntorsion}.  Here, we work with Lie algebraic data $(\ff / \ff^0;\ell,\cD)$ with 
 \begin{align}
 \ff^0 = \langle f_{01} \rangle,\quad
 \ff^{-1} / \ff^0 = \ell \op \cD, \quad 
 \ell \equiv \langle X_3 + a X_2 \rangle, \quad 
 \cD \equiv \langle X_1, X_2 \rangle \quad \mod \ff^0,
 \end{align}
 and $X_1,X_2, X_3$ specified in \eqref{E:M6N-123}.  Take $L = X_3 + a X_2$ (labelled $X$ in the derivation above).
 
 Let $X_i' = [L,X_1]$ and $X_i'' = [L,X_i']$.  A direct computation leads to:
 \begin{align}
 X_1'' &= 2a(3-a) X_1 - 2a(a^2 - 4a + 6) X_2 + 2a(a-3) L - 3(a-2) X_1' - 2a(a-1) X_2'\\
 X_2'' &= (a-6) X_2' + 12 a f_{01}
 \end{align}
 Consequently, in \eqref{E:Xpp} we have the constant matrices:
 \begin{align}
 \alpha = \begin{pmatrix}
 - 3(a-2) & - 2a(a-1)\\
 0 & a-6
 \end{pmatrix}, \quad
 \beta = \begin{pmatrix}
 2a(3-a) & -2a(a^2 - 4a + 6)\\
 0 & 0 
 \end{pmatrix}.
 \end{align}
 Then we obtain \eqref{E:M6Ntorsion} by computing \eqref{E:XXOtorsion}:
 \begin{align}
 \Xi = \frac{1}{4} \alpha^2 + \beta = \begin{pmatrix}
 \frac{1}{4} (a-6)^2 & -a(a-3)(a-4)\\
 0 & \frac{1}{4} (a-6)^2
 \end{pmatrix}, \quad  \cT = \begin{pmatrix}
 0 & -a(a-3)(a-4)\\
 0 & 0
 \end{pmatrix}.
 \end{align}
 \end{example}


 \section{An approach based on Cartan's equivalence method}
 

In principle, there are also other methods that are well-suited to solve the classification problem of homogeneous structures presented in this article. In this Appendix, we outline a powerful approach  based on Cartan's equivalence and reduction methods. On the one hand, it quickly gives rise to new invariants that distinguish locally inequivalent structures and is also suitable for the study of non-homogeneous structures. On the other hand, a full execution of this approach  requires careful book-keeping and  non-trivial, lengthy computations that would  be ultimately difficult to present. 
 
 \subsection{The invariant coframe} 
 We first show how to invariantly associate a coframe in dimension $9$ to an XXO-structure with flat, maximally non-integrable rank $2$ distribution $\mathcal{D}$. (In fact we could drop the requirement that $\mathcal{D}$ be flat, but we will restrict to this case to simplify the presentation.)

 Consider the flat $(2,3,5)$ distribution on $N$ with coordinates $(x,y,p,q,z)$  as in Example \ref{ex-235+line}.
Introduce the coframe $(\omega^i)$, 
 \begin{align}
 \omega^1=-(\mathsf{d}z-\tfrac{1}{2} q^2\mathsf{d}x)+q(\mathsf{d}p-q\mathsf{d}x),\,\,\,
 \omega^2=\mathsf{d}y -p\mathsf{d}x,\,\,\,
 \omega^3=\mathsf{d}p-q\mathsf{d}x,\,\,\,
 \omega^4=\mathsf{d}q,\,\,\,
 \omega^5=\mathsf{d}x.
 \end{align}
Then  $\cD=\mathsf{ker}(\omega^1,\omega^2, \omega^3), \quad \cH=[\cD,\cD]=\mathsf{ker}(\omega^1,\omega^2)$ and
\begin{align}
 \mathsf{d}\omega^1=&-\omega^3\wedge\omega^4,\quad
 \mathsf{d}\omega^2=-\omega^3\wedge\omega^5,\quad
 \mathsf{d}\omega^3=-\omega^4\wedge\omega^5,\quad
 \mathsf{d}\omega^4=0,\quad
 \mathsf{d}\omega^5=0.
 \end{align}

Let $(X_i)$ be  the dual frame, in particular,
 \begin{align}
 X_4=\partial_q,\quad X_5=\partial_x+p\partial_y+q\partial_p+\tfrac12 q^2\partial_z,\quad X_3=\partial_p+q\partial_z.
 \end{align}
A general line field $\ell$ in the rank 3 distribution $\cH$ complementary to the rank 2 distribution $\cD$, is spanned by the vector field
 \begin{align}
 \hat{X}_3=X_3-A X_4- B X_5,
 \end{align}
where $A=A(x,y,p,q,z)$ and $B=B(x,y,p,q,z)$ are arbitrary smooth fuctions.

Next we introduce a new coframe $(\hat{\omega}^i)$, $i=1,2,3,4,5$, adapted to $(\cD,\ell)$, i.e. such that \begin{equation}\label{eq-H,l}\cD=\mathsf{ker}(\hat{\omega}^1,\hat{\omega}^2, \hat{\omega}^3)\quad\mbox{and}\quad\ell=\mathsf{ker}(\hat{\omega}^1,\hat{\omega}^2,\hat{\omega}^4,\hat{\omega}^5)=\left\langle \hat{X}_3\right\rangle,\end{equation} where $(\hat{X}_i)$ is the dual frame.
There are plenty of such coframes; the one that we define is
\begin{equation}
\begin{pmatrix} \hat{\omega}^1\\\hat{\omega}^2\\\hat{\omega}^3\\\hat{\omega}^4\\\hat{\omega}^5\end{pmatrix}=\begin{pmatrix}
1&0&0&0&0\\
0&1&0&0&0\\
0&0&1&0&0\\
0&0&Y&1&0\\
0&0&Z&0&1
\end{pmatrix} \begin{pmatrix} \omega^1\\\omega^2\\\omega^3\\\omega^4\\\omega^5\end{pmatrix}.
\end{equation}

We now apply Cartan's equivalence method. We consider the lifted coframe
\begin{equation}\label{gencoframe}
\begin{pmatrix} \psi^1\\\psi^2\\\psi^3\\\psi^4\\\psi^5\end{pmatrix}=\begin{pmatrix}
u_1&u_2&0&0&0\\
u_3&u_4&0&0&0\\
u_{10}&u_{11}&u_5&0&0\\
u_{12}&u_{13}&0&u_6&u_7\\
u_{14}&u_{15}&0&u_8&u_9
\end{pmatrix} \begin{pmatrix} \hat{\omega}^1\\\hat{\omega}^2\\\hat{\omega}^3\\\hat{\omega}^4\\\hat{\omega}^5\end{pmatrix}
\end{equation}
on $\cG=N\times H$, where $H\subset\mathrm{GL}(5,\bbR)$ is the subgroup preserving  \eqref{eq-H,l}.
  
Now we can make the normalizations
\begin{equation}\label{normal1}\begin{aligned}
\mathsf{d}\psi^1\wedge\psi^1\wedge\psi^2=-\psi^1\wedge\psi^2\wedge\psi^3\wedge\psi^4,\\
\mathsf{d}\psi^2\wedge\psi^1\wedge\psi^2=-\psi^1\wedge\psi^2\wedge\psi^3\wedge\psi^5
\end{aligned}
\end{equation}
by solving for 
\begin{equation}
u_6=\frac{u_1}{u_5}\quad\&\quad u_7=\frac{u_2}{u_5}\quad\&\quad u_8=\frac{u_3}{u_5}\quad\&\quad u_9=\frac{u_4}{u_5}.\label{fnor}
\end{equation}
Then from
 \begin{align}
 \mathsf{d}\psi^3\wedge\psi^1\wedge\psi^2\wedge\psi^3=-\frac{u_5^3}{\delta}\psi^1\wedge\psi^2\wedge\psi^3\wedge\psi^4\wedge\psi^5,
 \end{align}
where 
 \begin{equation}\label{delta}
\delta=(u_2u_3-u_1u_4),
 \end{equation}
we conclude that we may further normalize
 \begin{equation} \label{normal2}
 \mathsf{d}\psi^3\wedge\psi^1\wedge\psi^2\wedge\psi^3=-\psi^1\wedge\psi^2\wedge\psi^3\wedge\psi^4\wedge\psi^5.
 \end{equation}
By \eqref{normal1},  there exist linearly independent $1$-forms $\Theta^1, \Theta^2, \Theta^3, \Theta^4$ such that 
\begin{equation}\label{normal3}
\begin{aligned}
&\left(\mathsf{d}\psi^1-\psi^1\wedge(\Theta^1-\Theta^2)+\psi^2\wedge\Theta^4+\psi^3\wedge\psi^4\right)\wedge\psi^1=0 \\
&\left(\mathsf{d}\psi^1-\psi^1\wedge(\Theta^1-\Theta^2)+\psi^2\wedge\Theta^4+\psi^3\wedge\psi^4\right)\wedge\psi^2=0 \\
& \left(\mathsf{d}\psi^2-\psi^2\wedge(2\Theta^1+\Theta^2)+\psi^1\wedge\Theta^3+\psi^3\wedge\psi^5\right)\wedge\psi^1=0\\
& \left(\mathsf{d}\psi^2-\psi^2\wedge(2\Theta^1+\Theta^2)+\psi^1\wedge\Theta^3+\psi^3\wedge\psi^5\right)\wedge\psi^2=0\\
\end{aligned}
\end{equation}
The $1$-forms $\Theta^1, \Theta^2, \Theta^3, \Theta^4$ are not uniquely defined by these equations, but  there remains a freedom of adding arbitrary multiples of $\psi^1, \psi^2$ to each of them.
We can now impose the further normalizations
\begin{equation}
\begin{aligned}\label{normal4}
&\left(\mathsf{d}\psi^3-\psi^3\wedge\Theta^1\right)\wedge\psi^1\wedge\psi^2\wedge\psi^4=0\\
&\left(\mathsf{d}\psi^3-\psi^3\wedge\Theta^1\right)\wedge\psi^1\wedge\psi^2\wedge\psi^5=0\\
&\left(\mathsf{d}\psi^4+\psi^4\wedge\Theta^2+\psi^5\wedge\Theta^4\right)\wedge\psi^1\wedge\psi^2\wedge\psi^4=0 \\
&\left(\mathsf{d}\psi^4+\psi^4\wedge\Theta^2+\psi^5\wedge\Theta^4\right)\wedge\psi^1\wedge\psi^2\wedge\psi^5=0 \\
& \left(\mathsf{d}\psi^5 +\psi^4\wedge\Theta^3-\psi^5\wedge(\Theta^1+\Theta^2)\right)\wedge\psi^1\wedge\psi^2\wedge\psi^4=0\\
& \left(\mathsf{d}\psi^5 +\psi^4\wedge\Theta^3-\psi^5\wedge(\Theta^1+\Theta^2)\right)\wedge\psi^1\wedge\psi^2\wedge\psi^5=0\\
\end{aligned}
\end{equation}
solving for $u_{10},\dots,u_{15}$. 
The residual structure group preserving the coframe normalizations \eqref{normal1}, \eqref{normal2}, \eqref{normal3}, and \eqref{normal4} is the following subgroup $H^1\subset\mathrm{GL}(5,\bbR)$ isomorphic to $\mathrm{GL}(2,\bbR)$:
\begin{equation} \label{reducedgroup}
\begin{pmatrix}
u_1&u_2&0&0&0\\
u_3&u_4&0&0&0\\
0&0&\delta^{\tfrac{1}{3}}&0&0\\
0&0&0&u_1/\delta^{\tfrac{1}{3}}&u_2/\delta^{\tfrac{1}{3}}\\
0&0&0&u_3/\delta^{\tfrac{1}{3}}&u_4/\delta^{\tfrac{1}{3}}
\end{pmatrix} .
\end{equation}
In particular, this implies the following:
\begin{prop} There is an invariantly defined rank $2$-distribution $\mathcal{K}$ complementary to the rank $3$ distribution $\cH=\cD\oplus\ell$, and  the tangent bundle splits as $TN=\cD\oplus\ell\oplus \mathcal{K}$.
\end{prop}

In order to uniquely pin down the $1$-forms $\Theta^1, \Theta^2, \Theta^3, \Theta^4$ further normalizations need to be imposed; for example the forms are uniquely determined if one further requires that
\begin{equation}
\label{normal5}
\begin{aligned}
&\left(\mathsf{d}\psi^1-\psi^1\wedge(\Theta^1-\Theta^2)+\psi^2\wedge\Theta^4+\psi^3\wedge\psi^4\right)=0 \\
& \left(\mathsf{d}\psi^2-\psi^2\wedge(2\Theta^1+\Theta^2)+\psi^1\wedge\Theta^3+\psi^3\wedge\psi^5\right)=0\\
&\left(\mathsf{d}\psi^3-\psi^3\wedge\Theta^1+\psi^4\wedge\psi^5\right)\wedge\psi^1\wedge\psi^4\wedge\psi^5=0\\&
\left(\mathsf{d}\psi^3-\psi^3\wedge\Theta^1+\psi^4\wedge\psi^5\right)\wedge\psi^2\wedge\psi^4\wedge\psi^5=0\\
&\left(\mathsf{d}\psi^4+\psi^4\wedge\Theta^2+\psi^5\wedge\Theta^4\right)\wedge\psi^1\wedge\psi^3\wedge\psi^4=0,\\
&\left(\mathsf{d}\psi^5 +\psi^4\wedge\Theta^3-\psi^5\wedge(\Theta^1+\Theta^2)\right)\wedge\psi^2\wedge\psi^3\wedge\psi^5=0,\\
&\left(\mathsf{d}\psi^4+\psi^4\wedge\Theta^2+\psi^5\wedge\Theta^4\right)\wedge\psi^2\wedge\psi^3\wedge\psi^4\\&+\left(\mathsf{d}\psi^4+\psi^4\wedge\Theta^2+\psi^5\wedge\Theta^4\right)\wedge\psi^1\wedge\psi^3\wedge\psi^5\\&+
\left(\mathsf{d}\psi^5 +\psi^4\wedge\Theta^3-\psi^5\wedge(\Theta^1+\Theta^2)\right)\wedge\psi^1\wedge\psi^3\wedge\psi^4=0,\\
&\left(\mathsf{d}\psi^5 +\psi^4\wedge\Theta^3-\psi^5\wedge(\Theta^1+\Theta^2)\right)\wedge\psi^1\wedge\psi^3\wedge\psi^5
\\& +\left(\mathsf{d}\psi^4+\psi^4\wedge\Theta^2+\psi^5\wedge\Theta^4\right)\wedge\psi^2\wedge\psi^3\wedge\psi^5\\&+\left(\mathsf{d}\psi^5  +\psi^4\wedge\Theta^3-\psi^5\wedge(\Theta^1+\Theta^2)\right)\wedge\psi^2\wedge\psi^3\wedge\psi^4=0
\end{aligned}.
\end{equation}

\begin{remark}
The chosen normalization condition is $H^1$-invariant, and one can show that we can indeed associate a canonical Cartan connection to our non-integrable XXO-structure. But since this fact is not directly  relevant for the classification problem,  we will not elaborate on it. 
\end{remark}

\subsection{Invariants}
 We associated to our  XXO-structure a coframe $(\psi^1, \psi^2, \psi^3, \psi^4, \psi^5, \Theta^1,\Theta^2,\Theta^3,\Theta^4)$ on $\mathcal{G}\subset N\times H^1$, uniquely determined by  \eqref{normal1}, \eqref{normal2}, \eqref{normal3}, \eqref{normal4}, and \eqref{normal5}. Equivalently, this means that the coframe satisfies structure equations of the form
\begin{equation}\label{torsion_c}
\begin{aligned}
\mathrm{d}\psi^1=&\psi^1\wedge (\Theta^1-\Theta^2)-\psi^2\wedge\Theta^4-\psi^3\wedge\psi^4 \\
\mathrm{d}\psi^2=&-\psi^1\wedge\Theta^3+\psi^2\wedge (2 \Theta^1+\Theta^2)-\psi^3\wedge\psi^5 \\
\mathrm{d}\psi^3=&\psi^3\wedge\Theta^1-\psi^4\wedge\psi^5
+\tfrac{1}{\delta^{2/3}}\beta \psi^1\wedge\psi^2\\&
+\tfrac{1}{\delta^{4/3}}\left(u_3^2\gamma_C+u_4^2\gamma_A+2 u_3 u_4 \gamma_B\right)\psi^1\wedge\psi^4\\
&-\tfrac{1}{\delta^{4/3}}\left(u_2 u_4 \gamma_A +u_2 u_3 \gamma_B +u_1 u_3 \gamma_C + u_1 u_4 \gamma_B\right) \psi^1\wedge\psi^5
\\
&-\tfrac{1}{ \delta^{4/3}}\left( u_2 u_4 \gamma_A +u_2 u_3 \gamma_B+u_1u_3\gamma_C+u_1u_4\gamma_B \right)\psi^2\wedge\psi^4\\
&+\tfrac{1}{\delta^{4/3}}\left(u_1^2\gamma_C+u_2^2\gamma_A+2 u_1u_2\gamma_B\right)\psi^2\wedge\psi^5\\
\mathrm{d}\psi^4=&-\psi^4\wedge\Theta^2-\psi^5\wedge\Theta^4+\tfrac{1}{\delta^{4/3}}\left(u_2\delta_B+u_1\delta_A\right)\psi^1\wedge\psi^2\\
&-\tfrac{1}{\delta^{5/3}}\left( u_2u_4 \tau_{A}+u_1u_3  \tau_{C}+ u_2u_3( \rho -\tau_{B})+ u_1u_4(-\rho+5\tau_{B}) \right)\psi^1\wedge\psi^3\\
&+\tfrac{1}{\delta}\left(u_4 \alpha_A- u_3 \alpha_B  \right)\psi^1\wedge\psi^4-\tfrac{1}{\delta}\left(\tfrac{7}{11} u_2 \alpha_A-\tfrac{1}{11} u_1 \alpha_B  \right)\psi^1\wedge\psi^5\\
&+\tfrac{1}{\delta^{5/3}}\left(u_1^2\tau_{C}+4u_1u_2\tau_{B}+u_2^2\tau_{A}\right)\psi^2\wedge\psi^3-\tfrac{1}{\delta}\left(\tfrac{4}{11}u_2 \alpha_A-\tfrac{4}{11}u_1 \alpha_B\right)\psi^2\wedge\psi^4
\\
\mathrm{d}\psi^5=&-\psi^4\wedge\Theta^3+\psi^5\wedge(\Theta^1+\Theta^2)+\tfrac{1}{\delta^{4/3}}\left(u_4\delta_B+u_3\delta_A\right)\psi^1\wedge\psi^2\\
&-\tfrac{1}{\delta^{5/3}}\left(u_3^2\tau_{C}+4u_3u_4\tau_{B}+u_4^2\tau_{A}\right)\psi^1\wedge\psi^3+\tfrac{1}{\delta}\left(\tfrac{4}{11}u_4\alpha_A-\tfrac{4}{11}u_3 \alpha_B\right)\psi^1\wedge\psi^5\\
&+\tfrac{1}{\delta^{5/3}}\left( u_1u_3 \tau_{C}+u_2 u_4\tau_{A}+ u_1u_4( \rho -\tau_{B}) +u_2u_3(-\rho+5\tau_{B}) \right)\psi^2\wedge\psi^3\\
&+\tfrac{1}{\delta}\left(\tfrac{7}{11}u_4 \alpha_A-\tfrac{1}{11}u_3\alpha_B\right)\psi^2\wedge\psi^4+\tfrac{1}{\delta}\left(-u_2 \alpha_A+ u_1\alpha_B\right)\psi^2\wedge\psi^5
\end{aligned}
 \end{equation}
 Here  $\gamma_A, \gamma_B, \gamma_C, \alpha_A, \alpha_B,\beta, \delta_A, \delta_B,  \tau_{A}, \tau_{B}, \tau_{C},\rho$ do not depend on the group parameters $u_1,\dots,u_4$.
\begin{remark}
 Inspecting the transformation of the coefficients in the above structure equations under the residual structure group exhibits several (weighted) tensorial invariants of the structure. In particular, it is visible that both $(\gamma_A,\gamma_B,\gamma_C)$ and $(\tau_A,\tau_B,\tau_C)$ define \emph{quadric invariants} of the structure.
\end{remark}

%
Furthermore, the geometric meaning of the invariant condition $\tau_{A}=\tau_B=\tau_C= 0$ can be seen from the structure equations.
 \begin{prop} Consider the metric 
 \begin{align}
 \tilde{\sfg}=2\psi^2\psi^4-2\psi^1\psi^5 \quad\mbox{on}\quad \mathcal{G}\subset N\times H^1,
 \end{align}
 which is degenerate along the distribution defined as the kernel of $(\psi^3,\Theta^1,\Theta^2,\Theta^3,\Theta^4)$. Then $\tilde{\sfg}$ descends to a well-defined conformal class $[\sfg]$ of non-degenerate metrics of signature $(2,2)$ on the $4$-dimensional local leaf space  if and only if the torsion condition  $\tau_{A}=\tau_B=\tau_C= 0$ is satisfied.
 \end{prop}
 
 \begin{proof}
 Using Cartan's formula for the Lie derivative and the structure equations, it is straightforward to verify that $\tilde\sfg$ transforms conformally when Lie derived along any vector field $\xi\in \mathsf{ker}(\psi^3,\Theta^1,\Theta^2,\Theta^3,\Theta^4)$ if and only if  $\tau_A, \tau_B$, and $\tau_C$ vanish identically. 
 \end{proof}

 Henceforth, we will assume that this condition is satisfied. 
 \begin{remark}
 Note that $\tau_A$, $\tau_B$ and $\tau_C$ define a map $\ell\times TN/\mathcal{H}\to\mathcal{D}$; it can be identified with the harmonic invariant $\mathcal{T}$ from Table \ref{F:KH}. An alternative way of computing this invariant was discussed in Appendix  \ref{S:torsion}.
 \end{remark}


 \subsection{Classification problem} 
 We now sketch how we could proceed to obtain an alternative derivation of the classification of  conformal structures with multiply transitive XXO-structures and $\Lie(\mathrm{G}_2)$-symmetric twistor distribution.  
 
Consider a coframe $(\psi^1,\cdots,\psi^5,\Theta^1,\cdots,\Theta^4)$  satisfying \eqref{torsion_c} with $\tau_A=\tau_B=\tau_C=0$. It defines $\mathcal{D}$ and $\ell$ via \eqref{eq-H,l}. (Note that the torsion condition \eqref{torsion_c} does \emph{not} already imply that the distribution $\mathcal{D}$ is $\Lie(\mathrm{G}_2)$-symmetric; we will implement this condition later.) Let $\gamma\in S^2\mathcal{D}^*$ denote the symmetric tensorfield defined by $\gamma_A$, $\gamma_B$, and $\gamma_C$.
 
  \begin{enumerate}
  \item We first assume that $\gamma\equiv 0$. Then applying $\mathsf{d}$ to \eqref{torsion_c}, one shows that this assumption implies that $\alpha_A=\alpha_B=0$, $\delta_A=\delta_B=0$, $\rho=\beta$ and
 \begin{equation}
 \begin{aligned}
 \mathrm{d}\Theta^1 &= 0\\
 \mathrm{d}\Theta^2 &= \phi_B\psi^1\wedge\psi^4+\phi_C\psi^1\wedge\psi^5+\phi_C\psi^2\wedge\psi^4+\phi_E\psi^2\wedge\psi^5-\Theta^3\wedge\Theta^4\\
 \mathrm{d}\Theta^3 &= (-\beta+\phi_C)\psi^1\wedge\psi^4+\phi_E\psi^1\wedge\psi^5+\phi_E\psi^2\wedge\psi^4+\psi_D\psi^2\wedge\psi^5+\Theta^1\wedge\Theta^4+2\Theta^2\wedge\Theta^3\\
 \mathrm{d}\Theta^4 &= 
\phi_A\psi^1\wedge\psi^4-\phi_B\psi^1\wedge\psi^5-\phi_B\psi^2\wedge\psi^4+(\beta-\phi_C)\psi^2\wedge\psi^5-\Theta^1\wedge\Theta^3-2\Theta^2\wedge\Theta^4\\
 \end{aligned}
 \end{equation}
 Since we are looking for homogeneous structures, we may assume that $\phi_A,\phi_B,\phi_C,\phi_D$ and $\beta$ are constants. If we further impose the condition that $\mathcal{D}$ be flat (i.e., the Cartan quartic vanishes, see e.g.\ \cite{AN2014}), then 
   one can show that there are precisely three XXO-structures satisfying these conditions:
     these are the maximally symmetric model $\mathsf{M9}$  ($\beta=0$) and the two sub-maximally symmetric models $\mathsf{M8.1}$ and  $\mathsf{M8.2}$ ($\beta\neq 0$).
  \item Next we assume that $\gamma$ does not vanish identically, but $\Delta=(\gamma_B)^2-\gamma_A\gamma_C=0$. In this case we can use the $H$-action to normalize to $\gamma_A=\epsilon=\pm 1$, $\gamma_B=0$ and $\gamma_C=0$. The residual structure group preserving this normalization is $2$-dimensional and so  the maximal possible symmetry dimension in this branch is $7$. Assuming that the XXO-structure be homogeneous  with $7$-dimensional symmetry algebra and $\mathcal{D}$ be $\mathsf{Lie}(\mathrm{G}_2)$-symmetric, one ends up with two one-parameter families of structures in this branch: these are the $\mathsf{M7}_\sfa$ models (which satisfy $\beta= 0$).
    
   If we do not insist on maximal symmetry in this branch, then a further reduction leads to the $\mathsf{M6N}$ model (which satisfies $\beta\neq 0$). More work needs to be done if one wants to confirm  that the $\mathsf{M6N}$ model is the \emph{only} model with $6$-dimensional symmetry algebra in the branch where $\gamma\neq 0$ and $\Delta=0$. However, we showed that this is indeed  the case by a different method in Section 4.
  \item Finally, if $\Delta=(\gamma_B)^2-\gamma_A\gamma_C\neq 0$ there are three cases: we may use the group action to normalize either to $\gamma_A=\gamma_C=0$, $\gamma_B=1$ or to  $\gamma_A=\gamma_C=\epsilon=\pm 1$, $\gamma_B=0$. In each of these cases the residual structure group preserving the normalization has dimension $3$ and the the maximal possible symmetry dimension of a homogeneous model in the branch is $6$. Imposing further that $\mathcal{D}$ be flat leads to the models $\mathsf{M6S.1},$ $\mathsf{M6S.2},$ and $\mathsf{M6S.3}$ (in each of these cases $\beta\neq 0$). 
  \end{enumerate} 

 \bibliographystyle{amsplain}

 \end{document}

%% file: Dynkin.tex
\newcommand\tcirc[3]{
	\ifthenelse{\equal{#1}{w}}{\filldraw[fill=white,draw=black] (#2) circle (0.075);}{}%
	\ifthenelse{\equal{#1}{b}}{\filldraw[black] (#2) circle (0.075);}{}%
	\draw (#2) ++(0,0.35) node {$#3$};
	}

\newcommand\bond[1]{\draw (#1) -- +(1,0)}

\newcommand\vbond[1]{\draw (#1) -- +(0,-1)}

\newcommand\diagbond[2]{
	\ifthenelse{\equal{#1}{u}}{
		\draw[semithick] (#2) -- +(0.5,0.865);
	}{}
	\ifthenelse{\equal{#1}{d}}{
		\draw[semithick] (#2) -- +(0.5,-0.865);
	}{}
	}

\newcommand\dbond[2]{
	\draw (#2) ++(0.03,0.03) -- +(0.94,0);
	\draw (#2) ++(0.03,-0.03) -- +(0.94,0);
	\ifthenelse{\equal{#1}{r}}{
		\draw[semithick] (#2) ++(0.6,0) ++(-0.15,0.2) -- ++(0.15,-0.2) -- +(-0.15,-0.2);
	}{}
	\ifthenelse{\equal{#1}{l}}{
		\draw[semithick] (#2) ++(0.45,0) ++(0.15,0.2) -- ++(-0.15,-0.2) -- +(0.15,-0.2);
	}{}
	}

\newcommand\tbond[2]{
	\draw (#2)  -- +(1,0);
	\draw (#2) ++(0.05,0.06) -- +(0.9,0);
	\draw (#2) ++(0.05,-0.06) -- +(0.9,0);
	\ifthenelse{\equal{#1}{r}}{
		\draw[semithick] (#2) ++(0.6,0) ++(-0.15,0.2) -- ++(0.15,-0.2) -- +(-0.15,-0.2);
	}{}
	\ifthenelse{\equal{#1}{l}}{
		\draw[semithick] (#2) ++(0.45,0) ++(0.15,0.2) -- ++(-0.15,-0.2) -- +(0.15,-0.2);
	}{}
	}

\newcommand\tcross[2]{
	\draw (#1) ++(0,0.35) node {$#2$};
	\draw[semithick] (#1) ++(-0.15,-0.15)-- +(0.3,0.3);
	\draw[semithick] (#1) ++(-0.15,0.15)-- +(0.3,-0.3);
	}

\newcommand\tsquare[2]{
		\draw[semithick,color=blue] (#1) ++(-0.15,-0.15) rectangle ++(0.3,0.3);
		\tcross{#1}{#2};
		}

\newcommand\tstar[2]{
	\draw[color=red] (#1) node {\Large$*$};
	\draw (#1) ++(0,0.35) node {$#2$};
	}

\newcommand\DDnode[3]{
\ifthenelse{\equal{#1}{w}}{\tcirc{w}{#2}{#3}}{}		
\ifthenelse{\equal{#1}{b}}{\tcirc{b}{#2}{#3}}{}		
\ifthenelse{\equal{#1}{x}}{\tcross{#2}{#3}}{}		
\ifthenelse{\equal{#1}{s}}{\tstar{#2}{#3}}{}		
\ifthenelse{\equal{#1}{q}}{\tsquare{#2}{#3}}{}		
}

 \newcommand\Athree[2]{
 \begin{tikzpicture}[scale=1,baseline=-3pt]

 \bond{0,0};		
 \bond{1,0};		

 \StrBefore{#2}{,}[\labelone]
 \StrBetween[1,2]{#2}{,}{,}[\labeltwo]
 \StrBehind[2]{#2}{,}[\labelthree]

 \StrChar{#1}{1}[\nodetype];
 \DDnode{\nodetype}{0,0}{\labelone};
 \StrChar{#1}{2}[\nodetype];
 \DDnode{\nodetype}{1,0}{\labeltwo};
 \StrChar{#1}{3}[\nodetype];
 \DDnode{\nodetype}{2,0}{\labelthree};
 \useasboundingbox (-.4,-.2) rectangle (2.4,0.55); 
 \end{tikzpicture}
 }

\newcommand\Gdd[2]{
 \begin{tikzpicture}[scale=1,baseline=-3pt]
 \tbond{l}{0,0};	

 \StrBefore{#2}{,}[\labelone]
 \StrBehind{#2}{,}[\labeltwo]
 \StrChar{#1}{1}[\nodetype];

 \DDnode{\nodetype}{0,0}{\labelone};
 \StrChar{#1}{2}[\nodetype];
 \DDnode{\nodetype}{1,0}{\labeltwo};
 \useasboundingbox (-.4,-.2) rectangle (1.4,0.55); 
 \end{tikzpicture}
 }
 
\newcommand\Edd[2]{
 \begin{tikzpicture}[scale=1,baseline=-3pt]
 \foreach \x in {0,1,2,3} {
	\bond{\x,0};
 }
 \vbond{2,0};
 
 \StrLen{#1}[\Ernk]
 
 \StrChar{#1}{1}[\nodetype];
 \DDnode{\nodetype}{0,0}{\StrBefore{#2}{,}};
 \StrChar{#1}{2}[\nodetype];
 \DDnode{\nodetype}{2,-1}{\StrBetween[1,2]{#2}{,}{,}};
 \StrChar{#1}{3}[\nodetype];
 \DDnode{\nodetype}{1,0}{\StrBetween[2,3]{#2}{,}{,}};
 \StrChar{#1}{4}[\nodetype];
 \DDnode{\nodetype}{2,0}{\StrBetween[3,4]{#2}{,}{,}};
 \StrChar{#1}{5}[\nodetype];
 \DDnode{\nodetype}{3,0}{\StrBetween[4,5]{#2}{,}{,}};
 \StrChar{#1}{6}[\nodetype];

 \ifthenelse{\equal{\Ernk}{6}}{
 		\DDnode{\nodetype}{4,0}{\StrBehind[5]{#2}{,}};
 		\useasboundingbox (-.4,-1.2) rectangle (4.4,0.55);
	}{}%
 
 \ifthenelse{\equal{\Ernk}{7}}{
 		\bond{4,0};
 		\DDnode{\nodetype}{4,0}{\StrBetween[5,6]{#2}{,}{,}};
		\StrChar{#1}{7}[\nodetype];
		\DDnode{\nodetype}{5,0}{\StrBehind[6]{#2}{,}};
 		\useasboundingbox (-.4,-1.2) rectangle (5.4,0.55);
	}{}%

 \ifthenelse{\equal{\Ernk}{8}}{
 		\bond{4,0};
 		\bond{5,0};
 		\DDnode{\nodetype}{4,0}{\StrBetween[5,6]{#2}{,}{,}};
		\StrChar{#1}{7}[\nodetype];
		\DDnode{\nodetype}{5,0}{\StrBetween[6,7]{#2}{,}{,}};
		\StrChar{#1}{8}[\nodetype];
		\DDnode{\nodetype}{6,0}{\StrBehind[7]{#2}{,}};
		\useasboundingbox (-.4,-1.2) rectangle (6.4,0.55);
	}{}%

 \end{tikzpicture}
 }